\documentclass[a4paper,10pt]{amsart}

\usepackage{psfrag}
\usepackage{epstopdf}	
\usepackage{amsmath, amssymb, amscd, amsfonts, mathrsfs, epsfig, epic, verbatim}
\usepackage{color}
\usepackage[mathcal]{euscript}
\usepackage{epsfig}
\usepackage{times}
\usepackage{hyperref}
\usepackage{MnSymbol}
\usepackage{enumerate}
\usepackage{array}
\usepackage{multirow}
\usepackage{hhline}
\usepackage[all]{xy}

 \newcommand{\Z}{{\mathbb Z}}               
 \newcommand{\R}{{\mathbb R}}               
 \newcommand{\C}{{\mathbb C}}               
 \newcommand{\Lef}{{\mathbf L}}             
 \newcommand{\Pt}{\mathbf{1}}               
 \newcommand{\CP}[1]{\mathbb{C}P^{#1}}      
 \newcommand{\E}{{\mathcal E}}              
 \newcommand{\Et}[1]{\widetilde{E}_{#1}}    
 \newcommand{\F}{{\mathcal F}}              
 \newcommand{\Qt}{{\mathcal Q}}             
 \newcommand{\Ft}{{\mathcal S}}             
 \renewcommand{\L}{{\mathcal L}}            
 \renewcommand{\t}{\theta}                  
 \renewcommand{\O}{{\mathcal O}}            
 \newcommand{\Mod}{{\mathcal M}}            
 \renewcommand{\d}{\mbox{d}}                
 \newcommand{\gl}{\mathfrak{gl}}            
 \newcommand{\End}{{\mathcal E}nd}
 \newcommand{\Tor}{{\mathcal T}or}

 \newcommand{\CPbar}{{\overline {{\mathbb {C}}P}}^2}
 \newcommand{\Hecke}[2]{{\mathcal H}ecke_{#1}^{#2}}
 \newcommand{\Mot}{K_0(\Var_{\C})}
 \newcommand{\pardeg}{\deg_{\vec{\alpha}}}
 \newcommand{\parslope}{\mu_{\vec{\alpha}}}

 \DeclareMathOperator{\coker}{coker}

 \DeclareMathOperator{\rank}{rk}
 
 \DeclareMathOperator{\Spec}{Spec}
 \DeclareMathOperator{\Var}{Var}
 \DeclareMathOperator{\Torsion}{Tor}

 \DeclareMathOperator{\tr}{tr}
 
 \DeclareMathOperator{\Res}{Res}

\newtheorem{prop}{Proposition}[section]
\newtheorem{lem}[prop]{Lemma}

\newtheorem{defn}[prop]{Definition}

\newtheorem{thm}[prop]{Theorem}
\newtheorem{mainthm}[prop]{Main Theorem}
\newtheorem{assn}[prop]{Assumption}
\newtheorem{remark}[prop]{Remark}

\newtheorem{notation}[prop]{Notation}

\newtheorem{example}[prop]{Example}

\newcolumntype{C}[1]{>{\centering\arraybackslash\hspace{0pt}}m{#1}}
\newcommand{\defin}[1]{{\bf\emph{#1}}}

\title[Irregular Hitchin systems and wall-crossing] {Hitchin
  fibrations on moduli of irregular Higgs bundles and motivic
  wall-crossing} 
\author{P\'eter Ivanics}  
\address{Budapest University of Technology and Economics, 1111. Budapest,
Egry J\'ozsef utca 1. H \'ep\"ulet, Hungary}
\email{ipe@math.bme.hu}

\author{Andr\'as Stipsicz} 
\address{R\'enyi Institute of Mathematics, 1053. Budapest, Re\'altanoda
utca 13-15. Hungary}
\email{stipsicz@renyi.hu}

\author{Szil\'ard Szab\'o}
\thanks{Corresponding author: Szil\'ard Szab\'o}
\address{Budapest University of Technology and Economics, 1111. Budapest,
Egry J\'ozsef utca 1. H \'ep\"ulet, Hungary}
\email{szabosz@math.bme.hu}

\begin{document}

\begin{abstract}
In this paper we give a complete description of the Hitchin fibration
on all $2$-dimensional moduli spaces of rank-$2$ irregular Higgs
bundles with two poles on the projective line.  We describe the
dependence of the singular fibers of the fibration on the eigenvalues
of the Higgs fields, and describe the corresponding motivic
  wall-crossing phenomenon in the parameter space of parabolic
  weights.
\end{abstract}
\maketitle

\section{Introduction}
\label{sec:intro}

Moduli spaces of Higgs bundles with irregular singularities
on K\"ahler manifolds have been extensively investigated over the last
few decades from a variety of perspectives. One salient feauture of
these spaces is the existence of a proper map to an affine space 
called the Hitchin fibration \cite{Hit}.
In mirror symmetric considerations, the singular fibers of this fibration
play a major role, see relevant remarks in Subsection~\ref{ssec:Mirror}.

In this paper we study certain rank-$2$ irregular Higgs bundles $(\E,
\t)$ defined over $\CP1$, where $\E$ is a rank-$2$ vector bundle and
$\t$ is a meromorphic section of $\End(\E) \otimes K$ called the Higgs
field.  We set $\deg (\E ) = d.$ We will limit ourselves to the case
where $\t$ has two poles $q_1$ and $q_2$, and the sum of the order of
the poles is $4$.  The order of the poles are both $2$ in the first
subcase, and are $3$ and $1$ in the second subcase, hence (in the
respective cases) $\t$ is a holomorphic homomorphism
\begin{equation}\label{eq:divisor}
	\theta : \E \to \E \otimes K(D)
\end{equation}
where $D$ is either $2\cdot \{ q_1 \} + 2 \cdot \{q_2\}$ or
$3\cdot \{ q_1 \} + \{q_2\}$.  
Up to an isomorphism of $\CP1$ we may fix the points $q_1 = [0:1]$ and
$q_2=[1:0]$. 

We pick parameters $\alpha_{i}^{j} \in [0,1)$ for $j \in \{ 1,2 \}$
  and $i \in \{ +,- \}$ and for simplicity we write $ \vec{\alpha} =
  (\alpha_+^1 , \alpha_-^1 , \alpha_+^2 , \alpha_-^2).  $ The moduli
  spaces of interest to us (parameterizing
  $\vec{\alpha}$-(semi-)stable irregular Higgs bundles of rank $2$
  over $\CP1$ with two poles and fixed polar part) will be denoted
  $\Mod^{(s)s}(\vec{\alpha})$. By definition
  $\Mod^{s}(\vec{\alpha})\subset \Mod^{ss}(\vec{\alpha})$, and
  $\Mod^{ss}(\vec{\alpha})\setminus \Mod^{s}(\vec{\alpha})$ is finite.
  These spaces are equipped with a morphism
\begin{equation}\label{eq:Hitchin}
   h: \Mod^{(s)s}(\vec{\alpha}) \to B, 
\end{equation}
called the \defin{irregular Hitchin map}, where $B$ is an affine line
over ${\mathbb {C}}$.  This map is a straightforward generalization of
the Hitchin map on moduli spaces of holomorphic Higgs bundles
\cite{Hit}. 
For details on the irregular Hitchin map, see Subsection
\ref{subsec:Hitchin}.  The moduli spaces depend on some further
parameters that we will tag $(S)$, $(N)$, $(s)$ and $(n)$.  
In this paper we give a complete description of the singular fibers of
$h$ depending on the parameters, under the following condition. 
\begin{assn}\label{assn:elliptic}
 At least one (equivalently, the generic) fiber of $h$ is a smooth elliptic curve. 
\end{assn}

The irregular Hitchin map $h$ extends to a map ${\overline {h}}\colon
\Mod^{ss}\cup E_{\infty}\to \CP{1}$, where $E_{\infty}$ is a complex
curve, which is the fiber at infinity of the fibration ${\overline
  {h}}$. It turns out that ${\overline {h}}$ is generically an
elliptic fibration on the complex surface $\Mod^{ss}\cup E_{\infty}$,
which, on the other hand, is a (Zariski) open subset of a rational
elliptic surface (which surface is diffeomorphic to $\CP{2}\#
9{\overline {\CP{}}}^2$).  The singular fibers of an elliptic
fibration have been classified by Kodaira \cite{Kodaira}, and the
relevant singular fibers will be recalled in
Section~\ref{sec:ell_penc}.  The \emph{combinatorial type} of an
elliptic fibration is the list of the singular fibers (with
multiplicity) arising in the particular fibration. In \cite{Miranda,
  Persson, SSS} the complete list of all possible combinatorial types
of elliptic fibrations on a rational elliptic surface have been
determined.

This classification turns out to be useful in understanding irregular
Hitchin fibrations.  In the following we will examine various cases of
Equation~\eqref{eq:Hitchin}: we deal with the two cases when
$D=2q_1+2q_2$ or $D=3q_1+q_2$ in Equation~\eqref{eq:divisor}, and in
each case we have to separate further subcases depending on behaviour
of the Higgs field at the poles --- indeed, we will distinguish
regular semi-simple (denoted by $(S)$ for the pole at $q_1$ and by
$(s)$ at $q_2$) and nilpotent cases (denoted by $(N)$ and $(n)$,
respectively).  Indeed, for $D=2q_1+2q_2$ we will have three subcases
to distinguish (listed in
Subsection~\ref{subsec:result_painleve_III}), while for $D=3q_1+q_2$
there are four distinct cases (listed in
Subsection~\ref{subsec:result_painleve_IV}).  In each case the map $h$
depends on the complex parameters defining the moduli space; these
cases are denoted $(S),(N), (s)$ and $(n)$ in
Section~\ref{sec:PreciseMain}.

Before giving the precise (and somewhat tedious) forms of our results,
here we just state the main principle.  For the exact formulae and the
possible combinatorial types of the seven cases see the expanded
versions of the Main Theorem in Section~\ref{sec:PreciseMain}.

\begin{mainthm}\label{thm:MainThm}
In each case there is a precise formula in terms of the complex
parameters of the polar part of the Higgs field which determines the
class in the Grothendieck ring of varieties of all individual
singular fibers of the Hitchin fibration.
\end{mainthm}

Notice that some of the Hitchin fibers we find do not belong to
Kodaira's list because the Hitchin fibers may be noncompact.  We
explain this phenomenon in detail at the beginning of
Section~\ref{sec:PreciseMain}. (This phenomenon has been also
discussed in the Painlev\'e VI case \cite[Proposition~2.9]{ISS_I0*}.)

Theorem~\ref{thm:MainThm} is different from the authors' results
\cite{ISS} in the case of a single irregular singularity; in fact the
latter case arises as a degeneration of the setup of the present
paper.

The statements and arguments in the paper seem to be rather
repetitive.  Although the driving ideas in the cases are very similar,
the actual shapes of the computations (and henceforth the statements
themselves) are quite different. Indeed, a wall crossing phenomenon
(to be detailed in the next Subsection) arises in three of the seven
cases.  The key idea in each proof is that we determine the number of
roots of certain polynomials.  In Theorems~\ref{thm:PIII(D6)} and
\ref{thm:PIV} use a conversion to symmetric polynomials and other
special polynomials to find the singular fibers of the fibration,
while the proof of Theorem~\ref{thm:PIII(D8)}, on the other hand, is
rather simple.
The proofs of Theorems~\ref{thm:PIV_degenerate} and
\ref{thm:PII_degenerate} involve the essential use of the blow-up
procedure.
For the sake of completeness we therefore decided to give
full arguments rather than just sketching the proofs in the individual
subcases.

\subsection{Wall-crossing}
\label{subsec:wallcross}

In this Subsection we highlight a feature of the technical results of
the paper.  Namely, in the detailed versions of
Theorem~\ref{thm:MainThm} we determine the diffeomorphism class (or,
in some cases the class in the Grothendieck ring $K_0(\Var_{\C} )$,
cf.  Subsection~\ref{subsec:Grothendieck}) of the fibers of the
irregular Hitchin map.  Now, in the assertions where only the class of
the fiber in $K_0(\Var_{\C} )$ is specified, we could be more precise
by attaching integer indices $(\delta_+, \delta_-)$ to the components
of the given fiber corresponding to the bidegree imposed on the
torsion-free sheaves parametrized by the given class.  For the notion
of bidegree we refer to~\eqref{eq:bidegree}.  In
Subsections~\ref{subsec:generic_weights},~\ref{subsec:special_weights},~\ref{subsec:generic_weights_I2}
and~\ref{subsec:special_weights_I2} we do include the bidegree in the
notation as an index of the components of the Hitchin fiber.  The
notions of generic and special parabolic weight will be provided in
Definition~\ref{defn:generic_parabolic_weights}. It turns out that all
possible parabolic weights constitute a real vector space of dimension
one and special weigths form $\Z \subset \R$.  We call this copy of
$\Z\subset \R$ the set of \defin{walls}.  In Sections~\ref{sec:III}
and~\ref{sec:I2} we will prove the following simple wall-crossing
result concerning
cases~\eqref{thm:PIII(D6)_III},~\eqref{thm:PIII(D6)_I2_I2}
and~\eqref{thm:PIII(D6)_I2_I1} of Theorem~\ref{thm:PIII(D6)},
cases~\eqref{thm:PIV_III},~\eqref{thm:PIV_I2_II}
and~\eqref{thm:PIV_I2_I1} of Theorem~\ref{thm:PIV}, and
cases~\eqref{thm:PIV_degenerate_III} and~\eqref{thm:PIV_degenerate_I2}
of Theorem~\ref{thm:PIV_degenerate}.
  \begin{thm}\label{thm:wall-crossing} 
  The class of the singular fiber in $K_0(\Var_{\C} )$ is the same 
  on both sides of a wall, but the index of the given classes changes 
  from $(\delta_+, \delta_-)$ to $(\delta_+ \pm 1, \delta_- \mp1)$.
  \end{thm}
From a physical point of view, this wall-crossing phenomenon was studied in 
\cite[Subsection 9.4.6]{GMN}.

\subsection{Mirror symmetry and Langlands duality}
\label{ssec:Mirror}

Let us now comment on one consequence of our results. Namely,
Lemma~\ref{lem:adjoint_orbit} of Section~\ref{sec:I2_I3_III_IV} may be
reformulated as saying that in the degenerate cases one of the Hitchin
fibers is a certain Jacobian within a compactified Jacobian of a
singular curve, and the compactifying point corresponds to a Higgs
bundle with the required eigenvalues but vanishing nilpotent part.
Hence, as may be expected, the completion of these moduli spaces
arises by allowing for more special residue conditions with the same
characteristic polynomial.  This phenomenon may be interpreted as an
Uhlenbeck-type compactification result for moduli spaces of irregular
Higgs bundles.  Indeed, as we will see in Lemma
\ref{lem:sphere_in_fiber}, in the degenerate cases one singular
spectral curve has one component which is an exceptional divisor in
the blow-up of the Hirzebruch surface.  Now, under suitable degree
conditions the direct image with respect to the ruling morphism $p$ of
sheaves on such a special curve gives rise to a non-locally free sheaf
on the base curve, i.e. a sheaf with one fiber of dimension higher
than $2$ --- an analogue of infinitely concentrated (or Dirac)
instantons from gauge theory in the context of irregular Higgs
bundles.

It is known \cite{Donagi_Pantev, Hausel_Thaddeus} that Hitchin moduli 
spaces $\Mod_{G}(C)$ and  $\Mod_{^LG}(C)$ on a given curve $C$ corresponding 
to Langlands dual groups $G, ^LG$ are mirror partners in the sense of 
Strominger--Yau--Zaslow \cite{SYZ}. 
Moreover, it is expected from mirror symmetry considerations of Gukov 
\cite{Hausel,Hit_mirror} that BAA-branes (flat bundles over Lagrangian 
subvarieties) on $\Mod_{G}(C)$ should be mirror to BBB-branes 
(hyperholomorphic sheaves) on $\Mod_{^LG}(C)$. 
In \cite[Section 7]{Hit_mirror}, Hitchin proposes a candidate for such a mirror dual 
pair in the case $G = \mbox{Gl}(2m, \C)$. 
The role of the BAA-brane in this setup is played by the trivial bundle 
over the character variety for the real form $G^r = U(m,m)$, 
with corresponding BBB-brane a certain holomorphic vector bundle 
over the moduli space of $\mbox{Sp}(2m, \C)$ Higgs bundles. 
Essentially, Hitchin proves that away from the discriminant locus 
(i.e. for Higgs bundles with smooth spectral curve) the kernel of an even exterior power of 
a certain Dirac-operator and the moduli space of $U(m,m)$-Higgs bundles with given characteristic 
classes are in a Fourier--Mukai type of duality. 
The interpretation of this relationship as a Fourier--Mukai transform breaks down over the 
singular fibers (which are no longer tori). We hope that our results, providing a complete understanding 
of the singular fibers of the irregular version of the Hitchin map,
will be of use in order to verify a similar phenomenon for the spaces
we consider.

As for an application of our results in another direction, in \cite{Sz_PW} the third named author computes the perverse filtration 
on the cohomology of the $2$-dimensional moduli spaces of irregular parabolic Higgs bundles on the projective line, 
and compares them to the mixed Hodge structure on the corresponding wild character variety.

The paper is composed as follows.  In Section~\ref{sec:PreciseMain} we
provide the full statements of Main Theorem~\ref{thm:MainThm} in the
various cases. In Section \ref{sec:prep_mat} we collect some (mostly
standard) material about irregular Higgs bundles, their moduli spaces,
the analog of Hitchin's fibration in this setup and the Grothendieck
ring of varieties.  In Section \ref{sec:ell_penc} we discuss some
general properties of elliptic fibrations on rational elliptic
surfaces.  In Sections \ref{sec:22} and \ref{sec:31} we carry out a
complete analysis of the singular fibers of the elliptic fibrations
obtained from an elliptic pencil on a Hirzebruch surface, in terms of
the parameters specifying their base locus.  Finally, in Sections
\ref{sec:I1_II}, \ref{sec:III}, \ref{sec:I2} and
\ref{sec:I2_I3_III_IV} we determine the families of torsion-free
sheaves supported on the singular curves giving rise to
$\vec{\alpha}$-(semi-)stable irregular Higgs bundles.

\bigskip

\noindent {\bf Acknowledgements:} {A.\;S. was supported by ERC
  Advanced Grant LDTBud and by NKFIH 112735.}  {A.\;S. and
  Sz.\;Sz. were supported by the \emph{Lend\"ulet} grant Low
  Dimensional Topology of the Hungarian Academy of
  Sciences. P.\;I. and Sz.\;Sz. were partially supported by NKFIH
  120697. The authors are grateful to an anonymous referee for many
  insightful comments and suggestions.}


\section{The precise versions of the Main Theorem}
\label{sec:PreciseMain}

Now we turn to stating the precise versions of our resuts, which have
been summarized in Main Theorem~\ref{thm:MainThm}.  For the sake
  of simplicity, in this section and in Section~\ref{sec:I2_I3_III_IV}
  we loosen standard terminology as follows: by elliptic fibration we
  mean a morphism $X \to C$ from a (possibly non-compact) surface $X$
  to a compact curve $C$ if the generic fiber is a compact smooth
  elliptic curve.  Irregular Hitchin fibrations on the moduli spaces
  of Higgs bundles under consideration are biholomorphic to the
  complement of one singular fiber in an elliptic fibration in this
  more general sense. The fundamental reason that the fibration is
  elliptic only in this broader sense is that in case the orbit of the
  residue of the Higgs field at the logarithmic point is
  non-semisimple, a sequence of Higgs bundles with given
  characteristic polynomial and with residue in the non-semisimple
  orbit may converge to a Higgs bundle with the same characteristic
  polynomial and residue in the closure of the given orbit (rather
  than the orbit itself).  The geometric manifestation of this
  phenomenon is the existence of some irreducible components of fibers
  in elliptic fibrations mapping to a point under the ruling of the
  Hirzebruch surface.  In these (so-called degenerate) cases, one
  needs a finer analysis of the possible spectral sheaves giving rise
  to parabolically stable Higgs bundles. This analysis will be carried out in
  Section~\ref{sec:I2_I3_III_IV}, and will show the existence of  non-compact
  fibers (even though the spectral curves themselves are compact).  It
  follows that the elliptic fibrations that we discuss throughout the
  paper are honest elliptic fibrations in the usual sense (as they are
  obtained from pencils of spectral curves), except for this section
  and Section~\ref{sec:I2_I3_III_IV}. We chose to keep the usual
  terminology for surfaces with non-compact fibers because
  non-compactness only appears at the last step, where a simple
  comparison of classes in Grothendieck ring makes it obvious which
  fibers are not compact.

\subsection{Statement of results in the Painlev\'e III cases}
\label{subsec:result_painleve_III}

Consider first the case when the order of the poles of $\t$ is $2$ at
both points.  We distinguish three subcases, according to whether the
polar part of the Higgs field is semisimple (referred to as $(S)$ or
$(s)$) or has nonvanishing nilpotent part (referred to as $(N)$ or
$(n)$) near $q_1, q_2$.  Via nonabelian Hodge theory, the
corresponding meromorphic connections of these subcases give rise to
\begin{enumerate}
 \item $PIII(D6)$ when both polar parts are semisimple;
 \item $PIII(D7)$ when exactly one polar part is semisimple and the other one has nonvanishing nilpotent part;
 \item $PIII(D8)$ when both polar parts have nonvanishing nilpotent part. 
\end{enumerate}
To define the moduli spaces of irregular Higgs bundles in
  these cases, we need to fix some parameters.  Namely, depending on
  the cases, we need to fix sets of parameters of the form $(Ss),
  (Sn), (Ns), (Nn)$ where the letters $S, N, s, n$ refer to the
  following sets of natural complex parameters 
\begin{align*}
  (S) &\quad a_+, a_-, \lambda_+, \lambda_- \\
  (N) &\quad a_{-4}, a_{-3}, a_{-2} \\
  (s) &\quad b_+, b_-, \mu_+, \mu_- \\
  (n) &\quad b_{-4}, b_{-3}, b_{-2}.
\end{align*}
\noindent The parameters appearing in the above lists have geometric meaning:
basically they encode the base locus of an elliptic pencil on the
  Hirzebruch surface ${\mathbb {F}}_2$, see \eqref{eq:22_ss_z1},
\eqref{eq:22_nil_z1}, \eqref{eq:22_ss_z2}, \eqref{eq:22_nil_z2}.  
For example, $a_+, a_-$ (and similarly $b_+, b_-$) determine the locations of
the base points on the fibers of the Hirzebruch surface,
see also Figures~\ref{fig:blowup22} and \ref{fig:blowup22bis}.
In order to state our results, it will be useful to consider some
polynomial expressions $M, L, A$ and $B$ of these parameters:
\begin{equation}\label{eq:22_notation}
	A=a_--a_+,\quad  B=b_--b_+, \quad L=\lambda_--\lambda_+, \quad 
M=\mu_--\mu_+.
\end{equation}
In the case $(S)$ (or $(s)$), if the condition $A\neq 0$ (respectively, $B\neq 0$) holds, 
we call the semisimple polar part \defin{regular}. Geometrically, this amounts to requiring 
that there are two distinct base points on the corresponding fiber.

Moreover, we refer to Definition \ref{defn:generic_parabolic_weights}
for the notion of generic and singular parabolic weights. In the
  statements of our results we need one more (abstract) concept: By
the class of a variety we mean its class in the Grothendieck ring
$K_0(\Var _{\mathbb {C}})$ of varieties, see Subsection
\ref{subsec:Grothendieck}.  In particular, $\Lef$ stands for the class
of the affine line and $\Pt$ for that of a point; as a consequence
$\Lef + \Pt$ and $\Lef - \Pt$ denote the classes of $\CP1$ and
$\C^{\times}$, respectively. Finally, for notations and (standard)
  conventions regarding singular elliptic fibers (from Kodaira's list
  \cite{Kodaira}), see Section~\ref{sec:ell_penc}.

In the next theorem $\Delta=-256 A^3 B^3+192 A^2 B^2 L M-3 A B \left(9
L^4-2 L^2 M^2+9 M^4\right)+4 L^3 M^3$ (a certain 
discriminant naturally associated to a degree-4 polynomial specified by
the problem).
\begin{thm}[PIII(D6)]\label{thm:PIII(D6)}
Assume that the polar
part of the Higgs field is of order $2$ and regular semisimple both near $q_1$
and near $q_2$, that is, we are in case $(Ss)$.  
Then {the irregular Hitchin fibration $h$ on}
$\Mod^{ss} (\vec{\alpha})$ is biregular to the complement of the fiber
at infinity which is of type $I_2^*$ (equivalently $\tilde{D}_6$) in
an elliptic fibration of the rational elliptic surface such that the
set of the other singular fibers is:
  \begin{enumerate}
		\item if $\Delta=0$ and $L^2=M^2\neq0$ then an $I_1$ curve and \label{thm:PIII(D6)_III}
		\begin{enumerate}
		 \item for generic weights a further type $III$ curve, 
		 \item for special weights a fiber in the class $\Lef + \Pt$, with the class of the corresponding fiber of $\Mod^{s} (\vec{\alpha})$ given by $\Lef$; 
		\end{enumerate}
		\item 
if $\Delta=0$, $L^2=-M^2\neq0$ and $M^3=8ABL$, then two type $II$
fibers; \label{thm:PIII(D6)_II_II}
		\item 
if  $\Delta=0$, $L^2=-M^2\neq0$ and $M^3 \neq 8ABL$, then a type $II$ and two $I_1$ fibers;\label{thm:PIII(D6)_II_I1}
\item
if $\Delta=0$ and $L^2\neq \pm M^2$, then a type $II$ and two $I_1$ fibers again;\label{thm:PIII(D6)_II_I1valtozat}
\item if $\Delta \neq 0$ and $L=M=0$ then \label{thm:PIII(D6)_I2_I2}
		  \begin{enumerate}
		   \item for generic weights two type $I_2$ fibers,
		   \item for special weights two fibers in the class $\Lef$, with the classes of the corresponding fibers of $\Mod^{s} (\vec{\alpha})$ given by $\Lef - \Pt$; 
		  \end{enumerate}
		\item if $\Delta \neq 0$ and $L^2=M^2\neq0$ then two $I_1$ fibers and \label{thm:PIII(D6)_I2_I1}
		\begin{enumerate}
		 \item for generic weights a further type $I_2$ fiber, 
		 \item for special weights a fiber in the class $\Lef$, with the class of the corresponding fiber of $\Mod^{s} (\vec{\alpha})$ given by $\Lef - \Pt$;
		\end{enumerate}
		\item 
if $\Delta \neq 0$ and $L^2 \neq M^2$, then four type $I_1$
fibers.\label{thm:PIII(D6)_I1}
  \end{enumerate}
	We note that $\Delta=0$ and $L=M=0$ implies $A=0$ or $B=0$, hence this
case is {not in $(Ss)$,} therefore the above items cover all cases.
\end{thm}

In the next theorem we use the discriminant $\Delta = 4 A^3
b_{-3} \left(2 L^3-27 A b_{-3}\right)$ in the case $(Sn)$ and $\Delta
= 4 B^3 a_{-3} \left(2 M^3-27 B a_{-3}\right)$ in the case $(Ns)$.
 \begin{thm}[PIII(D7)]\label{thm:PIII(D7)}
   Assume that the polar part of the Higgs field is of order $2$ and regular 
   semisimple near $q_1$ and of order $2$ and non-semisimple near
   $q_2$ (or vice versa), i.e we are in case $(Sn)$ (or in $(Ns)$, respectively).
   Then {the irregular Hitchin fibration $h$ on} $\Mod^{ss}
   (\vec{\alpha})$ is biregular to the complement of the fiber at
   infinity which is of type $I_3^*$ (equivalently $\tilde{D}_7$) in
   an elliptic fibration of the rational elliptic surface such that
   the set of other singular fibers of the fibration is:
  \begin{enumerate}
		\item if $\Delta = 0$, then a type $II$ and an $I_1$ fibers; 
		\item if $\Delta \neq 0$, then three type $I_1$ fibers. 
  \end{enumerate}
 \end{thm}

 \begin{thm}[PIII(D8)]\label{thm:PIII(D8)}
   Assume that the polar part of the Higgs field is of order $2$ and
   non-semisimple both near $q_1$ and near $q_2$, i.e. we are in $(Nn)$.  
	Then the irregular
     Hitchin fibration $h$ on $\Mod^{ss} (\vec{\alpha})$ is biregular
   to the complement of the fiber at infinity which is of type $I_4^*$
   (equivalently $\tilde{D}_8$) in an elliptic fibration of the
   rational elliptic surface such that the set of other singular
   fibers of the fibration is:
  \begin{enumerate}
		\item if $a_{-3} b_{-3}\neq 0$, then two type $I_1$ fibers. 
  \end{enumerate}
	We note that if $a_{-3}=0$ or $b_{-3}=0$ then the fibration is not elliptic.
 \end{thm}


\subsection{Statement of results in the Painlev\'e II and IV cases}
\label{subsec:result_painleve_IV}

Next we turn to the cases when the orders of the poles of $\t$ are $3$
at $q_1$ and $1$ at $q_2$.  Just as before, we distinguish several
subcases, once again based on semisimplicity. In this case, we have
four subcases, and the corresponding results read as follows.
The natural parameters for the moduli of Higgs bundles are
  again of the form $(Ss), (Sn), (Ns), (Nn)$, this time the
  letters $S, N, s, n$ referring to sets of natural complex
  parameters 
\begin{align*}
  (S) &\quad a_+, a_-, b_+, b_-, \lambda_+, \lambda_- \\
  (N) &\quad b_{-6}, b_{-5}, b_{-4}, b_{-3}, b_{-2} \\
  (s) &\quad \mu_+, \mu_- \\
  (n) &\quad b_{-1}
\end{align*}
\noindent of geometric meaning detailed in \eqref{eq:31_ss3},
\eqref{eq:31_nil3}, \eqref{eq:31_ss1}, \eqref{eq:31_nil1}, see also
Figures~\ref{fig:blowup31A1}, \ref{fig:blowup31A2} and
\ref{fig:blowup31B}.  Once again, we derive new symbols $M,L, A$ and
$B$ out of these natural parameters as follows:
\begin{equation}\label{eq:31_notation}
	A=a_--a_+,\quad  B=b_--b_+,\quad  L=\lambda_--\lambda_+, 
\quad M=\mu_--\mu_+,  \quad
	Q=8b_{-5},\quad  R=b_{-4}^2+4 b_{-3}.
\end{equation}
Again, the semisimple polar parts are called \emph{regular} if $A \neq 0$ (respectively, $M\neq 0$).

In the next theorem the appropriate discriminant $\Delta $ is equal to
\[
48 A^4 \left(L^2+3 M^2\right)^2+64 A^3 B^2 L \left(L^2-9 M^2\right)+24 A^2 B^4 \left(L^2+3 M^2\right)-B^8.
\]
 \begin{thm}[PIV]\label{thm:PIV}
   Assume that the polar part of the Higgs field is of order $3$ and
   regular semisimple near $q_1$ and of order $1$ and regular 
	semisimple near $q_2$,
   i.e., we consider the case (Ss).  Then the irregular Hitchin
     fibration $h$ on $\Mod^{ss} (\vec{\alpha})$ is biregular to the
   complement of the fiber at infinity which is of type $\Et6$ in an
   elliptic fibration of the rational elliptic surface such that the
   set of other singular fibers of the fibration is:
  \begin{enumerate}
		\item if $L=\pm M$  and $B^2=\pm 4AM$ 
(consequently $\Delta = 0$) then an $I_1$ fiber and \label{thm:PIV_III}
		\begin{enumerate}
		 \item for generic weights a type $III$ fiber,
		 \item for special weights a fiber in the class $\Lef + \Pt$, with the class of the corresponding fiber of $\Mod^{s} (\vec{\alpha})$ given by $\Lef$; 
		\end{enumerate}
		\item if $L=\pm M$ and $B^2=\mp 12AM$ 
(consequently $\Delta = 0$) then a type $II$ fiber and \label{thm:PIV_I2_II}
		\begin{enumerate}
		 \item for generic weights an $I_2$ fiber,
		 \item for special weights a fiber in the class $\Lef$, with the class of the corresponding fiber of $\Mod^{s} (\vec{\alpha})$ given by $\Lef - \Pt$; 
		\end{enumerate}
		\item if $L=\pm M$, $B^2 \neq \pm 4AM$ and $B^2 \neq \mp 12AM$ 
(consequently $\Delta \neq 0$) then two $I_1$ fibers and \label{thm:PIV_I2_I1}
		\begin{enumerate}
		 \item for generic weights an $I_2$ fiber,
		 \item for special weights a fiber in the class $\Lef$, with the class of the corresponding fiber of $\Mod^{s} (\vec{\alpha})$ given by $\Lef - \Pt$; 
		\end{enumerate}
		\item 
if $L \neq \pm M$ and $\Delta \neq 0$, four type $I_1$
fibers.\label{thm:PIV_I1}
		\item 
if $L \neq \pm M$, $\Delta = 0$ and $B=0$, then two type $II$
fibers;\label{thm:PIV_II_II}
		\item 
if $L \neq \pm M$, $\Delta = 0$ and $B\neq 0$, then a type $II$ and
two $I_1$ fibers;\label{thm:PIV_II_I1}
  \end{enumerate}
\end{thm}

In the next theorem $\Delta = 4 A^2 \left(B^2-6 A L\right)$. 
 \begin{thm}[Degenerate PIV]\label{thm:PIV_degenerate}
   Assume that the polar part of the Higgs field is of order $3$ and regular 
   semisimple near $q_1$ and of order $1$ and non-semisimple near
   $q_2$, that is, we are in case $(Sn)$.  Then {the irregular Hitchin
     fibration $h$ on} $\Mod^{ss} (\vec{\alpha})$ is biregular to the
   complement of the fiber at infinity (of type $\Et6$) in an elliptic
   fibration of the rational elliptic surface such that the set of
   other singular fibers of the fibration is:
  \begin{enumerate}
		\item if $\Delta = 0$ and $L=0$ then \label{thm:PIV_degenerate_III}
		 \begin{enumerate}
		  \item for generic weights a fiber in the class $2 \Lef$,  
		  \item for special weights a fiber in the class $\Lef + \Pt$, with the class of the corresponding fiber of $\Mod^{s} (\vec{\alpha})$ given by $\Lef$; 
		 \end{enumerate}
		\item 
if $\Delta = 0$ and $L\neq 0$, then a type $II$ fiber and a fiber in
the class $\Lef - \Pt$ ;
		\item if $\Delta \neq 0$ and $L=0$ then an $I_1$ fiber and \label{thm:PIV_degenerate_I2}
		\begin{enumerate}
		  \item for generic weights a fiber in the class $2  \Lef -  \Pt $,  
		  \item for special weights a fiber in the class $\Lef + \Pt$, with the class of the corresponding fiber of $\Mod^{s} (\vec{\alpha})$ given by $\Lef$; 
		\end{enumerate}
		\item if $B^2=-2AL$ and $L \neq 0$ (consequently $\Delta \neq 0$) then an $I_1$ fiber and 
		a fiber in the class $\Lef$; 
		\item if $\Delta \neq 0$ and $L \neq 0$ and $B^2 \neq-2AL$ then two $I_1$ fibers and a fiber in the class $\Lef - \Pt$.
  \end{enumerate}
 \end{thm}

 In the next theorem $\Delta = M^2 \left(27 M^2 Q^2-4 R^3\right)$.
 \begin{thm}[PII]\label{thm:PII}
   Assume that the polar part of the Higgs field is of order $3$ and
   non-semisimple near $q_1$ and of order $1$ and regular semisimple near
   $q_2$, i.e. we are in case $(Ns)$.  Then {the irregular Hitchin fibration $h$ on} $\Mod^{ss}
   (\vec{\alpha})$ is biregular to the complement of the fiber at
   infinity (of type $\Et7$) in an elliptic fibration of the rational
   elliptic surface such that the set of other singular fibers of the
   fibration is:
  \begin{enumerate}
		\item if $\Delta = 0$ and $Q \neq 0$, then a type $II$
                  and an $I_1$ fiber;
		\item if $\Delta \neq 0$ and $Q \neq 0$, then three
                  type $I_1$ fibers.
  \end{enumerate}
 We note that if $Q=0$ (equivalently, $b_{-5}=0$), then the fibration
 is not elliptic.
 \end{thm}

 \begin{thm}[Degenerate PII]\label{thm:PII_degenerate}
   Assume that the polar part of the Higgs field is of order $3$ and
   non-semisimple near $q_1$ and of order $1$ and non-semisimple near
   $q_2$, i.e., we are in $(Nn)$.  Then {the irregular Hitchin fibration $h$ on} $\Mod^{ss}
   (\vec{\alpha})$ is biregular to the complement of the fiber at
   infinity (of type $\Et7$) in an elliptic fibration of the rational
   elliptic surface such that the set of other singular fibers of the
   fibration is:
  \begin{enumerate}
		\item if $R = 0$ and $Q \neq 0$, then a fiber in the
                  class $\Lef$;
		\item if $R \neq 0$ and $Q \neq 0$ then an $I_1$ fiber
                  and a fiber in the class $\Lef - \Pt$;
  \end{enumerate}
	We note that if $Q=0$ (equivalently, $b_{-5}=0$), then the
        fibration is not elliptic.
 \end{thm}


\section{Preparatory material}\label{sec:prep_mat}
\subsection{Irregular Higgs bundles of rank $2$ on curves}

Let $C$ be a smooth projective curve over $\C$ and $D$ an effective Weil divisor over $C$ (possibly non-reduced). 
Throughout the main body of this paper we will be interested in the case $C = \CP1$ and 
\begin{align}
 D & = 2 \cdot \{ q_1 \} + 2 \cdot \{ q_2 \} \tag{{2,2}} \label{eq:D=2+2}\\
   &  \qquad \mbox{or} \notag \\
 D & = 3 \cdot \{ q_1 \} + \{ q_2 \} \tag{{3,1}} \label{eq:D=3+1} 
\end{align}
for some distinct points $q_1, q_2 \in \CP1$. 

\begin{defn}
 A rank-$2$ \defin{irregular Higgs bundle} is a pair $(\E , \theta)$ where $\E$ is a rank-$2$ vector bundle over $C$ and 
 $$
  \theta \in H^0 (C, \End (\E ) \otimes K_C (D)). 
 $$
\end{defn} 

For the local forms of the Higgs fields that we will consider, see
\eqref{eq:22_ss_z1}, \eqref{eq:22_ss_z2}, \eqref{eq:31_ss3} and
\eqref{eq:31_ss1} (regular semisimple case) and \eqref{eq:22_nil_z1},
\eqref{eq:22_nil_z2}, \eqref{eq:31_nil3} and \eqref{eq:31_nil1}
(non-semisimple case).

\begin{defn}
 A \defin{compatible quasi-parabolic structure} on an irregular Higgs bundle $(\E , \theta)$ at $q_j$ is the choice of a generalized eigenspace of the leading order term of $\theta$ at $q_j$ with 
 respect to some local coordinate. 
 A \defin{compatible parabolic structure} on $(\E , \theta)$ at $q_j$ is a compatible quasi-parabolic structure endowed with a real number $\alpha^j_i \in [0,1)$ 
 (the \defin{parabolic weight}) attached to every generalized eigenspace of the leading order term of $\theta$ at $q_j$. 
\end{defn}

A compatible quasi-parabolic structure for a Higgs bundle with
non-semisimple singular part at $q_j$ is redundant, whereas in the
regular semi-simple case it amounts to the choice of one of the two
eigendirections of its expansion.  A compatible parabolic structure
for a Higgs bundle with non-semisimple singular part at $q_j$ is just
the choice of a parameter $\alpha^j \in [0,1)$, whereas in the regular
  semi-simple case it amounts to the choice of a parabolic weight
  $\alpha^j_i \in [0,1)$ for each eigenvalue of its most singular
    term, where $i \in \{ +, - \}$.  For coherence of notations, even
    in the non-semisimple case we will use the notations $\alpha^j_+$
    and $\alpha^j_-$ and set
\begin{equation}\label{eq:par_wt_nil}
   \alpha^j_+ = \alpha^j_- = \alpha^j. 
\end{equation}
This convention reflects the fact that $\alpha^j$ has multiplicity $2$. 

\begin{defn}
  The \defin{parabolic degree} of an irregular Higgs bundle $(\E , \theta)$ endowed with a parabolic structure at both $q_1, q_2$ is 
  $$
    \pardeg (\E ) = \deg (\E ) + \sum_{j=1}^2 \sum_{i \in \{ +, - \}} \alpha^j_i, 
  $$ where $\deg (\E)=\langle c_1(\E ), [ C]\rangle$. The
      \defin{parabolic slope} of $(\E , \theta)$ is
  $$
    \parslope (\E ) = \frac{\pardeg (\E )}{\rank (\E)} = \frac{\pardeg (\E )}2.
  $$
\end{defn}

If $\pardeg (\E ) = 0$, then non-abelian Hodge theory establishes a
diffeomorphism between irregular Dolbeault and de Rham moduli
spaces~\cite{Biq-Boa}. Moreover, the combinatorics of the stability
  condition and the resulting geometry of the singular Hitchin fibers
  would be very similar if we fixed the parabolic degree to be equal
    to some other constant. Therefore we make the following
\begin{assn}\label{assn:par_deg}
 We will suppose $\pardeg (\E ) = 0$.
\end{assn}

\begin{defn}
 A rank-1 \defin{irregular Higgs subbundle} of an irregular Higgs bundle
 $(\E , \theta)$ is a couple $(\F, \theta_{\F})$ where $\F \subset \E$
 is a rank-1 subbundle such that $\theta$ restricts to
 $$
  \theta_{\F} \in H^0 (C, \End (\F ) \otimes K_C (D)). 
 $$
\end{defn}

In the rank-$2$ case, non-trivial Higgs subbundles are exactly rank
$1$ Higgs subbundles, hence from now on we only deal with rank-1
subbundles. It is easy to see that if $(\F, \theta_{\F})$ is an
irregular Higgs subbundle then for both $j\in \{ 1, 2 \}$ the fiber
$\F_{q_j}$ of $\F$ at $q_j$ must be a subspace of one of the
generalized eigenspaces of the leading order term of $\theta$ at
$q_j$.  Notice that the fiber $\F_{q_j}$ has no non-trivial
filtrations. These observations show that the following definition
makes sense.

\begin{defn}
 The \defin{induced parabolic structure} on an irregular Higgs subbundle $(\F, \theta_{\F})$ is the choice of the parabolic weight 
 $$
  \alpha^j (\F ) = \alpha^j_i
 $$
 where $\alpha^j_i$ is the parabolic weight of $\E$ at $q_j$ corresponding to the generalized eigenspace containing $\F_{q_j}$. 
 The \defin{parabolic degree} of $(\F, \theta_{\F})$ is 
 $$
    \pardeg (\F ) = \deg (\F ) + \sum_{j=1}^2 \alpha^j(\F ). 
 $$
 The \defin{parabolic slope} of $(\F, \theta_{\F})$ is 
 $$
  \parslope (\F ) = \pardeg (\F ).
 $$
\end{defn}

 For an irregular Higgs subbundle $(\F, \t|_{\F})$ of $(\E, \t)$, $\theta$ induces a morphism on the quotient vector bundle 
 $$
   \Qt = \E / \F;
 $$
 we denote the resulting irregular Higgs field by
 $$
   {\bar{\t}} : \Qt \to \Qt \otimes K_C(D ). 
 $$
 \begin{defn}
  A \defin{quotient irregular Higgs bundle} of $(\E , \theta)$ is the irregular Higgs bundle $(\Qt , {\bar{\t}})$ obtained as above for some irregular Higgs subbundle $(\F, \theta_{\F})$ of $(\E , \theta)$.
  The \defin{induced parabolic structure} on a quotient irregular Higgs bundle $(\Qt , {\bar{\t}})$ is defined by the parabolic weight $\alpha^j(\Qt )$ such that 
  $$
    \{ \alpha^j(\F ), \alpha^j (\Qt ) \} = \{ \alpha^j_+, \alpha^j_- \}. 
  $$
  The \defin{parabolic {slope and} degree} of a quotient irregular Higgs bundle $(\Qt , {\bar{\t}})$ are defined as 
  $$ 
    {\parslope (\Qt ) =} \pardeg (\Qt ) = \deg(\Qt ) + \sum_{j=1}^2 \alpha^j(\Qt ).
  $$
 \end{defn}

 Let $(\F, \t|_{\F})$ be an irregular Higgs subbundle of $(\E, \t)$ and $(\Qt , {\bar{\t}})$ be the corresponding quotient irregular Higgs bundle of $(\E , \theta)$. 
 Then, by additivity of the degree we have 
 $$
  \pardeg (\F ) + \pardeg (\Qt ) = \pardeg (\E ). 
 $$

\begin{defn}
 An irregular Higgs bundle $(\E, \theta)$ is 
 \begin{itemize}
  \item \defin{{$\vec{\alpha}$-}semi-stable} if for any non-trivial irregular Higgs subbundle $(\F, \theta_{\F})$ we have 
 $$
  {\pardeg (\F )} \leq  \frac{\pardeg (\E )}{2}; 
 $$
 equivalently, if for any non-trivial irregular quotient Higgs bundle $(\mathcal{Q}, \bar{\t})$ we have 
 $$
  {\pardeg (\mathcal{Q} )} \geq  \frac{\pardeg (\E )}{2};
 $$
 \item \defin{{$\vec{\alpha}$-}stable} if the corresponding strict inequalities hold in the definition of {$\vec{\alpha}$-}semi-stability; 
 \item \defin{strictly $\vec{\alpha}$-semi-stable} if it is $\vec{\alpha}$-semi-stable but not $\vec{\alpha}$-stable; 
 \item \defin{$\vec{\alpha}$-polystable} if it is a direct sum of two
   rank-$1$ irregular Higgs bundles of the same parabolic slope as
   $(\E, \theta)$;
 \item \defin{strictly $\vec{\alpha}$-polystable} if it is $\vec{\alpha}$-polystable but not $\vec{\alpha}$-stable.
 \end{itemize} 
\end{defn}

Because of Assumption \ref{assn:par_deg}, the {$\vec{\alpha}$-}semi-stability condition boils down to 
$$
  {\pardeg (\F )} \leq 0
$$
and 
$$
  {\pardeg (\mathcal{Q} )} \geq  0, 
$$ and similarly for {$\vec{\alpha}$-}stability with strict
    inequalities. Moreover, if $\E = \E_1 \oplus \E_2 $ the
    $\vec{\alpha}$-polystability condition means
$$
	\parslope (\E )=\parslope (\E_1)=\parslope (\E_2)=0.
$$

\begin{prop}\label{propdef:JH}
Let $(\E, \theta)$ be a strictly $\vec{\alpha}$-semi-stable irregular Higgs bundle. 
 Then, there exists a filtration 
 $$
  \E = \E_0 \supset \E_1 \supset \E_2 = 0
 $$
 by subbundles preserved by $\theta$ so that the irregular Higgs bundles induced on the vector bundles 
 $$
   \E_0 / \E_1, \quad \E_1
 $$ are $\vec{\alpha}$-stable of the same parabolic slope as $(\E,
   \theta)$.  Moreover, the isomorphism classes of the associated
   graded irregular Higgs bundles with respect to this filtration are
   uniquely determined up to reordering.
 \end{prop}
This filtration is called the \defin{Jordan--H\"older filtration}, see \cite{Seshadri}. 

\begin{proof}
If $\F$ is a destabilizing subbundle then we set $\E_1 = \F$.  
Then, $\Qt = \E / \F$ is a vector bundle because $\F$ is a subbundle rather than just a subsheaf. 
Stability of the rank-$1$ irregular Higgs bundles on $\F, \Qt$ is obvious. 
The slope condition immediately follows from additivity of the parabolic degree. 

As for uniqueness, assume there exists another filtration
$$
  \E = \E_0 \supset \E'_1 \supset \E_2 = 0
$$
satisfying the same properties, with $\E'_1 \neq \E_1$. Then, on a Zariski open subset of $\CP1$ there is a direct sum decomposition 
\begin{equation}\label{eq:direct_sum}
   \E = \E_1 \oplus \E'_1 
\end{equation}
preserved by $\theta$. The set of parabolic weights of $\E$ is the union of the sets of parabolic weights of $\E_1$ and of $\E'_1$. 
On the other hand, we clearly have an inclusion of sheaves 
$$
  \E \supseteq \E_1 \oplus \E'_1 
$$
over $\CP1$. It follows from the above observation and the equality of parabolic slopes that the algebraic degree of the two sides of 
this formula agree. We infer that \eqref{eq:direct_sum} holds over $\CP1$, in particular the couples of rank-$1$ Higgs bundles 
$$
   \E_0 / \E_1 \cong \E'_1, \quad \E_1
$$
and
$$
   \E_0 / \E'_1 \cong \E_1 , \quad \E'_1
$$
agree up to transposition. 
\end{proof}

\begin{remark}
 The Jordan--H\"older filtration of an $\vec{\alpha}$-stable irregular
 Higgs bundle $(\E, \theta)$ is defined to be the trivial filtration.
 Clearly, this filtration also has the property that the associated
 graded object with respect to it only contains stable objects.
\end{remark}

\begin{defn}
 Let $(\E_1, \theta_1)$ and $(\E_2, \theta_2)$ be two semi-stable irregular Higgs bundles of rank $2$. 
 We say that they are \defin{S-equivalent} if the associated graded Higgs bundles for their 
 Jordan--H\"older filtrations are isomorphic. 
\end{defn}
In particular, if $(\E_1, \theta_1)$ and $(\E_2, \theta_2)$ are stable then they are S-equivalent if and only if they are isomorphic.

\subsection{Irregular Dolbeault moduli spaces}
The results of this section hold in (or at least, can be directly
generalized to) the case of irregular Higgs bundles of arbitrary rank.

Let us spell out the basic existence results that we will use.  These
results follow from the work of O.\,Biquard and Ph.\,Bolach in the
semi-simple case, and from the work of T.\,Mochizuki in the general
case.
\begin{thm}\label{thm:BB1}
 There exists a smooth hyperK\"ahler manifold $\Mod^s
 (\vec{\alpha})$ parameterizing isomorphism classes of
 $\vec{\alpha}$-stable irregular Higgs bundles of the given
 semi-simple irregular types with fixed parameters.
\end{thm}

\begin{proof}
 Assume first that the irregular type of $(\E ,\theta )$ near the
 marked points is semi-simple, i.e. that the local forms of $\theta$
 are given by \eqref{eq:22_ss_z1} and \eqref{eq:22_ss_z2} in the case
 \eqref{eq:D=2+2} and by \eqref{eq:31_ss3} and \eqref{eq:31_ss1} in
 the case \eqref{eq:D=3+1}.  \cite[ Theorem~5.4]{Biq-Boa} shows that
 irreducible solutions of Hitchin's equations in certain weighted
 Sobolev spaces up to gauge equivalence form a smooth hyperK\"ahler
 manifold.  \cite[Theorem~6.1]{Biq-Boa} implies that if
 $\mu_{\vec{\alpha}} (\E) = 0$ then irreducible solutions of Hitchin's
 equations up to gauge equivalence are in bijection with analytically
 stable irregular Higgs bundles up to gauge equivalence. Finally,
 according to \cite[Section~7]{Biq-Boa} the category of analytically
 stable irregular Higgs bundles (with gauge transformations as
 morphisms) is in equivalence with the groupoid of algebraically
 stable irregular Higgs bundles, and moreover this equivalence
 respects the analytic and algebraic $\vec{\alpha}$-stability
 conditions. This proves the statement in the semi-simple case.
 
 The proof of the general case follows similarly from the existence of
 a harmonic metric, see \cite[Corollary~16.1.3]{Mocsizuki}.
\end{proof}
It is possible to extend this result slightly in order to take into account all $\vec{\alpha}$-semi-stable irregular Higgs bundles. 
\begin{thm}\label{thm:213}
 There exists a {moduli stack} $\Mod^{ss} (\vec{\alpha})$ parameterizing 
 S-equivalence classes of $\vec{\alpha}$-semi-stable irregular Higgs bundles of the given semi-simple irregular 
 types with fixed parameters. 
\end{thm}
\begin{proof}
 It follows from \cite[Theorem~6.1]{Biq-Boa} {in the
  semi-simple case and 
\cite[Corollary~16.1.3]{Mocsizuki} in general}
that there exists a compatible Hermitian--Einstein metric for the
irregular Higgs bundle $(\E ,\theta )$ if and only if it is
$\vec{\alpha}$-poly-stable. 
 We deduce that the space 
 \begin{equation}\label{eq:poly-stable_moduli}
    \{ \vec{\alpha}\mbox{-polystable } (\E ,\theta ) \} / \mbox{gauge equivalence} 
 \end{equation}
 is the quotient of an infinite-dimensional vector space by the action of an (infinite-dimensional) gauge group. 
This endows \eqref{eq:poly-stable_moduli}
 with the structure of a {stack in groupoids}. The stable locus is
 treated in Theorem \ref{thm:BB1}.  It is therefore sufficient to
 prove:
 \begin{lem}\label{lem:bijection}
  There is a bijection between the set of isomorphism classes of strictly $\vec{\alpha}$-polystable irregular Higgs bundles and the set of 
  S-equivalence classes of strictly $\vec{\alpha}$-semi-stable irregular Higgs bundles. 
 \end{lem}
 \begin{proof}
  According to Proposition~\ref{propdef:JH}, the map that associates
  to a strictly $\vec{\alpha}$-semi-stable irregular Higgs bundle the
  isomorphism class of the associated graded object of its
  Jordan--H\"older filtration is well-defined. By the definition of
  S-equivalence, this map factors to an injective map $\iota$ from
  $$
    \{ \mbox{strictly } \vec{\alpha}\mbox{-semi-stable } (\E ,\theta ) \} / \mbox{S-equivalence} 
  $$
  to 
  $$
    \{ \mbox{strictly } \vec{\alpha}\mbox{-polystable } (\E ,\theta ) \} / \mbox{isomorphism} .
  $$ As any strictly $\vec{\alpha}$-polystable object is also strictly
    $\vec{\alpha}$-semi-stable, $\iota$ is also surjective.
 \end{proof}
\noindent This concludes the proof of Theorem~\ref{thm:213}.
\end{proof}

\subsection{Irregular Hitchin fibration}\label{subsec:Hitchin}
Let us denote by $\mbox{Tot} (K_C (D))$ the total space of the line bundle $K_C (D)$ and let 
$Z$ stand for the compactification of $\mbox{Tot} (K_C (D))$ by one curve at infinity: 
\begin{equation}
 Z_C (D) = \mathbb{P}_C (K_C (D) \oplus \O_C). 
\end{equation}
The surface $Z_C (D)$ is projective with a natural inclusion of $C$
given by the $0$-section of $K_C (D)$.  In the case $C = \CP1$ and
\eqref{eq:D=2+2} or \eqref{eq:D=3+1} we have
\begin{equation}
\label{eq:HirzeZCD}
  Z_C (D) = \mathbb{F}_2, 
\end{equation}
the Hirzebruch surface of degree $2$. We will denote by 
\begin{equation}\label{eq:projection}
 p \colon  Z_C (D) \to C
\end{equation}
the canonical projection. By an abuse of notation, we will also denote by $p$ the restriction of this projection to any subscheme of $Z_C (D)$. 
Let $\zeta$ denote the canonical section of $p^* K_C (D)$. 

Consider an irregular Higgs bundle $(\E, \theta)$ of rank $2$.  For
the identity automorphism $\mbox{I}_{\E}$ of $\E$ we may consider the
characteristic polynomial
\begin{equation}\label{eq:characteristic_polynomial}
  {\chi_{\theta} (\zeta ) = \det (\zeta \mbox{I}_{\E} - \theta )} =
  \zeta^2 + s_1 \zeta + s_2,
\end{equation}
where we naturally have 
$$
  s_1 \in H^0 (C, K_C (D) ), \quad s_2 \in H^0 (C, K_C^2 (2\cdot D) ). 
$$

\begin{defn}
 The \defin{irregular Hitchin map} of $\Mod^{ss}$ is defined by 
 \begin{align*}
  h\colon \Mod^{ss}(\vec{\alpha}) & \to H^0 (C, K_C (D) ) \oplus  H^0 (C, K_C^2 (2\cdot D) ) \\
  (\E, \theta) & \mapsto (s_1 , s_2 ). 
 \end{align*}
  The curve ${Z}_{(s_1, s_2)}$ in $Z_C (D)$ with Equation~\eqref{eq:characteristic_polynomial} is called the \defin{spectral curve} of $(\E , \theta )$. 
\end{defn}

For reasons that will become clear in the discussion preceding \eqref{eq:tq2} and \eqref{eq:tq3}, we use the simpler notation 
\begin{equation}\label{eq:t}
 t = (s_1, s_2). 
\end{equation}
This quantity $t$ is a natural coordinate of the Hitchin base $B$;
the curve $Z_{(s_1,s_2)}$ will be denoted by $Z_t$. 

\begin{thm}[\cite{Sz-spectral}]\label{thm:spectral}
 There exists a ruled surface $\widetilde{Z}_C (D)$ birational to $Z_C (D)$ such that the groupoid of irregular Higgs bundles of the given semi-simple irregular 
 types with fixed parameters is isomorphic to the relative Picard groupoid of torsion-free coherent sheaves of rank $1$ over an open subset in a 
 Hilbert scheme of curves in $\widetilde{Z}_C (D)$. \qed 
\end{thm}

The surface $\widetilde{Z}_C (D)$ can be explicitly described in terms
of a sequence of blow-ups, depending on the parameters appearing in
the irregular type.  The conditions that one needs to impose on the
support curves of the torsion-free sheaves on $\widetilde{Z}_C (D)$
are also very explicit.  We do not give a detailed description neither
of these conditions nor of $\widetilde{Z}_C (D)$ in complete
generality, because they involve much notation.  We recommend the
interested reader to refer to \cite{Sz-spectral}.  In the following, we
will use Theorem \ref{thm:spectral} in two particular cases, and we
will spell out the surface $\widetilde{Z}_C (D)$ and the conditions on
torsion-free sheaves resulting from the general construction of
\cite{Sz-spectral} only in these cases, see
Figures~\ref{fig:blowup22}-\ref{fig:blowup31B}.
\begin{notation}\label{nota:X}
 In this paper we will write 
 $$
  X = \widetilde{Z}_C (D). 
 $$
\end{notation}
This shorthand is justified because we will consider normalizations of
individual fibers $X_t$ of $\widetilde{Z}_C (D)$, traditionally
denoted by $\tilde{X}_t$, and we prefer to avoid double tildes.

It follows from Theorem \ref{thm:spectral} that in the semi-simple
case the image of $h$ is a {Zariski open} subset of a linear system
$B$ in the complete linear system $L = |r C|$ of curves in $Z_C(D)$.
We will show similar statements in some non-semisimple cases, see
Lemma \ref{lem:adjoint_orbit}.
\begin{defn}
For $t = (s_1, s_2) \in B$ the \defin{semi-stable Hitchin fiber} over $t$ is 
\begin{equation}\label{eq:Hitchin_fiber}
  \Mod_t^{ss}(\vec{\alpha}) = h^{-1} (t). 
\end{equation}
The Hitchin fiber over $t$ has a Zariski open subvariety 
$$
  \Mod_t^{s}(\vec{\alpha}) \subseteq \Mod_t^{ss}(\vec{\alpha})
$$ 
called the \defin{stable Hitchin fiber} parameterizing stable irregular Higgs bundles in $h^{-1} (t)$.
\end{defn}
If the curve ${Z}_t$ corresponding to some $t \in B$ is irreducible and reduced (in particular, if it is smooth) then we have 
$$
  \Mod_t^{s}(\vec{\alpha}) = \Mod_t^{ss}(\vec{\alpha}).
$$
Indeed, as the spectral curve of any sub-object $(\F, \theta_{\F})$ of any $(\E , \theta ) \in \Mod_t^{ss}(\vec{\alpha})$ is a subscheme of ${Z}_t$, we see that 
under the above assumptions any $(\E , \theta ) \in \Mod_t^{ss}(\vec{\alpha})$ is in fact irreducible, hence stable.

\subsection{The Grothendieck ring}\label{subsec:Grothendieck}
Let $\Var_{\C}$ be the category of algebraic varieties over $\C$.  We
let $\Z [\Var_{\C}]$ stand for the abelian group of {formal} linear
combinations of varieties with integer coefficients.  We introduce a
ring structure on $\Z [\Var_{\C}]$ by the defining the product as the
Cartesian product.
We introduce the equivalence relation $\sim$ on $\Z [\Var_{\C}]$
generated by the following relations: for any variety $X$ and proper
closed subvariety $Y \subset X$ we let
$$
  X \sim (X \setminus Y) + Y. 
$$
\begin{defn}
 The \defin{Grothendieck ring of varieties} $\Mot$ is the quotient ring 
 $$
  \Mot = \Z [\Var_{\C}] / \sim. 
 $$
 The class of an algebraic variety $X$ is denoted by $[X]$. 
\end{defn}
We use the notation 
$$
  \Lef = [\C ]
$$
for the class of the line and 
$$
  \Pt = [\mbox{point} ]
$$
for the class of a point. 
In particular, we have 
\begin{align*}
  [\CP1] & = \Lef + \Pt, \\
  [\C^{\times}] & = \Lef - \Pt .
\end{align*}
We will be interested in the classes $[\Mod_t^{s}(\vec{\alpha})]$ and
$[\Mod_t^{ss}(\vec{\alpha})]$ of the (semi-)stable Hitchin fibers over
all points $t\in B$.


\section{Elliptic pencils on the Hirzebruch surface ${\mathbb {F}}_2$}
\label{sec:ell_penc}

By Equation~\eqref{eq:HirzeZCD}, the ruled surface $Z_C(D)$ can be 
identified with the second Hirzebruch surface ${\mathbb {F}}_2$.
In this section we will examine pencils on ${\mathbb {F}}_2$
generated by the following two curves. The curve $C_{\infty}$ \emph{at
  infinity} has three components: the section at infinity (the one
with homological square $-2$) with multiplicity two together with two
fibers, which have (a) multiplicities two (called the $(2,2)$-case),
or (b) one of them is of multiplicity three, the other is of
multiplicity one (which is referred to as the $(3,1)$-case).  The
other curve generating the pencil is disjoint from the section at
infinity and intersects the generic fiber twice.  Such a curve is
called a \emph{double section} of the ruling on the Hirzebruch surface
${\mathbb {F}}_2$.
 
A simple homological computation shows that the two curves above are
homologous: if $S_{\infty}$ denotes the homology class of the section
at infinity, $S_0$ is the homology class of the 0-section and $F$ is the
homology class of the fiber of the ruling $p\colon {\mathbb {F}}_2\to
\CP{1}$, then the identity $S_{\infty}=S_0-2F$ implies that the double
section and the curve at infinity described above are
homologous. Since the homological square of $2S_0$ is eight, the
pencil becomes a fibration once we blow up the Hirzebruch surface
eight times. Since there are higher order base points in the pencil,
we need to apply infinitely close blow-ups. Indeed, in each case there
are two, three or four base points.  It is a simple fact that the
eight-fold blow-up of ${\mathbb {F}}_2$ (which itself is diffeomorphic
to $\CP{1} \times \CP{1}$) is diffeomorphic to the rational elliptic
surface, that is, the 9-fold blow-up of the projective plane, denoted
as $\CP{2} \# 9 \CPbar$.

According to Assumption~\ref{assn:elliptic}, in the following we will
consider only those pencils which result in elliptic fibrations; in
particular, the pencil should contain a smooth curve. In the above
setting this condition is equivalent to requiring that the double
section intersects the fiber component(s) of the curve $C_{\infty}$ at
infinity with multiplicity $>1$ only in smooth points.

\subsection{Singular fibers in elliptic fibrations}
\label{sec:sing_fib}
Singular fibers in an elliptic fibration have been classified by
Kodaira~\cite{Kodaira}. For description of these fibers, see also
\cite{HarerKasKirby, SSS}. In the following we will need only
a subset of all potential singular fibers, so we recall only those.

\begin{itemize} 
\item The \emph{fishtail fiber} (also called $I_1$) is
topologically an immersed sphere with one positive double point.
\item The \emph{cusp fiber} (also called $II$) is a sphere with a
single singular point, and the singularity is a cusp singularity (that
is, a cone on the trefoil knot).  
\item The $I_n$ fiber ($n\geq 2$) is a collection of $n$ spheres of
self-intersection $-2$, all with multiplicity one, intersecting each
other transversally in a circular manner, as shown by
Figure~\ref{fig:fibers}. In this paper we will need only 
the cases when $n=2,3$.

\begin{figure}[hb] 
\begin{center}
\includegraphics[width=2cm]{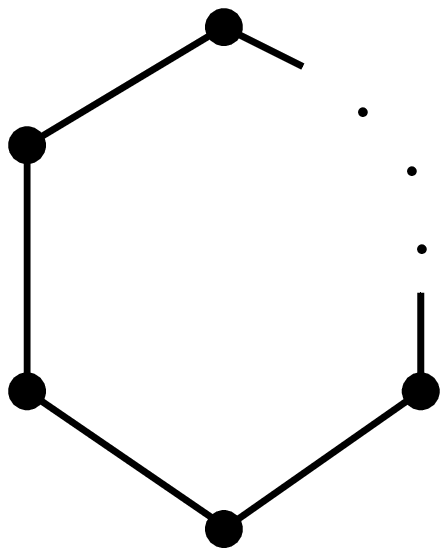} 
\end{center} 
\caption{\quad Plumbing graph of the singular fiber of type
  $I_n$. Dots denote rational curves of self-intersection $-2$ (and
  multiplicity one), and the dots are connected if and only if the
  corresponding curves intersect each other transversally in a
    unique point. In $I_n$ there are $n$ curves, intersecting along
  the circular manner shown by the diagram.}
\label{fig:fibers} 
\end{figure} 

\item The $I_n^*$-fiber ($n\geq 0$) contains $n+5$ transversally intersecting
  $(-2)$-spheres, as shown by Figure~\ref{fig:regi}(a). We will have
  fibers of such type for $n=2,3,4$.

\item The ${\tilde {E}}_6$, ${\tilde {E}}_7$, $III$ and $IV$
fibers all consist of $(-2)$-spheres intersecting according to 
the diagrams of Figures~\ref{fig:regi}(c), (d) and \ref{fig:34}(a) and (b).
\end{itemize}

\begin{figure}[hb] 
\begin{center}
\includegraphics[width=10cm]{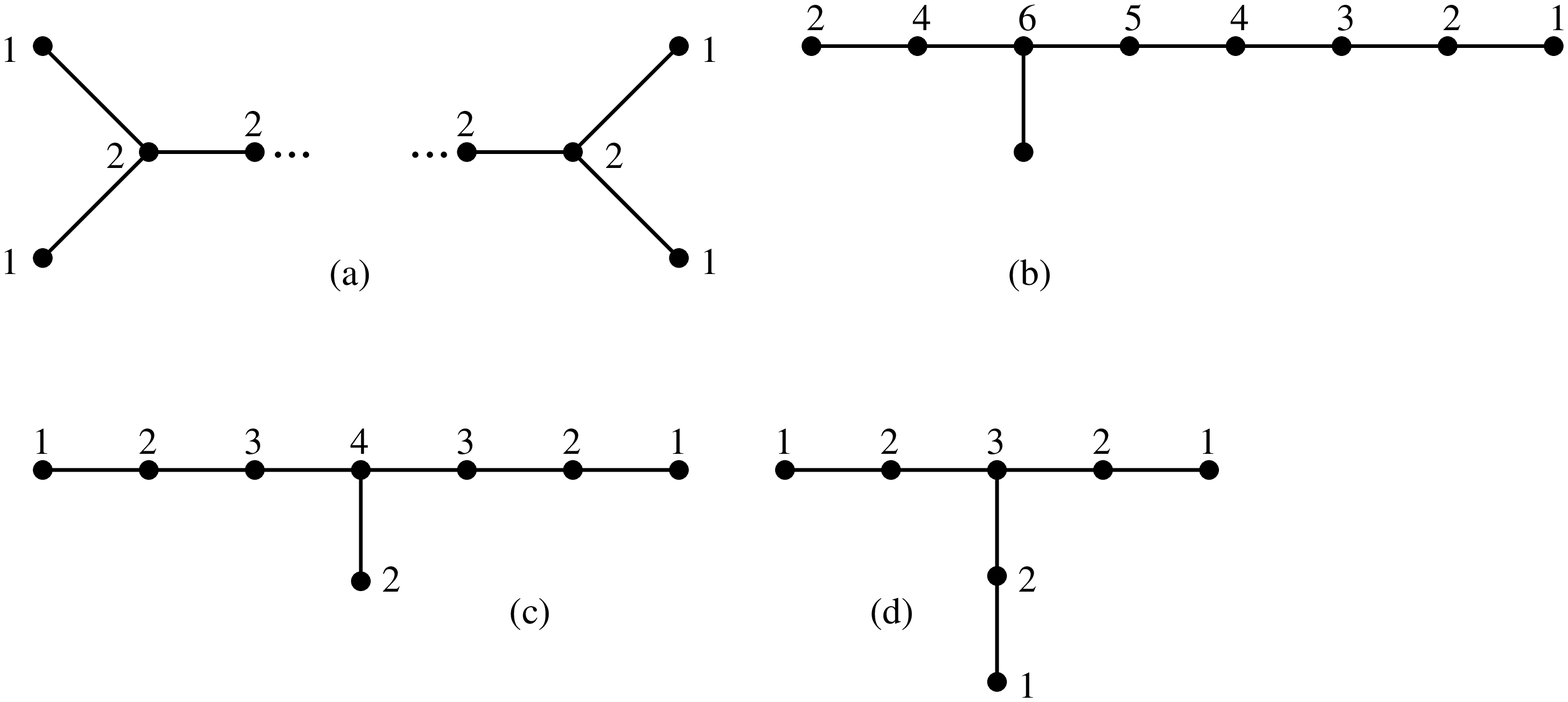} 
\end{center} 
\caption{\quad Plumbings of singular fibers
of types (a) $I_n^*$, (b) ${\tilde {E}}_8$, (c) ${\tilde {E}}_7$, and
(d) ${\tilde {E}}_6$.  Integers next to vertices indicate the
multiplicities of the corresponding homology classes in the fiber. All
dots correspond to rational curves with self-intersection $-2$. In
$I_n^*$ we have a total of $n+5$ vertices; in particular, $I_0^*$ admits
a vertex of valency four.} 
\label{fig:regi}
\end{figure}

\begin{figure}[hb] 
\begin{center}
\includegraphics[width=4cm]{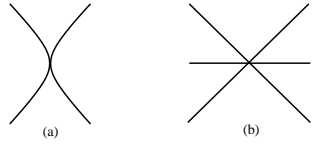} 
\end{center} 
\caption{\quad Singular fibers of types $III$ and $IV$ in elliptic fibrations.
In (a) the two curves are tangent with multiplicity two, and in (b) 
the three curves pass through one point and intersect each other there 
transversely.}
\label{fig:34} 
\end{figure}

A simple blow-up sequence shows that in case the curve
at infinity in the pencil is of type $(2,2)$ (that is, contains
two fibers, each with multiplicity two), then the fibration will have
\begin{enumerate}
\item An $I_4^*$-fiber if the pencil has two base points;
\item An $I_3^*$-fiber if the pencil has three base points;
\item An $I_2^*$-fiber if the pencil has four base points.
\end{enumerate}
 In more details, the blow-up process can be pictured as in
  Figures~\ref{fig:blowup22} and \ref{fig:blowup22bis}. In the diagram
  we only picture the blow-up of the base points on one of the fibers
  of the Hirzebruch surface. The base points of the pencil are smooth
  points of all the curves other than the curve at infinity, hence we
  have two cases: when there are two base points on the given fiber
  (depicted in Figure~\ref{fig:blowup22}) and when there is a single
  one (in which case the curves in the pencil are tangent to the fiber
  of the Hirzebruch surface) --- shown by
  Figure~\ref{fig:blowup22bis}. Each case requires 4 (infinitely
  close) blow-ups. The fibers of the ruling on the Hirzebruch surface
${\mathbb {F}}_2$ which are part of the curve at infinity (both 
with multiplicity 2) will be denoted by $F_2$ and $F_2'$, respectively.

\begin{figure} 
\input{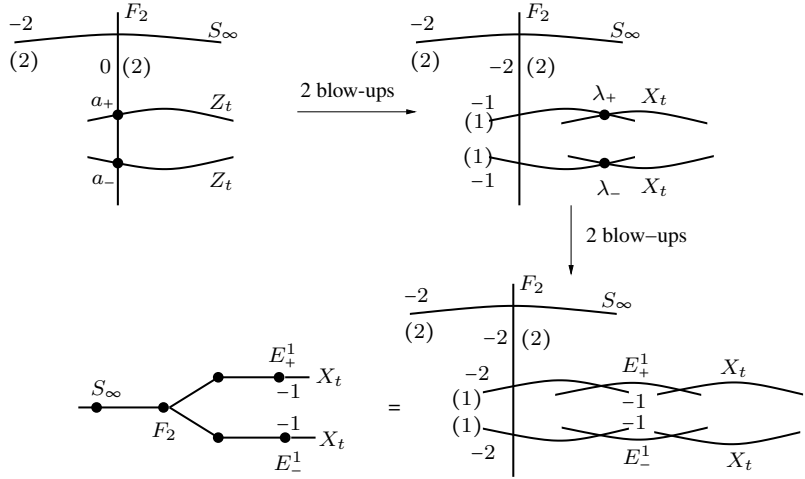} 
\caption{ The diagram shows the blow-up of the two base points on the
  fiber $F_2$ in two steps (4 blow-ups altogether).  This case was
  denoted by $(S)$ in Subsection~\ref{subsec:result_painleve_III}
  (while $(s)$ is the similar case on the other fiber $F_2'$ of the
  Hirzebruch surface with multiplicity 2, after substituting
  $(a_{\pm}, \lambda _{\pm})$ with $(b_{\pm}, \mu _{\pm})$ and
  $E^1_{\pm}$ with $E^2_{\pm}$).  The curves are denoted by arcs, the
  negative numbers next to them are the self-intersections, while the
  parenthetical positives are the multiplicities in the fiber at
  infinity. The curve of the pencil (giving rise to the fiber over
  $t\in B$) is denoted by $Z_t$ and its proper transform is by
  $X_t$. (Every proper transform, even in the intermediate steps will
  be denoted by $X_t$; hopefully this sloppyness in the notation will
  not create any confusion.)  Solid dots indicate the points where the
  next blow-up will be applied. We also include the plumbing
  description of the (relevant part of the) fiber at infinity.}
\label{fig:blowup22} 
\end{figure}

\begin{figure} 
\input{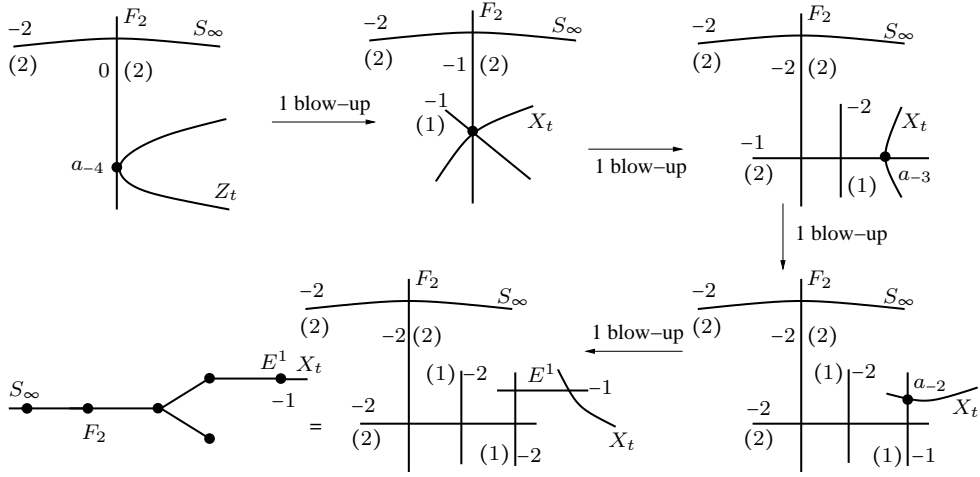} 
\caption{In this diagram we blow up the single base point four times.
  (This is the case denoted by $(N)$ in
  Subsection~\ref{subsec:result_painleve_III}; the corresponding case
  $(n)$ is given by considering $F_2'$ instead of $F_2$ and changing
  $a_{-j}$ to $b_{-j}$ for $j=4,3,2$ and $E^1$ to $E^2$.)  We use the
  same conventions as before.}
\label{fig:blowup22bis} 
\end{figure}

The classification of other singular fibers next to $I_2^*, I_3^*$ or 
$I_4^*$ reads as follows.

\begin{prop} [\cite{Miranda, Persson, SSS}]
\label{prop:I*_sing_fib}
An elliptic fibration on the rational elliptic surface
$\CP{2}\# 9 \CPbar$ with 
\begin{itemize} 
\item an $I_4^*$-fiber has two further $I_1$-fibers;
\item an $I_3^*$-fiber has further singular fibers which are either
(i) three $I_1$-fibers or (ii) an $I_1$-fiber and a fiber of type
$II$;
\item an $I_2^*$-fiber has further singular fibers as follows: either
(i) four $I_1$-fibers, or (ii) a type $II$ and two $I_1$, or (iii) an
$I_2$ and two $I_1$, or (iv) two $I_2$, or (v) two type $II$, or (vi)
a type $III$ and an $I_1$. \qed
\end{itemize}
\end{prop}

The similar blow-up sequence as before, now applied to the case $(3,1)$
(that is, when the curve at infinity has two fiber components, one with
multiplicity three, and the other with multiplicity one) will have
\begin{enumerate}
\item an ${\tilde {E}}_7$-fiber if the fiber with multiplicity
three contains a unique base point, and 
 \item an ${\tilde {E}}_6$-fiber if the fiber with multiplicity
three contains two base points.
\end{enumerate}

The blow-up sequence in this case is slightly longer, requires the
  analysis of more cases; these cases will be shown by
  Figures~\ref{fig:blowup31A1}, ~\ref{fig:blowup31A2}
  and~\ref{fig:blowup31B}.  Once again, we only depict one of the
  fibers of the Hirzebruch surface, which is part of the curve at
  infinity. Since the two multiplicities are different, they need
  different treatment. The fiber of the Hirzebruch surface in the
  curve at infinity having multiplicity 3 will be denoted by $F_3$,
  while the fiber with multiplicity 1 is $F_1$.

\begin{figure} 
\input{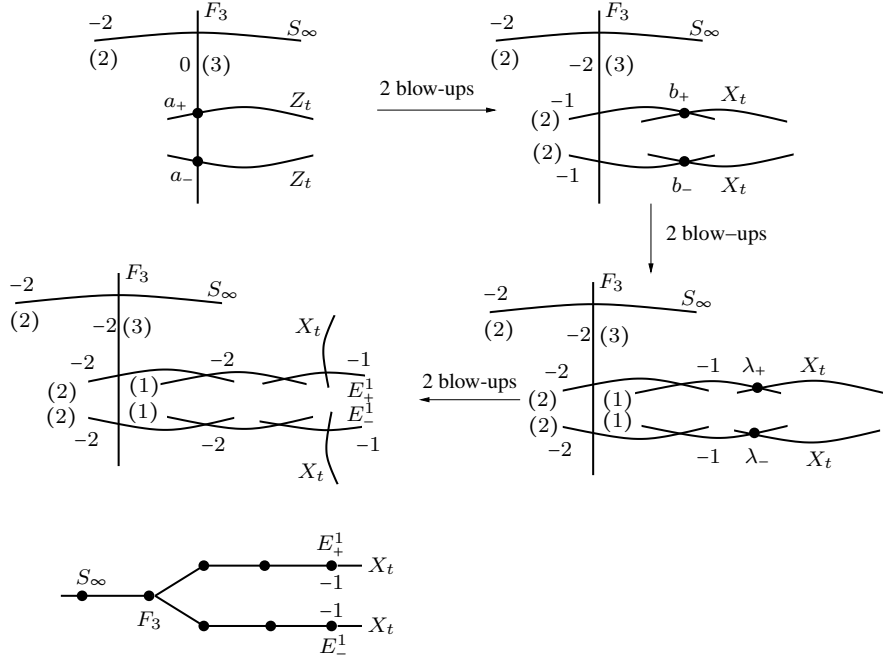} 
\caption{ The blow-up of the two base points on the fiber $F_3$ with
  multiplicity 3 (6 blow-ups altogether).  This is the case listed as
  $(S)$ in Subsection~\ref{subsec:result_painleve_IV}.  We use the
  same conventions as before.}
\label{fig:blowup31A1} 
\end{figure}

\begin{figure} 
\input{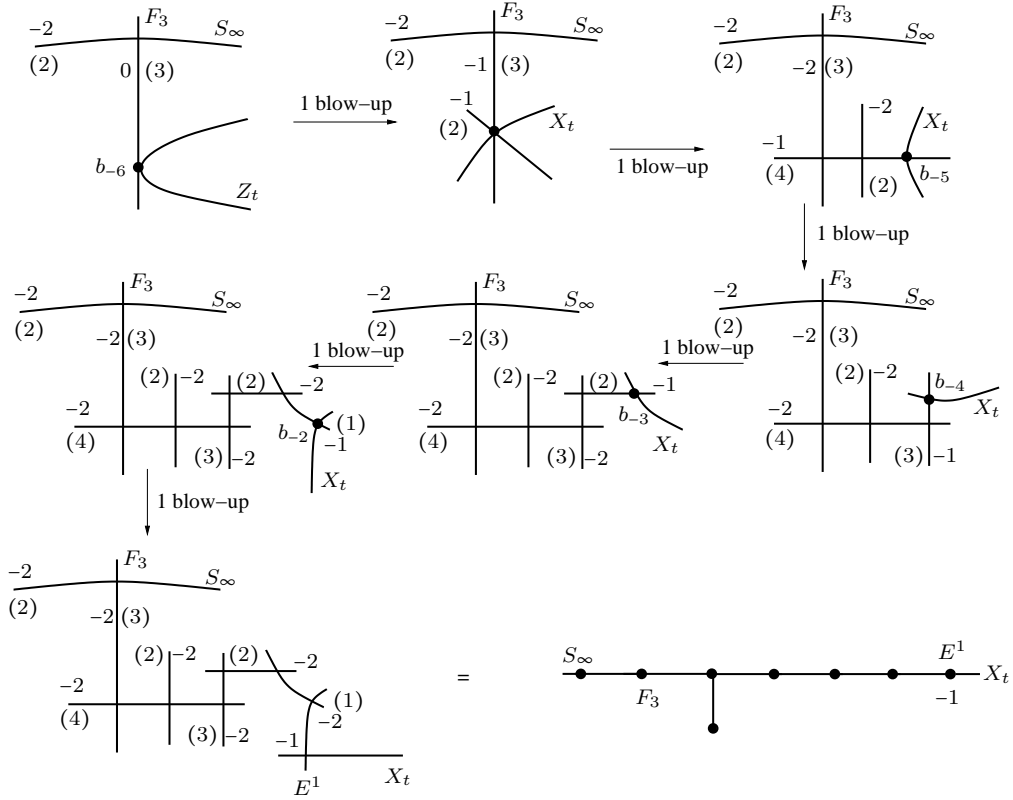} 
\caption{ In the diagram we blow up the single base point on the fiber
  $F_3$ of multiplicity 3 six times. 
This case corresponds to $(N)$ of Subsection~\ref{subsec:result_painleve_IV}.
 We use the same conventions as
  before.}
\label{fig:blowup31A2} 
\end{figure}

\begin{figure} 
\input{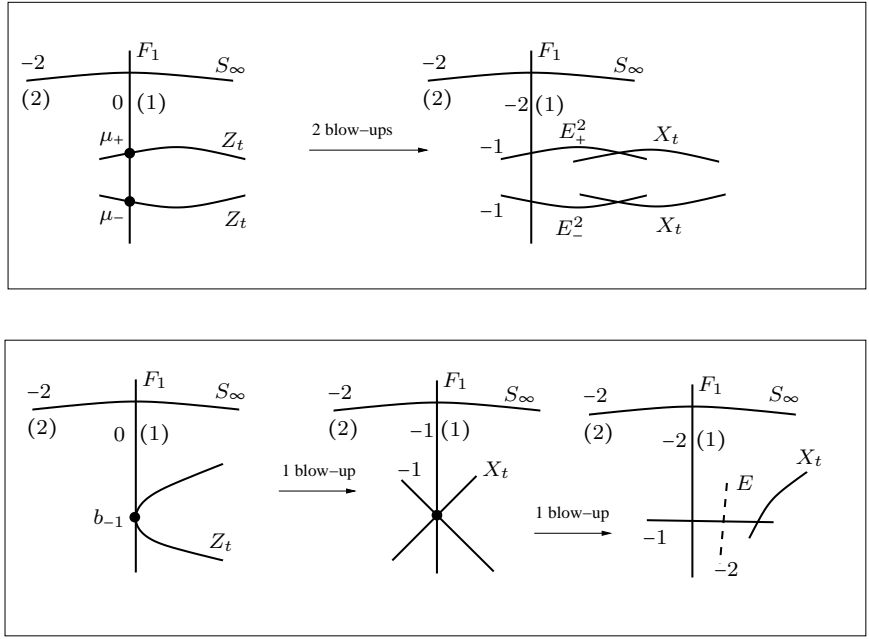} 
\caption{ The upper diagram shows the blow-up of the two base points
  on the fiber $F_1$ of multiplicity 1 (2 blow-ups altogether); this
  case corresponds to case $(s)$ of
  Subsection~\ref{subsec:result_painleve_IV}.  In the lower diagram we
  blow up the single base point on the fiber of multiplicity 1
  twice. The dashed curve $E$ is a rational $(-2)$-curve, which is part of
  another singular fiber.  This is the case which corresponds to $(n)$
  of Subsection~\ref{subsec:result_painleve_IV}.}
\label{fig:blowup31B} 
\end{figure}

The classification result in these cases reads as follows:
\begin{prop}[\cite{Miranda, Persson, SSS}]
\label{prop:E6-E7_sing_fib}
An elliptic fibration on the rational elliptic surface $\CP{2}\# 9 \CPbar$
with 
\begin{itemize}
\item an ${\tilde {E}}_7$-fiber has either (i) three $I_1$-fibers,
(ii) an $I_2$ and an $I_1$, (iii) a type $II$ and an $I_1$ or (iv)
a type $III$-fiber.
\item an ${\tilde {E}}_6$-fiber has either (i) four $I_1$-fibers,
(ii) an $I_2$ and two $I_1$, (iii) a type $II$ and two $I_1$, (iv)
a type $II$ and an $I_2$, (v) two type $II$, (vi) an $I_3$ and an $I_1$,
(vii) a type $III$ and an $I_1$, or (viii) a type $IV$ fiber. \qed
\end{itemize}
\end{prop}

\subsection{Pencils with sections}
We need to pay special attention to those fibrations which have fibers
with more than one component (besides the fiber coming from the curve at
infinity in the pencil). 
We say that the pencil
\emph{contains a section} if there is a curve in the pencil (other
than the curve at infinity) which
has more than one component. Notice that if the curve in the pencil
has more than one component, then it is the union of two sections of
the Hirzebruch surface ${\mathbb {F}}_2$ (equipped with its ${\mathbb
{CP}}^1$-fibration). Such curves give rise to singular fibers in the
elliptic fibration we get after blowing up ${\mathbb {F}}_2$ eight times, which
(according to the classification 
result recalled above) is either of type $III$, type $IV$, $I_2$ or an $I_3$ fiber.

In some cases, the existence of certain singular fibers indeed imply
the existence of sections in the pencil. Below we list two such cases.

\begin{lem}\label{lem:ss_section}
Consider the second Hirzebruch surface ${\mathbb {F}}_2$ with a pencil
of type $(2,2)$ or $(3,1)$ having four base points. If the elliptic
fibration resulting from the blow-up of the pencil
has a singular fiber of type $III$, $I_2$ or $I_3$ (besides
the type $I_2^*$-fiber in the $(2,2)$ case or the $\Et6$-fiber in the $(3,1)$
case), then the pencil contains a section.
\end{lem}

\begin{proof} Notice that all these fibers have more than one
components. So when blowing them down to ${\mathbb {F}}_2$, we either
get a curve in the pencil with more than one components (i.e., we have
a pencil with a section), or one or two components of the fiber must
be blown down. In this case, however, we get a curve in the pencil
which has a singular point in one of the base points of the
pencil. Since we consider only pencils containing at least one smooth
curve, this singular point cannot be on the fiber with multiplicity
two or three. Since there are two base points on the fiber with
multiplicity one in the $(3,1)$ case, both must be smooth points of
all the curves in the pencil, hence this second case cannot occur, and
we conclude that the pencil contains a section.
\end{proof}

\begin{lem}\label{lem:31_sn_section} 
Consider the second Hirzebruch surface ${\mathbb {F}}_2$ equipped with
a pencil of type $(3,1)$ having three base point, two in the fiber
with multiplicity three and one in the fiber with multiplicity one. If
the elliptic fibration resulting from the blow-up of the pencil on the
Hirzebruch surface has a singular fiber (besides the one of type
$\Et6$) of type $I_3$ or $IV$, then the pencil on the Hirzebruch
surface contains a section.
\end{lem} 

\begin{proof} Consider the blow-down of the $I_3$ or $IV$ singular
fiber. If the result in the pencil is not connected, we found a
section in the pencil.  If the result is a connected curve, it must
have a singularity, which is either a nodal or a cusp
singularity. Since by our assumption the pencil contains smooth
curves, the singularity must be in the base point on the fiber of
multiplicity one.  Blowing up this base point once, we get a curve of
two components intersecting each other either in two points, or in one
point with multiplicity two. Since these intersection points are not
base points of the resulting pencil, we will not blow them up again,
and so in the resulting fibration we will have such singular
fibers. Notice however, that there are no such pairs of curves in an
$I_3$ or a type $IV$ fiber, verifying that this second case is not
possible, hence proving the lemma.
\end{proof}

Notice, however, that (as the above proof shows) the fibration can
have fibers with more than one components even if the pencil has no
section. Indeed, consider the case $(3,1)$ with a single base point on
the fiber of multiplicity one.  Suppose first that a curve in the
pencil has a singularity at this base point, but it is not a section
of the Hirzebruch surface.  This implies that the curve is either
nodal or cuspidal, and when blowing up the base point once, we get an
$I_2$-fiber in case the double section had a node at the base point,
or a type $III$-fiber in case it had a cusp there.  In this case a
component of the singular fiber originates from the exceptional curve
of one of the blow-ups --- a similar phenomenon happens when we have
one base point on the fiber with multiplicity one and two on the 
fiber with multiplicity three and two sections either intersect
transversely or are tangent at the base point on the fiber with
multiplicity one.  In the first case this configuration provides an
$I_3$-fiber, in the second case we get a type $IV$-fiber.

Indeed, fibers of the elliptic fibration with more than one components
are always present in some cases:
\begin{lem}
\label{lem:sphere_in_fiber}
Suppose that we are in case the $(3,1)$ and on the fiber with multiplicity
one we have a single base point. Then there is a fiber (besides the
one originated from the curve at infinity, which is either $\Et6$ or
$\Et7$) which contains a $(-2)$-sphere that maps to a point under $p$,
where $p$ is the ruling, introduced in Equation~\eqref{eq:projection}.
\end{lem}

\begin{proof}
Consider a double section $C$ in the pencil. If the base point $P$ on 
the fiber of multiplicity one is a singular point of this curve,
then the above mentioned blow-up shows that the fiber coming from 
$C$ has a component which is a $(-2)$-sphere. 

Assume now that $P$ is a smooth point of $C$. Since in this case $C$
is tangent to the fiber, we need to apply two infinitely close
blow-ups, and the exceptional divisor of the first blow-up will become
a $(-2)$-sphere in one of the fibers of the resulting fibration, see
the dashed curve $E$ in the lower diagram of
Figure~\ref{fig:blowup31B}.
\end{proof}

We will need one further result similar to the previous ones:
\begin{lem}
\label{lem:31_ns_nextE7}
Suppose that we are in case the $(3,1)$, and on the fiber with multiplicity
one we have two base points, while on the fiber with multiplicity three
we have one. Then the fiber
originating from the curve at infinity is of type 
$\Et7$, and we cannot have a type $III$ or $I_2$ fiber next to it.
\end{lem}
\begin{proof}
It is a simple calculation to check that the fiber at infinity is of
type $\Et7$, cf. Figure~\ref{fig:blowup31A2} and the upper part of
Figure~\ref{fig:blowup31B}.  It is easy to see that the pencil cannot
contain a section, since the corresponding curve would admit a
singularity at the base point on the fiber with multiplicity three,
hence the pencil would not contain smooth curves. Therefore, a fiber
of type $III$ or $I_2$ would originate from a connected curve, hence
when blowing curves back down, one component of the fiber must be
blown down. This means that the other component has a singularity at
one of the base points. This is not possible in the base point on the
fiber of multiplicity three (since we assume the existence of a smooth
curve in the pencil). Similarly, the two base points on the fiber of
multiplicity one must be also smooth points of all curves, otherwise
the intersection multiplicity with that fiber would rise to at least
three. This concludes the argument.
\end{proof}

A certain converse of the above lemmas also holds:
\begin{lem}\label{lem:section_to_fiber}
Suppose that we are in case the $(2,2)$ or $(3,1)$. If the pencil contains
a section, then the elliptic fibration on the 8-fold blow up will contain
a fiber of type $I_2,I_3,III$ or $IV$.
\end{lem}

\begin{proof}
Assume first that the pencil has four base points.  Then the base points
are necessarily smooth points of each curve in the pencil.  The two
components of the reducible curve in the pencil can meet transversally
(in which case the fibration will have an $I_2$ fiber) or can be
tangent to each other (when the fibration will have a type $III$
fiber).

Suppose that there is a fiber of the Hirzebruch surface containing a
single base point, and the pencil contains a section. By our assumption (that
the generic curve in the pencil is smooth) this assumption implies
that the fiber in question is of multiplicity one, hence we need to
examine two further possibilities: we must be in case (3,1) and
the two components of the reducible curve in the
pencil either meet transversely in the  unique base point on the fiber
with multiplicity one, or the two curves are tangent there.
In the first case we will have a fiber of type $I_3$ in the pencil, and in
the second case we will get a fiber of type $IV$, verifying the claim of the
lemma.
\end{proof}


\section{The order of poles are $2$ and $2$}\label{sec:PIII}
\label{sec:22}

After these preliminaries, now we start proving the results announced 
in Section~\ref{sec:intro}. In this section we will discuss the 
complex surfaces relevant for the cases of 
Theorems~\ref{thm:PIII(D6)}, \ref{thm:PIII(D7)} and 
\ref{thm:PIII(D8)}. Recall from Notation~\ref{nota:X} the definition 
of the complex surface $X$.
We will prove: 

\begin{prop}\label{thm:main_22_ss}
   Assume that the polar part of the Higgs field is regular semisimple near
   $q_1$ and regular semisimple near $q_2$. Then ${X}$ is biregular to the
   complement of the fiber at infinity (of type $I_2^*$) in an
   elliptic fibration of the rational elliptic surface such that the
   set of other singular fibers of the fibration is:
  \begin{enumerate}
		\item if $\Delta=0$ and $L^2=M^2\neq0$, 
						then a type $III$ and an $I_1$ fibers;
		\item if $\Delta=0$, $L^2=-M^2\neq0$ and $M^3=8ABL$, 
						then two type $II$ fibers;
		\item if $\Delta=0$, $L^2=-M^2\neq0$ and $M^3  \neq 8ABL$, 
						then a type $II$ and two $I_1$ fibers;
		\item if $\Delta=0$ and $L^2\neq \pm M^2$, 
						then a type $II$ and two $I_1$ fibers again;
		\item if $\Delta \neq 0$ and $L=M=0$, 
						then two  $I_2$ fibers;
		\item if $\Delta \neq 0$ and $L^2=M^2\neq0$, 
						then a $I_2$ and two $I_1$ fibers ;
		\item if $\Delta \neq 0$ and $L^2 \neq M^2$, 
						then four $I_1$ fibers ;
  \end{enumerate}
	for $\Delta$ see (\ref{eq:22_ss_discr}).
	The case $\Delta=0$ and $L=M=0$ implies either $A=0$ or $B=0$, and 
	it does not give an elliptic fibration.
\end{prop}

The above statement is summarized by Table~\ref{tab:22_ss}. 

\begin{table}[htb]
\begin{center}
\begin{tabular}{c*{5}{>{\centering\arraybackslash}p{.17\linewidth}}}
	\multicolumn{1}{c|}{} & \multicolumn{1}{c|}{$\mathbf{\Delta=0}$} & \multicolumn{1}{c}{$\mathbf{\Delta\neq 0}$} \\
	\hhline{-#=|=}
	\multicolumn{1}{c||}{$\mathbf{L=M=0}$} & \multicolumn{1}{c|}{not semisimple} & \multicolumn{1}{c}{$2 I_2$} \\
	\hhline{-||-|-}
	\multicolumn{1}{c||}{$\mathbf{L^2=M^2\neq0}$} & \multicolumn{1}{c|}{$III+I_1$} & \multicolumn{1}{c}{$I_2+2I_1$} \\
	\hhline{-||-|-}
	\multicolumn{1}{c||}{$\mathbf{L^2=-M^2\neq0}$ and $\mathbf{M^3=8 A B L}$} & \multicolumn{1}{c|}{$2 II$} & \multicolumn{1}{c}{\multirow{3}{*}{$4I_1$}} \\
		\hhline{-||-|~}
	\multicolumn{1}{c||}{$\mathbf{L^2=-M^2\neq0}$ and $\mathbf{M^3 \neq 8 A B L}$} & \multicolumn{1}{c|}{\multirow{2}{*}{$II+2I_1$}} & \multicolumn{1}{c}{}  \\
	\hhline{-||~|~}
	\multicolumn{1}{c||}{$\mathbf{L^2 \neq \pm M^2}$} & \multicolumn{1}{c|}{} & \multicolumn{1}{c}{}
\end{tabular}
\end{center}
\caption{\quad 
The type of singular curves in (2,2) case with four base points.
We list the singular fibers next to the $I^*_2$ fiber.}
\label{tab:22_ss}
\end{table}

\begin{prop}\label{thm:main_22_sn}
   Assume that the polar part of the Higgs field is regular semisimple near
   $q_1$ and non-semisimple near $q_2$ (or vice versa). Then ${X}$ is
   biregular to the complement of the fiber at infinity (of type
   $I_3^*$) in an elliptic fibration of the rational elliptic surface
   such that the set of other singular fibers of the fibration is:
  \begin{enumerate}
		\item if $\Delta = 0$, then a type $II$ and an $I_1$ fibers;
		\item if $\Delta \neq 0$, then three  $I_1$ fibers;
  \end{enumerate}
	where $\Delta =	4 A^3 b_{-3} \left(2 L^3-27 A b_{-3}\right)$ in case $(Sn)$ 
	(or $\Delta = 4 B^3 a_{-3} \left(2 M^3-27 B a_{-3}\right)$ in case of $(Ns)$).
	If $b_{-3}=0$ (or $a_{-3}=0$) then the fibration is not elliptic.
\end{prop}

\begin{prop}\label{thm:main_22_nn}
   Assume that the polar part of the Higgs field is non-semisimple
   near $q_1$ and non-semisimple near $q_2$.  Then ${X}$ is biregular
   to the complement of the fiber at infinity (of type $I_4^*$) in an
   elliptic fibration of the rational elliptic surface such that the
   set of other singular fibers of the fibration is:
  \begin{enumerate}
		\item if $a_{-3} b_{-3}\neq 0$, then two type $I_1$ fibers.
  \end{enumerate}
	If $a_{-3}=0$ or $b_{-3}=0$ then the fibration is not elliptic.
\end{prop}

\begin{remark}
Indeed, all implications in the above propositions are if-and-only-if
statements.  We will prove implications in both directions.
\end{remark}

\subsection{Local form of irregular Higgs bundles}
\label{sec:22_local_form}
Our investigation is based on the local description of spectral 
curves on the Hirzebruch surface of degree 2. Introduce two local 
charts on $C=\CP{1}$: $U_1$ with $z_1\in \C$ (and 
$\{z_1=0\}=q_1$) and $U_2$ with $z_2 \in \C$ (and $\{ z_2=0\}=q_2$). 
In the case~\eqref{eq:D=2+2} the line bundle is 
$K_{\CP{1}} (2\cdot \{ q_1 \} + 2\cdot \{q_2\})$. 
The bundle $K_{\CP{1}} (2\cdot \{ q_1\} + 2\cdot \{q_2\})$ 
admits the trivializing sections $\kappa _i$ over $U_i$:
\begin{align*}
  \kappa_1 &= \frac{\d z_1}{z_1^2}, \\
	\kappa_2 &= \frac{\d z_2}{z_2^2}.
\end{align*}
The conversion from $\kappa_1$ to $\kappa_2$ is the following:
\begin{equation}
	\label{eq:22_triv_conv}
	\kappa_1 = \frac{\mathrm{d}z_1}{z_1^{2}}=-\d z_2 = -z_2^2 \kappa_2.
\end{equation}
The trivialization $\kappa_i$ induces a trivialization $\kappa_i^2$ on 
$K_C(2\cdot \{ q_1 \} + 2\cdot \{q_2\})^{\otimes 2}$, $i=1,2$.

The Hirzebruch surface can be covered by four charts. We will need
only two of those, since we only conisder curves disjoint from the
section at infinity (which is a component of the curve $C_{\infty}$ at
infinity). Let us denote $V_i \subset p^{-1} (U_i)$ the complement of
the section at infinity in $p^{-1} (U_i)$ ($i=1,2$). Let $\zeta \in
\Gamma \left(\mathbb{F}_2,p^* K_C(D)\right)$ be the canonical section,
and introduce $w_{i} \in \Gamma (V_i,\O)$ by
\begin{equation*}
	\zeta = w_{i} \otimes \kappa_i.
\end{equation*}
Use (\ref{eq:22_triv_conv}) for the conversion between $w_1$ to $w_2$:
\begin{equation*}
	w_2 \otimes \kappa_2=\zeta=w_1 \otimes \kappa_1=-z_2^2 w_1 \otimes \kappa_2.
\end{equation*}

Consider an irregular Higgs bundle $(\E, \theta)$ in the $\kappa_i$
trivializations ($i=1,2$). The local forms of $\theta$ near $q_1$
or $q_2$ are the following
\begin{equation}
	\label{eq:22_theta}
	\theta =  \sum_{n\geq -2} A_n z_1^n \otimes \d z_1 \quad \mbox{or} \quad \theta =  \sum_{n\geq -2} B_n z_2^n \otimes \d z_2,
\end{equation}
where $A_n, B_n \in \gl (2,\C )$.
Take the characteristic polynomial in Equation~(\ref{eq:characteristic_polynomial})
\begin{equation}\label{eq:22_char-poly1}
	\chi_{\theta} (\zeta ) = \det (\zeta \mbox{I}_{\E} - \theta) = \zeta^2 + s_1 \zeta + s_2, 
\end{equation} 
for some
\begin{equation*}
	s_1\in H^0(\CP1 , K(2\cdot \{ q_1 \} + 2\cdot \{q_2\})), \quad s_2 \in H^0(\CP1 , K(2\cdot \{ q_1 \} + 2\cdot \{q_2\})^{\otimes 2}). 
\end{equation*} 
This means that $s_1$ is a meromorphic differential and $s_2$ is a
meromorphic quadratic differential.

Let us set $\vartheta_1 = \sum_{n\geq 0} A_{n-2} z_1^n$ and 
$\vartheta_2 = \sum_{n\geq 0} B_{n-2} z_2^n$, so that we have
\begin{equation*}
	\theta  = \vartheta_i \otimes \kappa_i.
\end{equation*} 

If we divide by $\kappa_i$ in (\ref{eq:22_char-poly1}), then the characteristic polynomial may be rewritten as 
\begin{equation}
	\label{eq:22_char-poly2}
	\chi_{\vartheta_i} (w_{i} ) = \det (w_{i} \mbox{I}_{\E} - \vartheta_i) = w_{i}^2 + w_{i} f_i + g_i, 
\end{equation}
with 
\begin{equation*}
  s_1 = f_i \kappa_i, \quad s_2 = g_i \kappa_i^2 \quad (i=1, 2).
\end{equation*}

Now, as $K(2\cdot \{ q_1 \} + 2\cdot \{q_2\}) \cong \O(2)$, the
coefficients $f_1$ and $g_1$ in the $\kappa_1$ trivialization are polynomials
in $z_1$ of degree $2$ and $4$, respectively:
\begin{align*}
	f_1(z_1) &= - (p_2 z_1^2 + p_1 z_1 + p_0), \\
	g_1(z_1) &= - (q_4 z_1^4 + q_3 z_1^3 + q_2 z_1^2 + q_1 z_1 + q_0), 
\end{align*}
where all coefficients are elements of $\C$. 

Similarly, in the $\kappa_2$ trivialization using the formula~(\ref{eq:22_triv_conv}) we get
\begin{align*}
	f_2(z_2) &= p_0 z_2^2+p_1 z_2+p_2,\\
	g_2(z_2) &= -\left( q_0 z_2^4+q_1 z_2^3+q_2 z_2^2+q_3 z_2+q_4 \right).
\end{align*}

Finally, (\ref{eq:22_char-poly2}) gives two polynomials in variables $z_1, w_1$ or $z_2, w_2$. 
These polynomials are the local forms of spectral curves in $Z_C (D)$:
\begin{align}
	\label{eq:22_char-poly2_z1}
	\chi_{\vartheta_1} (z_1, w_{1}) & = w_1^2 - \left(p_2 z_1^2+p_1 z_1+p_0\right) w_1 -\left(q_4 z_1^4+q_3 z_1^3+q_2 z_1^2+q_1 z_1+q_0\right), \\
	\label{eq:22_char-poly2_z2}
	\chi_{\vartheta_2}  (z_2, w_{2}) & = w_2^2 + \left(p_0 z_2^2+p_1 z_2+p_2\right) w_2 - (q_0 z_2^4+q_1 z_2^3+q_2 z_2^2+q_3 z_2+q_4).
\end{align}

On the other hand, the spectral curve has an expansion near $q_1$ and 
$q_2$ in which these parameters have a geometric meaning: the 
matrices $A_n$ and $B_n$ $(n=-2,-1)$ in (\ref{eq:22_theta}) encode 
the base locus and the slope of the tangent line of a pencil. As 
indicated in Subsection~\ref{subsec:result_painleve_III} the letters 
$(S), (N), (s), (n)$ refer to the following cases which may occur 
independently of each other.

Near $z_1 = 0$: 
\begin{itemize}
  \item[$(S)$] the semisimple case
   \begin{equation} \label{eq:22_ss_z1}
   \t = \left[ \begin{pmatrix}
         a_+ & 0 \\
         0 & a_-   
        \end{pmatrix} z_1^{-2} 
				+
				\begin{pmatrix}
         \lambda_+ & 0 \\
         0 & \lambda_- 
        \end{pmatrix} z_1^{-1}
				+  O(1)
   \right] \otimes \d z_1,
   \end{equation}
  \item[$(N)$] the non-semisimple case 
   \begin{equation} \label{eq:22_nil_z1}
   \t = \left( \begin{pmatrix}
             a_{-4} & 1 \\
              0	& a_{-4}
            \end{pmatrix} z_1^{-2} 
            +
         \begin{pmatrix}
             0 & 0 \\
              a_{-3} & a_{-2}
            \end{pmatrix} z_1^{-1} 
            + O(1)
            \right) \otimes \d z_1.
   \end{equation}
\end{itemize}

Near $z_2 = 0$: 
\begin{itemize}
  \item[$(s)$] the semisimple case
   \begin{equation} \label{eq:22_ss_z2}
   \t = \left[ \begin{pmatrix}
         b_+ & 0 \\
         0 & b_-    
        \end{pmatrix} z_2^{-2} 
				+
				\begin{pmatrix}
         \mu_+  & 0 \\
         0 & \mu_- 
        \end{pmatrix} z_2^{-1} 
				+  O(1)  \right] \otimes \d z_2,
   \end{equation}
  \item[$(n)$] the non-semisimple case 
   \begin{equation} \label{eq:22_nil_z2}
   \t = \left( \begin{pmatrix}
             b_{-4} & 1 \\
              0	& b_{-4}
            \end{pmatrix} z_2^{-2} 
            +
         \begin{pmatrix}
             0 & 0 \\
              b_{-3} & b_{-2}
            \end{pmatrix} z_2^{-1} 
            + O(1)
            \right) \otimes \d z_2.
   \end{equation}
\end{itemize}

Since $\tr \t$ is a meromorphic function, 
the residue theorem implies 
\begin{equation*}
  \tr \Res_{z_1=0} \t + \tr \Res_{z_1=\infty} \t = 0.
\end{equation*}

According to the various cases regarding the semisimplicity or 
non-semisimplicity of the local forms at the two poles, this implies
\begin{subequations}
\begin{align} 
	\label{eq:residuum_22_ss}
  \lambda_+ + \lambda_- + \mu_+ + \mu_- & = 0, \\
	\label{eq:residuum_22_sn}
  \lambda_+ + \lambda_- + b_{-2}  & = 0, \\ 
	\label{eq:residuum_22_ns}
	\mu_+ + \mu_- + a_{-2}  & = 0, \\ 
	\label{eq:residuum_22_nn}
   a_{-2} + b_{-2}  & = 0.
\end{align}
\end{subequations}
Here and in what follows, for reasons of symmetry we do not consider 
the case where the local form at $z_1=0$ is non-semisimple and at $z_2=0$ 
is semisimple.

The roots of the characteristic polynomial (\ref{eq:22_char-poly2}) 
in $w_{i}$ have Puiseux expansions with respect to $z_i$. 
The first several terms of these expansions are equal to the 
eigenvalues of the matrices (\ref{eq:22_ss_z1}) -- (\ref{eq:22_nil_z2}).
In concrete terms:

\begin{itemize}
\item[$(S)$] The polar part of $\theta$ near $q_1$ has semisimple 
leading-order term. The series of the 'negative' 
root of $\chi_{\vartheta_1} (z_1, w_{1})$ in (\ref{eq:22_char-poly2_z1}) 
up to first order is equal to $a_-  + \lambda_- z_1$
and the 'positive' root up to first order is equal to 
$a_+ + \lambda_+ z_1$. Hence we get the equations
\begin{subequations}
\label{eq:22_ss_z1_eq}
\begin{align}
a_- =& \frac{p_0}{2} -\frac{1}{2}\sqrt{p_0^2+4 q_0}, \\
\lambda _- =& \frac{p_1 \sqrt{p_0^2+4 q_0}-p_0 p_1-2 q_1}{2 \sqrt{p_0^2+4 q_0}}, \\
a_+ =& \frac{p_0}{2}+\frac{1}{2} \sqrt{p_0^2+4 q_0}, \\
\lambda _+ =& \frac{p_1 \sqrt{p_0^2+4 q_0}+p_0 p_1+2 q_1}{2 \sqrt{p_0^2+4 q_0}}.
\end{align}
\end{subequations}

\item[$(N)$] The polar part of $\theta$ near $q_1$ has non-semisimple
leading-order term. This means that the polynomial 
$\chi_{\vartheta_1} (z_1, w_{1})$ has one ramified root $w_1$ with branch point $z_1=0$.
This leads to the formula $p_0^2+4 q_0=0$. We simplify the roots of 
$\chi_{\vartheta_1} (z_1, w_{1})$ using this condition. 
The Puiseux series of two roots up to first order 
are equal to the series of eigenvalues of the matrix 
(\ref{eq:22_nil_z1}) up to first order. Notice that the expansion 
allows half integer powers of the variable $z_1$. The resulting 
equations are:
\begin{subequations}
\label{eq:22_nil_z1_eq}
\begin{align}
a_{-4} =& \frac{p_0}{2}, \\
\sqrt{a_{-3}} =& \frac{1}{2} \sqrt{2 p_0 p_1+4 q_1}, \\
\frac{a_{-2}}{2} =& \frac{p_1}{2}.
\end{align}
\end{subequations}

\item [$(s)$] The polar part of $\theta$ near $q_2$ has semisimple 
leading-order term. The series of the 'negative' 
root of $\chi_{\vartheta_2} (z_2, w_{2})$ in (\ref{eq:22_char-poly2_z2}) 
up to first order is equal to $b_-  + \mu_- z_2$
and similarly the 'positive' root up to first order is equal to 
$b_+ + \mu_+ z_2$. The equations are: 
\begin{subequations}
\label{eq:22_ss_z2_eq}
\begin{align}
b_-=&-\frac{p_2}{2}-\frac{1}{2}\sqrt{p_2^2+4 q_4}, \\
\mu _-=&-\frac{p_1 \sqrt{p_2^2+4 q_4}+p_1 p_2+2 q_3}{2 \sqrt{p_2^2+4 q_4}}, \\
b_+=&-\frac{p_2}{2} +\frac{1}{2} \sqrt{p_2^2+4 q_4}, \\
\mu _+=&-\frac{p_1\sqrt{p_2^2+4 q_4}-p_1p_2-2 q_3}{2 \sqrt{p_2^2+4 q_4}}.
\end{align}
\end{subequations}

\item[$(n)$] The polar part of $\theta$ near $q_2$ has non-semisimple 
leading-order term, that requires the polynomial 
$\chi_{\vartheta_2}(z_2, w_{2})$ to have one ramified root 
$w_2$ with branch point $z_2=0$. This leads to the formula $p_2^2+4 q_4=0$. 
We use this condition to simplify the roots of $\chi_{\vartheta_2} (z_2, w_{2})$. 
The series of the roots of $\chi_{\vartheta_2} (z_2, w_{2})$ 
up to first order are equal to the series of eigenvalues of 
the matrix (\ref{eq:22_nil_z2}) up to first order. 
The expansions also allow half integer powers of $z_2$ The corresponding 
terms are:
\begin{subequations}
\label{eq:22_nil_z2_eq}
\begin{align}
b_{-4}=&-\frac{p_2}{2}, \\
\sqrt{b_{-3}}=&\frac{1}{2} \sqrt{2 p_1 p_2+4 q_3}, \\
\frac{b_{-2}}{2}=&-\frac{p_1}{2}.
\end{align}
\end{subequations}
\end{itemize}

Now, fix the polar part of $\theta$ near the points $q_1$ and $q_2$. 
The polar part near $q_1$ is independent of the polar part near 
$q_2$.  This means, that the choice between $(S)$ and $(N)$, and 
the choice between $(s)$ and $(n)$ can be done independently. 
Thus we have four possibilities regarding the local behavior of Higgs 
field: $(Ss), (Sn), (Ns), (Nn)$.

Since the above equations do not depend on $q_2$ (the coefficient of
$g_1(z_1)$ and $g_2(z_2)$), we set
\begin{equation}
\label{eq:tq2}
 t = q_2.
\end{equation}
Equations~(\ref{eq:22_ss_z1})--(\ref{eq:22_nil_z2}) have determined
the coefficients of $s_1$ and $s_2$ in any possible choice. The given
complex parameters (i. e. $a_{\pm}$, $\lambda_{\pm}$ etc.) define the
pencil of spectral curves of $(\E,\t)$ parametrized by $t$.

According to the introduction in Section~\ref{sec:ell_penc}, the
pencil $\chi_{{\theta}}$ gives rise to an elliptic fibration in
$\CP2\# 9 {\overline {\CP{}}}^2$ with some singular fibers.  Let us
denote the spectral curves by $\chi_{\vartheta_1}(z_1, w_{1},t)$ in the
$\kappa_1$ trivialization, and by $\chi_{\vartheta_2}(z_2, w_{2},t)$ in the
$\kappa_2$ trivialization.

According to the remark before Subsection~\ref{sec:sing_fib}, no 
curve in the pencil has a singular point on the fiber component of 
the curve $C_{\infty}$ at infinity with multiplicity $2$, thus it 
is sufficient to consider the $\kappa_1$ trivialization, i. e. the 
chart $(z_1, w_{1})$.  For identifying the singular fibers in 
the pencil, we look for triples $(z_1, w_{1},t)$ such that $(z_1, 
w_{1})$ fits the curve with parameter $t$ and the partial derivatives 
below vanish:
\begin{subequations}
\label{eq:22_partials}
\begin{align}
	\label{eq:22_partials_first}
	\chi_{\vartheta_1} (z_1, w_{1}, t) &= 0, \\
	\frac{\partial \chi_{\vartheta_1}(z_1, w_{1}, t)}{\partial w_{1}} &=0, \\
	\frac{\partial \chi_{\vartheta_1}(z_1, w_{1}, t)}{\partial z_1} &=0.
\end{align}
\end{subequations}
These triples are in one-to-one correspondence with singular points in
singular fibers. Every spectral curve $Z_t$ is a double section of the
ruling on the Hirzebruch surface $\mathbb{F}_2$, thus every triple
$(z_1, w_{1},t)$ satisfying Equations~\eqref{eq:22_partials} maps to
distinct points under the ruling~$p$.  Indeed, if one fiber (with
fixed $t$ value) contains two singular points with the same $z_1$
coordinate then the corresponding fiber of $p$ would intersect $Z_t$
with multiplicity higher than two. Furthermore, it does not happen
that two singular points with the same $z_1$ coordinate lies on
distinct fibers (two distinct $t$ values): we will see in
Equations~\eqref{eq:22_ss_t}, \eqref{eq:22_sn_t} and
\eqref{eq:22_nn_t} that the $t$ values are determined by the $z_1$
values.  Consequently the $z_1$-values from triples $(z_1, w_{1},t)$
are in one-to-one correspondence with singular points.

Before the discussion of the individual
cases, we prove a useful lemma about sections.
\begin{lem}
\label{lem:22_section_exist}
Consider the second Hirzebruch surface ${\mathbb {F}}_2$ with a pencil
of type $(2,2)$ having four base points, two in each fiber with
multiplicity two. The pencil contains a section if and only if
$L\pm M = 0$.
\end{lem}
\begin{proof}
If the pencil contains a section, then there is a section
\[
\t \in
\Gamma (\End(\E) \otimes K{_C (2\cdot \{ q_1 \} + 2\cdot \{q_2\})})
\]
whose spectral curve has two components. Denote the components by
$X_+$ and $X_-$.  Each of these components passes through exactly two
base points.  $X_{\pm}$ are locally the graphs of sections, which in
the $\kappa_i$ trivialization (by Equations~\eqref{eq:22_ss_z1} and
\eqref{eq:22_ss_z2}) are
\begin{align*}
	\theta_{\pm,1} &=(a_{\pm}+ \lambda_{\pm}z_1 + \dots) \kappa_1, \\
	\theta_{\pm,2} &=(b_{\pm}+ \mu_{\pm}z_1 + \dots) \kappa_2.
\end{align*}
The conversion $z_1=z_2^{-1}$ and Equation~\eqref{eq:22_triv_conv}
imply
\begin{align*}
	\mu_+&=-\lambda_{\pm}, \\ 
	\mu_-&=-\lambda_{\mp}.
\end{align*}
These equations imply $L\pm M=0$.

Conversely, if $L\pm M=0$ then the above equations hold, and 
the sections $\theta_{\pm}$ satisfy the conditions.
\end{proof}


\subsection{The discussion of cases appearing in Proposition~\ref{thm:main_22_ss}}
\label{sec:22_ss}

We assume that the polar part of the Higgs field is semisimple near 
$q_1$ and semisimple near $q_2$. Namely, the spectral curve $Z_{t}$ 
and the pencil (specified by this spectral curve, together with the 
curve at infinity) are determined by 
Equations~(\ref{eq:22_ss_z1_eq}) and (\ref{eq:22_ss_z2_eq}). The 
base locus of the pencil consists of four points: $(0, a_-)$ and 
$(0,a_+)$ in the chart $(z_1, w_{1})$ and $(0, b_-)$ and $(0,b_+)$ in 
the chart $(z_2, w_{2})$.  According to the list before
Proposition~\ref{prop:I*_sing_fib}, the fibration has a singular fiber
of type $I_2^*$.

Express the coefficients $p_i$ and $q_j$ from equations listed in
cases $(S)$ and $(s)$ above, $i=0,1,2$ and $j=0,1,3,4$.  The
characteristic polynomial (\ref{eq:22_char-poly2}) in the $\kappa_1$ and $\kappa_2$
trivializations becomes
\begin{align}
	\label{eq:22_ss_chi_1}
	\begin{split}
	\chi_{\vartheta_1}(z_1, w_{1},t)=& w_{1}^2 + \left(\left(b_-+b_+\right) z_1^2-
	\left(\lambda _-+\lambda _+\right) z_1 -\left(a_-+a_+\right)\right) w_{1} + b_- b_+ z_1^4 + \\
	&+ \left(b_+ \mu _-+b_- \mu _+\right)z_1^3 -t z_1^2+ \left(a_+ \lambda _-+a_- \lambda _+\right)z_1+a_- a_+, 
	\end{split} \\
	\label{eq:22_ss_chi_2}
	\begin{split}
	\chi_{\vartheta_2}(z_2, w_2,t)=& w_2^2 + \left(\left(a_-+a_+\right) z_2^2  
	+\left(\lambda _-+\lambda _+\right) z_2 - \left(b_-+b_+\right)\right) w_2 +a_- a_+ z_2^4 +\\
	&+\left(a_+ \lambda _-+a_- \lambda _+\right)z_2^3-t z_2^2+\left(b_+ \mu _-+b_- \mu _+\right)z_2+b_- b_+.
	\end{split}
\end{align}

It is enough to analyze the pencil $\chi_{\vartheta_1}(z_1, w_{1},t)$.
Consider Equations~(\ref{eq:22_partials}) to determine the singular points. 
We express $w_{1}$ and $t$ from the second and the third equations by $z_1$.
\begin{align}
	\nonumber
	w_{1} (z_1)=& \frac{1}{2} \left(\left(-b_--b_+\right) z_1^2+\left(\lambda _-+\lambda _+\right) z_1 + a_-+a_+\right) , \\
	\begin{split}
	\label{eq:22_ss_t}
	t (z_1) =& \frac{1}{4 z_1} \left(-2 \left(b_--b_+\right){}^2 z_1^3 
	+3 \left(b_+ \left(\lambda _-+\lambda _++2 \mu _-\right)+b_- \left(\lambda _-+\lambda _++2 \mu _+\right)\right)z_1^2 + \right.\\
	&+ \left.\left(2 \left(a_-+a_+\right) \left(b_-+b_+\right)-\left(\lambda _-+\lambda _+\right){}^2\right)z_1-\left(a_--a_+\right) \left(\lambda _--\lambda _+\right) \right)  .
	\end{split}
\end{align}
Now, we substitute the resulting expressions into the first 
equation and get
\begin{align*}
	\begin{split}
	0 =& \left(b_--b_+\right){}^2 z_1^4- \left(b_+ \left(\lambda _-+\lambda _++2 \mu _-\right)+b_- \left(\lambda _-+\lambda _++2 \mu _+\right)\right)z_1^3 + \\
	&+ \left(a_+-a_-\right) \left(\lambda _--\lambda _+\right) z_1-\left(a_--a_+\right){}^2.
	\end{split}
\end{align*}
We can rewrite the equation with the notation of (\ref{eq:22_notation}) 
and use Condition~(\ref{eq:residuum_22_ss}):
\begin{equation} 
\label{eq:22_ss_quartic}
	0= B^2 z_1^4+B M z_1^3-A L z_1-A^2.
\end{equation}

The roots of this polynomial correspond to the $z_1$ values of
singular points in the singular curves on the Hirzebruch surface
${\mathbb {F}}_2$, which become fibers on the 8-fold blow up.  Since
this is a degree-4 polynomial, generally we get four distinct roots,
and this corresponds to the fact that there are at most four singular
fibers in the fibration.

The quartic polynomial of (\ref{eq:22_ss_quartic}) with variable 
$z_1$ has multiple roots if and only if its discriminant 
\[	
	A^3 B^3 \left(192 A^2 B^2 L M-256 A^3 B^3-3 A B \left(9 L^4-2 L^2 M^2+9 M^4\right)+4 L^3 M^3\right)
\]	
vanishes.

\begin{lem}
\label{lem:22_aabb}
The cases $A=0$ (i. e. $a_-=a_+$) and $B=0$ (i. e. $b_-= b_+$) lead to 
the non-regular semisimple case and does not give an elliptic fibration.
\end{lem}
\begin{proof}
Let us consider the case $A=0$, consider the curves of the pencil
$\chi_{\vartheta_1}(z_1, w_{1},t)$ and substitute $a_-$ with $a_+$ in
the Equation~(\ref{eq:22_ss_chi_1}).  Compute the tangents of
$\chi_{\vartheta_1}(z_1, w_{1},t)$ at $(z_1=0, w_1=a_+)$ as the
implicit derivative of $\chi_{\vartheta_1} (z_1, w_1,t)$ in the point
$(0,a_+)$:
\begin{equation*}
	\frac{\frac{\partial \chi_{\vartheta_1}}{\partial w_1}}
	{\frac{\partial \chi_{\vartheta_1}}{\partial z_1}} = -\frac{2}{\lambda _-+\lambda _+}.
\end{equation*}
Since $\lambda _-+\lambda _+ \neq \infty$, both branches of the curves
intersect the $z_1=0$ axis transversely in the point
$(0,a_+)$. Therefore this is a singular point of all curves in the
pencil, consequently the pencil has no smooth curves, hence the
resulting fibration is not elliptic. (See remark before
Subsection~\ref{sec:sing_fib}.)
\end{proof}

According to Assumption~\ref{assn:elliptic}, we have $A \neq 0$, $B\neq 0$ and we define
\begin{equation}
\label{eq:22_ss_discr}
\Delta =192 A^2 B^2 L M-256 A^3 B^3-3 A B \left(9 L^4-2 L^2 M^2+9
M^4\right)+4 L^3 M^3.
\end{equation}

The further expressions below are connected to the fact whether the
quartic has double, triple or quadruple roots.
\begin{align*}
	\Delta_0 &= 3 A B (L M-4 A B),\\
	\Delta_1 &= B^4 \left(16 A B L M-64 A^2 B^2-3 M^4\right).
\end{align*}

\subsubsection{One root}
The quartic polynomial of (\ref{eq:22_ss_quartic}) has one root 
if and only if $\Delta=\Delta _0=\Delta _1=0$.
Simplify $\Delta_1$ using $\Delta_0=0$ 
to get $M=0$. Setting $M=0$ in $\Delta_0$ provides that 
$A$ or $B$ is equal to zero. 
This is a contradiction, hence this case does not occur. 

\subsubsection{Two roots (one triple root)}
\label{sec:22_ss_tripleroot}
First we consider if $\Delta$ and $\Delta_0$ vanish. 
Apply $\Delta_0=0$ to simplify $\Delta=0$:
\begin{equation*}
	0=-\frac{27}{4} L M \left(L^2-M^2\right)^2.
\end{equation*}
We have $L\neq 0$ because otherwise $\Delta_0=0$ becomes $-12A^2
B^2=0$ and this is not the regular semisimple case.  For the same reason $M \neq
0$, thus $L^2=M^2 \neq 0$.  From $\Delta_0$ we get $B=\frac{LM}{4 A}$.
Substituting $B$ and $L$ to the quartic~(\ref{eq:22_ss_quartic}) and
solving it we get a triple root $\mp\frac{2A}{M}$ and a simple root
$\pm\frac{2A}{M}$.  The plus or minus sign depends on the sign in
$L=\pm M$.  Hence the fibration has two singular fibers.
Proposition~\ref{prop:I*_sing_fib} provides that the fibers are either
$III+I_1$ or $2 II$.  By Lemma~\ref{lem:22_section_exist}, $L=\pm M$
means that the pencil has a section and Lemma~\ref{lem:section_to_fiber}
shows the existence of fiber of type $III$.

In the other direction we suppose the fibration has a type $III$ fiber 
and a fishtail. There are two singular points, i.e. the quartic has two roots. 
Consequently the discriminant vanishes and by Lemma~\ref{lem:ss_section} 
there is a section. This means that $L=\pm M$, or equivalently $L^2=M^2$. 

Moreover $\Delta=0$ and $L=\pm M$ guarantee $L M=4 A B$ and this leads to the 
appearance of a triple root and $L\neq 0$.

\subsubsection{Two roots (two double roots)}
\label{sec:22_ss_two_double}
In the second case, the quartic has two double roots, that is, 
the shape of the quartic equation is
\begin{equation*}
	0=c_1 (z-c_2)^2(z-c_3)^2,
\end{equation*}
where $c_i \in \mathbb{C}$ and $i=1,2,3$.
Expand the equation and denote the coefficients by 
$r_0$, $r_1$, $r_2$, $r_3$, $r _4$ in ascending order. There 
are two relations among the coefficients:
\begin{align*}
64 r_4^3 r_0 &= \left(r_3^2-4 r_2 r_4\right){}^2, \\
r_3^3+8 r_1 r_4^2 &= 4 r_2 r_3 r_4.
\end{align*}
Replace $r_i$ by the coefficients of the
quartic~(\ref{eq:22_ss_quartic}) and get
\begin{align*}
0 &= 64 A^2 B^2+M^4,\\
0 &= 8 A B^2 L-B M^3.
\end{align*}
Solve these equations for $B$ and $L$ with the 
assumptions $A \neq 0$ and $B \neq 0$.
\begin{align*}
B &= \pm \frac{i M^2}{8 A}, \\
L &= \mp i M.
\end{align*}
Since $L\pm M\neq 0$, the pencil has no section.
Equivalently, the 
 quartic has two double roots if and only if
\begin{equation}
	\label{eq:22_ss_double}
	M^3 =8 A B L \quad\mbox{and}\quad	L^2 =-M^2.
\end{equation}
It is easy to see that $\Delta=\Delta_1=0$ and $\Delta_0\neq 0$. 
In particular, if $L=M=0$ then $\Delta=-256 A^3 B^3$,
a contradiction since $A\neq 0$ and $B\neq 0$. Thus $L^2 =-M^2\neq 0$.

Substituting $L = \pm i M$ into $\Delta$, 
it becomes
\begin{equation*}
	-4 \left(8 A B \pm i M^2\right)^2 \left(A B\mp i M^2\right).
\end{equation*}
The condition $\Delta=0$ can be realized in two ways. First is the
above case, when $8 A B \pm i M^2=0$.  In the second case we solve the
quartic~(\ref{eq:22_ss_quartic}) using assumption $A B\mp i M^2=0$,
and we get three distinct roots.  But now we discuss the two roots
case, thus $8 A B\pm i M^2=0$ (consequently $A B\mp i M^2 \neq 0$).
Now Proposition~\ref{prop:I*_sing_fib} and Lemma~\ref{lem:ss_section}
ensure that the fibration has two cusps. The equations $L = \pm i M$
and $8 A B\pm i M^2=0$ lead to $M^3=8 A BL$.

Conversely, if the fibration has two fibers of type $II$ then
$\Delta=0$. There are two possibilities: either the quartic has a triple
root, or it has two double roots. We showed that the case of a triple
root is equivalent to the appearance of a fiber of type $III$. 
Consequently the remaining possibility is the case of two double roots. 
Above we analyzed the case of two double roots, which led to the 
equations $8 A B\pm i M^2=0$ and $L^2 =-M^2\neq 0$.

\subsubsection{Three roots}
The quartic has three distinct roots if the discriminant 
(\ref{eq:22_ss_discr}) vanishes, but $\Delta_0$ and $\Delta_1$ do 
not.  In light of the above results we have two cases.  The first is 
described in Subsection~\ref{sec:22_ss_two_double}, namely $L = \pm 
i M$ and $A B\mp i M^2=0$. Indeed, if we substitute $L = \pm i M$ 
and $A B\mp i M^2=0$ to $\Delta_0$ or $\Delta_1$, these two 
expressions do not vanish, but $\Delta=0$.  The second case comes 
from $L \neq \pm i M$ and $\Delta=0$. There is only one case in 
Proposition~\ref{prop:I*_sing_fib} where the fibration has three 
singular points; in both cases the singular fibers are $II+2I_1$.

In the other direction, the cusp and two fishtails imply that the quartic 
has three distinct roots, hence $\Delta=0$, $\Delta_0 \neq 0$ and 
$\Delta_1 \neq 0$. By the previous results, these are equivalent to 
the above two cases.

\subsubsection{Four roots with section}
\label{sec:22_ss_fourroots}
The quartic has four distinct roots if and only if $\Delta\neq 0$. The three possible cases are 
listed in Proposition~\ref{prop:I*_sing_fib} and we distinguish the cases 
based on the existence of a section and the number of singular fibers.

We compute the number of singular fibers from~(\ref{eq:22_ss_t}). 
We use notation of (\ref{eq:22_notation}):
\begin{equation}
\label{eq:22_ss_t2}
	t= -\frac{1}{4 z_1} \left(2 B^2 z_1^3+3 B M z_1^2+\left(\left(\mu _-+\mu _+\right){}^2-2 \left(a_-+a_+\right) \left(b_-+b_+\right)\right)z_1+A L\right).
\end{equation}
Let us denote the four distinct roots of (\ref{eq:22_ss_quartic}) by
$y_i$ ($i=1,...,4$). Denote the value of $t$ by $t_i$ after the 
substitution of $z_1$ with $y_i$ in Equation~(\ref{eq:22_ss_t2}).
Two roots (say $y_1$ and $y_2$) provide singularities on the same
curve, if and only if $t_1=t_2$.  Equivalently:
\begin{equation}
	\label{eq:22_ss_t1-t2}
	0=t_1-t_2=-\frac{y_1-y_2}{4y_1 y_2} \left(2 B^2 y_2 y_1^2+2 B^2 y_2^2 y_1+3 B M y_2 y_1-A L \right).
\end{equation}
We can simplify with the factor $-\frac{y_1-y_2}{4y_1 y_2}$. 
Similarly, we can express all $(t_i-t_j)$ factor, where $i<j$ and $i, j \in \{1,...,4\}$.
Obviously, the four distinct roots provide three or less 
values for $t$ if and only if
\begin{equation}
	\label{eq:22_ss_teegy}
	T_1:=(t_1-t_2)(t_1-t_3)(t_1-t_4)(t_2-t_3)(t_2-t_4)(t_3-t_4)
\end{equation}
vanishes.  {Plug the simplified $t_i-t_j$ factors (expression in
  equation~\eqref{eq:22_ss_t1-t2} and similar others) to
  \eqref{eq:22_ss_teegy}.  The expression $T_1$ is} a symmetric
polynomial in variables $y_1, y_2, y_3, y_4$, hence can be rewritten
as a polynomial of the elementary symmetric polynomials
$\sigma_1=y_1+y_2+y_3+y_4$,
$\sigma_2=y_1y_2+y_1y_3+y_1y_4+y_2y_3+y_2y_4+y_3y_4$,
$\sigma_3=y_1y_2y_3+y_1y_2y_4+y_1y_3y_4+y_2y_3y_4$ and
$\sigma_4=y_1y_2y_3y_4$.  
We do not reproduce the expression of $T_1$
  in terms of the $\sigma _i$ here, because we prefer to switch to the
  parameters $A, B, M$ and $L$. Indeed, by Vieta's formulas the
symmetric polynomials can also be expressed by the coefficients of the
quartic of~(\ref{eq:22_ss_quartic}):
\begin{align*}
	\sigma_1 = -\frac{4 B}{3 A}, \quad \sigma_2 = \frac{2 A L+B^2}{3 A^2},
\quad	\sigma_3 = 0,\quad 
	\sigma_4 = -\frac{M^2}{3 A^2}.
\end{align*}
Using these formulae, a tedious calculation provides the following 
form of the expression of \eqref{eq:22_ss_teegy}:
\begin{equation*}
\begin{split}
	T_1=&-\frac{2 A^5}{B} (L-M) (L+M) \left(-256 A^3 B^3+192 A^2 B^2 L M-\right. \\
			&\left.-3 A B \left(9 L^4-2 L^2 M^2+9 M^4\right)+4 L^3 M^3\right).
	\end{split}
\end{equation*}
Notice that the discriminant~(\ref{eq:22_ss_discr}) appears in the
expression, hence we get
\begin{equation*}
	T_1=-\frac{2 A^5}{B} (L-M) (L+M) \Delta.
\end{equation*}

Now, we consider another expression, which 
vanishes when at least three $t_i$ values are equal ($i=1,2,3,4$) or $t_i=t_j$ and $t_k=t_l$ for distinct indices ($i,j,k,l \in \{1,2,3,4\}$).
 
\begin{equation*}
	T_2:=\sum_{\substack{i, j=1 \\ i<j}}^4 \frac{1}{t_i-t_j} 
	\prod_{\substack{k, l=1 \\ k<l}}^4 (t_k-t_l).	
\end{equation*}
Expressing $T_2$ in terms of $A,B,L,M$ as we did for $T_1$, we get
\begin{equation*}
\begin{split}
 T_2=& -\frac{A^4}{B} \left(256 A^3 B^3 L+48 A^2 B^2 M \left(3 M^2-7 L^2\right)+ \right.\\
		 &\left.+12 A B L \left(9 L^4-8 L^2 M^2+3 M^4\right)-13 L^4 M^3+9 L^2 M^5\right).
\end{split}
\end{equation*}

Now we concentrate on the cases with section.
If $\Delta\neq 0$, then $T_1=0$ is equivalent to 
$L =\pm M$. This means that 
if the pencil has a section (by Lemma~\ref{lem:22_section_exist}) 
then Lemma~\ref{lem:ss_section} guarantees that a fiber of type $I_2$ occurs, hence the 
possibility of four fishtail fibers is excluded.
Moreover, $T_2$ simplifies to
\begin{equation*}
	T_2=\frac{4 A^4 M}{B}\left(M^2\mp 4 A B\right)^3,
\end{equation*}
where $\mp$ corresponds to the sign in $L =\pm M$.
If $T_1=T_2=0$ we have two cases. 

First, $4AB=\pm M^2$ and $L =\pm M$ lead to $4AB=LM$. The latter means
  $\Delta_0=0$ and in Subsection~\ref{sec:22_ss_tripleroot} we showed
  then that $\Delta$ vanishes, which is excluded from the present
  analysis.

Second, if $M=L=0$ then the discriminant becomes $-256 A^6 B^6$. This
does not vanish, thus we get four singular points in two singular
fibers, which is the $2 I_2$ case.

Finally, if $T_1=0$ but $T_2\neq 0$ then $L=M=0$ is excluded. Moreover if $L=\pm M \neq 0$ 
then $T_2 \neq 0, \Delta \neq 0 \Longleftrightarrow \Delta \neq 0, M^2=L^2\neq 0$ 
and this gives the $I_2+2I_1$ case.

The converse direction is obvious according to Proposition~\ref{prop:I*_sing_fib} 
and Lemma~\ref{lem:section_to_fiber}.

\subsubsection{Four roots without a section}
Now, if the pencil has no section (i.e. $L^2\neq M^2$) then
Lemma~\ref{lem:ss_section} provides that the fibration has no
$I_2$-fiber, hence by Proposition~\ref{prop:I*_sing_fib} the only
possibility is four fishtail fibers. We saw that four distinct $t$
values means $T_1\neq 0$, and then the four singular points
lie in four distinct fibers.

Conversely, if the fibration has $4 I_1$, then the quartic has four roots,
hence $\Delta\neq 0$.  Arguing indirectly, we suppose $L^2=M^2$, then
we do not have four distinct $t$ values, i. e.  four singular
fibers. Hence $L^2 \neq M^2$.

\begin{proof}[Proof of Proposition~\ref{thm:main_22_ss}]
The case-analysis above exhausts all possibilities, 
and hence verifies Proposition~\ref{thm:main_22_ss}.
The result is conveniently summarized in Table~\ref{tab:22_ss}.
\end{proof}


\subsection{The discussion of cases appearing in Proposition~\ref{thm:main_22_sn}}
\label{sec:22_sn}
We assume that the polar part of the Higgs field is semisimple near $q_1$ 
and non-semisimple near $q_2$.  The pencil is given by 
Equations~(\ref{eq:22_ss_z1_eq}) and (\ref{eq:22_nil_z2_eq}). The
base locus of the pencil consists of three points: $(0, a_-)$ and
$(0,a_+)$ in the chart $(z_1, w_{1})$ and $(0, b_{-4})$ in the chart
$(z_2, w_{2})$.  Consequently the fibration has a singular fiber of
type $I_3^*$.  The other possible singular fibers are listed in
Proposition~\ref{prop:I*_sing_fib}.

If the polar part of the Higgs field is non-semisimple near $q_1$ and 
semisimple near $q_2$, we get a very similar case.  We only have to
replace Equations~(\ref{eq:22_ss_z1_eq}) with (\ref{eq:22_ss_z2_eq}) and
Equations~(\ref{eq:22_nil_z2_eq}) with (\ref{eq:22_nil_z1_eq}).  

The notations of polynomials will be the same as 
Subsection~\ref{sec:22_ss} with different values, but this will not
lead to confusion.  We express the coefficients $p_i$ and $q_j$ from
Equations~{(\ref{eq:22_ss_z1_eq}) and (\ref{eq:22_nil_z2_eq})},
$i=0,1,2$ and $j=0,1,3,4$.  (Note that we also use the equation
$p_2^2+4 q_4=0$.)  The characteristic polynomial
(\ref{eq:22_char-poly2}) in both trivializations:
\begin{align}
	\notag
	\begin{split}
	\chi_{\vartheta_1}(z_1, w_{1},t)=& w_{1}^2 + \left(2 b_{-4} z_1^2+b_{-2} z_1-\left(a_-+a_+\right)\right) w_{1} + b_{-4}^2 z_1^4 +\\
	&+ \left(b_{-4} b_{-2}-b_{-3}\right) z_1^3 -t z_1^2+ \left(a_+ \lambda _-+a_- \lambda _+\right)z_1 +a_- a_+, 
	\end{split} \\
	\label{eq:22_sn_chi_2}
	\begin{split}
	\chi_{\vartheta_2}(z_2, w_{2},t)=& w_2^2+ \left(\left(a_-+a_+\right) z_2^2-b_{-2} z_2-2 b_{-4}\right) w_2+a_- a_+ z_2^4 \\ 
	&+\left(a_+ \lambda _-+a_- \lambda _+\right)z_2^3 -t z_2^2 +\left(b_{-4} b_{-2}-b_{-3}\right) z_2+b_{-4}^2.
	\end{split}
\end{align}

Similarly, we consider the partial derivatives (\ref{eq:22_partials}) of $\chi_{\vartheta_1}(z_1, w_{1},t)$ 
to determine the singular points. We express $w_{1}$ and $t$ from the 
second and the third equations and simplify using 
Condition~(\ref{eq:residuum_22_sn}):
\begin{align} 
	\notag
	w_{1} (z_1)=& \frac{1}{2} \left(-2 b_{-4} z_1^2-b_{-2} z_1+a_-+a_+\right), \\
	\label{eq:22_sn_t}
	t (z_1) =& \frac{1}{4 z_1} \left( -6 b_{-3} z_1^2 
	+\left(4 a_- b_{-4}+4 a_+ b_{-4}-b_{-2}^2\right)z_1 - (a_- -a_+)(\lambda_-\lambda_+) \right).
\end{align}
Substitute these into the Equation~\eqref{eq:22_partials_first} and get
\begin{equation*}
	0 = 2 b_{-3} z_1^3 + \left(\left(a_-+a_+\right) b_{-2}+2 a_+ \lambda _-
	+2 a_- \lambda _+\right)z_1 -\left(a_--a_+\right){}^2.
\end{equation*}
Rewrite the equation with the notation in (\ref{eq:22_notation}) 
and use Condition~(\ref{eq:residuum_22_sn}) to get
\begin{equation} 
\label{eq:22_sn_cubic}
	0= 2 b_{-3} z_1^3-A L z_1-A^2.
\end{equation}

This is a cubic polynomial, generally it has three distinct
roots, and this corresponds to the fact that there are at most three
singular fibers in the fibration.

The cubic polynomial of (\ref{eq:22_sn_cubic}) with variable 
$z_1$ has multiple roots if and only if its discriminant vanishes. 
Consider the discriminant 
\begin{equation*}
	\Delta= 4 A^3 b_{-3} \left(2 L^3-27 A b_{-3}\right).
\end{equation*}
Assumption~\ref{assn:elliptic} and a similar statement as Lemma~\ref{lem:22_aabb} implies $A\neq 0$.
The expression $\Delta_0 = 6 A b_{-3} L$ is related 
to the fact whether the cubic polynomial has double or triple roots.

We note that 
if we use the matrices (\ref{eq:22_nil_z1}) and (\ref{eq:22_ss_z2}) 
then the discriminant will be 
\begin{equation*}
4 B^3 a_{-3} \left(2 M^3-27 B a_{-3}\right).
\end{equation*}

\subsubsection{One root}
\label{sec:22_sn_oneroot}
The cubic equation of~(\ref{eq:22_sn_cubic}) has one root if
  and only if $\Delta=\Delta_0=0$, which is equivalent to $b_{-3}=0$
  in our case. The cubic of~(\ref{eq:22_sn_cubic}) therefore reduces to a
  linear equation
\begin{equation*}
	0=-A L z_1-A^2.
\end{equation*}

If $b_{-3}=0$ then the pencil $\chi_{\vartheta_2} (z_2, w_2,t)$ in
Equation~\eqref{eq:22_sn_chi_2} becomes the same as the pencil
$\chi_{\vartheta_2} (z_2, w_2,t)$ in the semisimple case in
Subsection~\ref{sec:22_ss} with the assumption $B=0$. Indeed, choose
the parameters $b_-=b_+=b_{-4}$, $\mu_+=0$ and $\mu_-=b_{-2}$ in
Equation~\eqref{eq:22_ss_chi_2}. Using
condition~\eqref{eq:residuum_22_sn} we get
Equation~\eqref{eq:22_sn_chi_2} with $b_{-3}=0$.  Consequently
Lemma~\ref{lem:22_aabb} applies and shows that the pencil has no
smooth curves and the resulting fibration is not elliptic.  Notice
that the tangents of the curves of the pencil
of~(\ref{eq:22_sn_chi_2}) at $(0,b_{-2})$ are $-\frac{2}{b_{-2}}$.

\subsubsection{Two roots}
The cubic has two roots iff $\Delta=0$ and $\Delta_0 \neq 0$. The latter 
inequality does not give constraints, because $A\neq 0$, $b_{-3} \neq 0$ 
and if $L=0$ this leads to $b_{-3}=0$ that is a contradiction. Thus 
$\Delta =0$ provides
\begin{equation*}
	2 L^3=27 A b_{-3}.
\end{equation*}
Compute the cubic's roots with this restriction,
arriving to two distinct roots: 
a single root $\frac{3A}{L}$ and a double root $-\frac{3A}{2L}$. 
According to
Proposition~\ref{prop:I*_sing_fib}, the fibration has a cusp and a
fishtail fibers.

Conversely, if the fibration has a fiber of type $II$ and an $I_1$,
then the cubic must have two roots. This happens exactly when
$\Delta=0$.

\subsubsection{Three roots}
Finally, if $\Delta\neq 0$ the cubic has three distinct roots and it
gives rise to three singular fibers which are 
(by Proposition~\ref{prop:I*_sing_fib}) all fishtails.  
The converse direction is also trivial.

\begin{proof}[Proof of Proposition~\ref{thm:main_22_sn}]
We conclude that the singular fibers next to the type $I_3^*$ fiber
are the following
\begin{itemize}
\item three $I_1$-fibers iff $\Delta\neq 0$, and 
\item a type $II$ fiber and an $I_1$-fiber iff $\Delta=0$,
\end{itemize}
hence we verified Proposition~\ref{thm:main_22_sn}.
\end{proof}


\subsection{The discussion of cases appearing in Proposition~\ref{thm:main_22_nn}}
\label{sec:22_nn}
The polar part of the Higgs field is non-semisimple near both $q_i$
  ($i=1,2$). Equations~(\ref{eq:22_nil_z1_eq}) and
  (\ref{eq:22_nil_z2_eq}) determine the spectral curve and the
  pencil. The base locus of the pencil consists of two points: $(0,
a_{-4})$ in the chart $(z_1, w_{1})$ and $(0, b_{-4})$ in the chart
$(z_2, w_{2})$.  The fiber at infinity is of type $I_4^*$.

Again, express the coefficients $p_i$ and $q_j$ from 
{equations~(\ref{eq:22_nil_z1_eq}) and (\ref{eq:22_nil_z2_eq})}, $i=0,1,2$ and $j=0,1,3,4$.
The characteristic polynomial (\ref{eq:22_char-poly2}) has
the following shape:
\begin{align*}
	\begin{split}
	\chi_{\vartheta_1}(z_1, w_{1},t)=& w_{1}^2 + \left(2 b_{-4} z_1^2-a_{-2} z_1-2 a_{-4}\right) w_{1} +b_{-4}^2 z_1^4 +\\
	& +\left(b_{-4} b_{-2}-b_{-3}\right) z_1^3-t z_1^2+\left(a_{-4} a_{-2}-a_{-3}\right) z_1+a_{-4}^2.	
	\end{split}
\end{align*}
Using Condition~(\ref{eq:residuum_22_nn}), we 
solve the second and third equations from (\ref{eq:22_partials}) 
to determine the singular points: 
\begin{align} 
	\notag
	w_{1} (z_1)=& \frac{1}{2} \left(-2 b_{-4} z_1^2+a_{-2} z_1+2 a_{-4}\right), \\
	\label{eq:22_nn_t}
	t (z_1) =& \frac{1}{4 z_1} \left(-6 b_{-3} z_1^2 +\left(8 a_{-4} b_{-4}-a_{-2}^2\right)z_1-2 a_{-3} \right).
\end{align}
Now, substitute $w_1$ and $t$ into Equation~(\ref{eq:22_partials_first}) and get
\begin{equation}
	\label{eq:22_nn_cubic}
	0 =b_{-3} z_1^3-a_{-3} z_1.
\end{equation}

Generally, a cubic polynomial has three distinct
roots. One of the roots of (\ref{eq:22_nn_cubic}) 
is $z_1=0$, but due to the remark before Subsection~\ref{sec:sing_fib} 
the fiber component of the curve at infinity with multiplicity $2$ 
has no singular point of any other curve in the pencil,
hence the cubic equation can be reduced to following 
quadric:
\begin{equation}
	\label{eq:22_nn_quadric}
	0 =b_{-3} z_1^2-a_{-3}.
\end{equation}

The quadric polynomial has a double root if and only if its
discriminant $\Delta=4 a_{-3} b_{-3}$ vanishes.

\subsubsection{One root}
The cubic has one root in two ways.

First, if $b_{-3}=0$ and $a_{-3}\neq 0$, then the
cubic~(\ref{eq:22_nn_cubic}) reduces to $0=a_{-3} z_1$. The singular
point lies on the line $z_1=0$, but this pencil does not give an
elliptic fibration. If $a_{-3}=b_{-3}=0$, then from
Equation~\eqref{eq:22_nn_quadric} we see that the pencil contains
curves only of higher multiplicity, hence it does not contain a smooth
curve.

If $b_{-3}\neq 0$ then the cubic~(\ref{eq:22_nn_cubic}) has one root if and only if $a_{-3}=0$, but this leads $z_1=0$ again.

\subsubsection{Two roots}
If the discriminant of the quadric of (\ref{eq:22_nn_quadric}) is nonzero, then 
the quadric always has two distinct roots which are never zero. In other words the fibration has 
two singular fibers which are fishtails due to Proposition~\ref{prop:I*_sing_fib}.

\begin{proof}[Proof of Proposition~\ref{thm:main_22_nn}]
The above discussion shows that no other possibility in the case of
two base points can arise, thus we proved
Proposition~\ref{thm:main_22_nn}.
\end{proof}


\section{The order of poles are $3$ and $1$}\label{sec:PII_PIV}
\label{sec:31}
This section contains the cases of Theorems~\ref{thm:PIV}, 
\ref{thm:PIV_degenerate}, \ref{thm:PII} and \ref{thm:PII_degenerate} 
without the parabolic weights. 
More precisely we will prove the following propositions.
(Recall that $X$ is given in Notation~\ref{nota:X}.)

\begin{prop}\label{thm:main_31_ss}
   Assume that the polar part of the Higgs field is regular semisimple near
   $q_1$ and regular semisimple near $q_2$. Then ${X}$ is biregular to the
   complement of the fiber at infinity (of type $\Et6$) in an elliptic
   fibration of the rational elliptic surface such that the set of
   other singular fibers of the fibration is:
  \begin{enumerate}
		\item if $L=\pm M$ and  $B^2=\pm 4AM$ (implying $\Delta = 0$), then a type $III$ and an $I_1$ fibers; 
		\item if $L=\pm M$ and $B^2=\mp 12AM$ (implying $\Delta = 0$), then a type $II$ and an $I_2$ fibers;
		\item if $L=\pm M$, $B^2 \neq \pm 4AM$ and $B^2 \neq \mp 12AM$ (and so $\Delta \neq 0$), then a type $I_2$ and two $I_1$ fibers;
		\item if $L \neq \pm M$, $\Delta = 0$ and $B=0$, then two type $II$ fibers;
		\item if $L \neq \pm M$,  $\Delta = 0$ and $B\neq 0$, then a type $II$ and two $I_1$ fibers; 
		\item if $L \neq \pm M$ and $\Delta \neq 0$, then four type $I_1$ fibers.
  \end{enumerate}
	where for $\Delta$ see (\ref{eq:31_ss_discr}).
\end{prop}

Once again, this result can be conveniently summarized in
Table~\ref{tab:31_ss}.
\begin{table}[htb]
\begin{center}
\begin{tabular}{c*{4}{>{\centering\arraybackslash}p{.17\linewidth}}}
	\multicolumn{1}{c|}{\multirow{2}{*}{}}&\multicolumn{2}{c|}{\multirow{2}{*}{$\mathbf{L =\pm M}$}}&\multicolumn{2}{c}{\multirow{2}{*}{$\mathbf{L \neq \pm M}$}}\\
	\hhline{~~~~~} \multicolumn{1}{c|}{} & \multicolumn{2}{c|}{} &\multicolumn{2}{c}{} 	\\
	\hhline{-#=|=|=|=}
	\multicolumn{1}{c||}{\multirow{2}{*}{$\mathbf{\Delta=0}$}}	& \multicolumn{1}{c|}{$\mathbf{B^2=\pm 4AM}$} & \multicolumn{1}{c|}{$III+I_1$} & \multicolumn{1}{c|}{$\mathbf{B = 0}$}& $2 II$		\\
	\hhline{~||----} 
	\multicolumn{1}{c||}{}  & \multicolumn{1}{c|}{$\mathbf{B^2=\mp 12AM}$} & \multicolumn{1}{c|}{$II+I_2$} & \multicolumn{1}{c|}{$\mathbf{B \neq 0}$}	& \multicolumn{1}{c}{$II+2I_1$} \\
	\hline
	\multicolumn{1}{c||}{\multirow{2}{*}{$\mathbf{\Delta\neq0}$}}&\multicolumn{2}{c|}{\multirow{2}{*}{$I_2+2I_1$}}&\multicolumn{2}{c}{\multirow{2}{*}{$4 I_1$}}\\
	\multicolumn{1}{c||}{}  & \multicolumn{2}{c|}{} &\multicolumn{2}{c}{} 
\end{tabular}
\end{center}
\caption{\quad 
The type of singular curves in (3,1) case with four base points. 
In this case the fiber at infinity is an ${\tilde {E}}_6$ fiber.}
\label{tab:31_ss}
\end{table}

\begin{prop}\label{thm:main_31_sn}
   Assume that the polar part of the Higgs field is regular semisimple near $q_1$ and non-semisimple near $q_2$. 
   Then ${X}$ is biregular to the complement of the fiber at infinity (of type $\Et6$) in an 
   elliptic fibration of the rational elliptic surface such that the set of other singular fibers of the fibration is:
  \begin{enumerate}
		\item if  $\Delta = 0$ and $L=0$, then a type $IV$ fiber; 
		\item if $\Delta = 0$ and $L\neq 0$, then a type $II$ and an $I_2$ fibers; 
		\item if $\Delta \neq 0$ and $L=0$, then a type $I_3$ and an $I_1$ fibers; 
		\item if $\Delta \neq 0$, $L\neq 0$ and $B^2=-2AL$, then a type $III$ and an $I_1$ fibers;
		\item if $\Delta \neq 0$, $L\neq 0$ and $B^2\neq-2AL$, then a type $I_2$ and two $I_1$ fibers; 
  \end{enumerate}
	where $\Delta = 4 A^2 \left(B^2-6 A L\right)$.
\end{prop}

The table summarizing this case has the following shape:
\begin{table}[htb]
\begin{center}
\begin{tabular}{c*{5}{>{\centering\arraybackslash}p{.17\linewidth}}}
	\multicolumn{1}{c|}{\multirow{2}{*}{}} & \multicolumn{2}{c|}{\multirow{2}{*}{$\mathbf{L =0}$}} & \multicolumn{2}{c}{\multirow{2}{*}{$\mathbf{L \neq 0}$}} \\
	\hhline{~~~|~~} 
	\multicolumn{1}{c|}{} & \multicolumn{2}{c|}{} &\multicolumn{2}{c}{} \\
	\hhline {-#==|==}
	\multicolumn{1}{c||}{\multirow{2}{*}{$\mathbf{\Delta=0}$}} & \multicolumn{2}{c|}{\multirow{2}{*}{$IV$}} & \multicolumn{2}{c}{\multirow{2}{*}{$II+I_2$}} \\
	\hhline{~~~|~~} 
	\multicolumn{1}{c||}{}  & \multicolumn{2}{c|}{\multirow{1}{*}{}} & \multicolumn{2}{c}{\multirow{1}{*}{}}\\
	\hline
	\multicolumn{1}{c||}{\multirow{2}{*}{$\mathbf{\Delta\neq0}$}}& \multicolumn{2}{c|}{\multirow{2}{*}{$I_3+I_1$}} & \multicolumn{1}{c|}{$\mathbf{B^2=-2AL}$}& $III+I_1$	\\
	\hhline{~||~~|-|-}
	\multicolumn{1}{c||}{} & \multicolumn{2}{c|}{} & \multicolumn{1}{c|}{$\mathbf{B^2\neq -2AL}$}	& \multicolumn{1}{c}{$I_2+2 I_1$}
\end{tabular}
\end{center}
\caption{\quad The type of singular curves in (3,1) case with three
  base points. In this case the fiber at infinity is an ${\tilde {E}}_6$ fiber.}
\label{tab:31_sn}
\end{table}

\begin{prop}\label{thm:main_31_ns}
   Assume that the polar part of the Higgs field is non-semisimple near $q_1$ and regular semisimple near $q_2$. 
   Then ${X}$ is biregular to the complement of the fiber at infinity (of type $\Et7$) in an 
   elliptic fibration of the rational elliptic surface such that the set of other singular fibers of the fibration is:
  \begin{enumerate}
		\item if $\Delta = 0$ and $Q \neq 0$, then a type $II$ and an $I_1$ fibers; 
		\item if $\Delta \neq 0$ and $Q \neq 0$, then three type $I_1$ fibers; 
  \end{enumerate}
	where $\Delta = M^2 \left(27 M^2 Q^2-4 R^3\right)$. 
	If $Q=0$ then the fibration is not elliptic.
\end{prop}

\begin{prop}\label{thm:main_31_nn}
   Assume that the polar part of the Higgs field is non-semisimple near $q_1$ and non-semisimple near $q_2$. 
   Then ${X}$ is biregular to the complement of the fiber at infinity (of type $\Et7$) in an 
   elliptic fibration of the rational elliptic surface such that the set of other singular fibers of the fibration is:
  \begin{enumerate}
		\item if $R = 0$ and $Q \neq 0$, then a type $III$ fiber; 
		\item if $R \neq 0$ and $Q \neq 0$, then a type $I_2$ and an $I_1$ fibers. 
  \end{enumerate} 
	If $Q=0$ then the fibration is not elliptic.
\end{prop}

\begin{remark}
Again, we will prove implications in converse directions as well.
\end{remark}


\subsection{Local form of irregular Higgs bundles}
\label{subsec:31_local_forms}
In this section we give the local description of the spectral curves 
on the Hirzebruch surface ${\mathbb {F}}_2$ in the 
case~\eqref{eq:D=3+1}. Let $U_i$ be suitable affine open charts 
near $q_i$ on $C=\CP{1}$. The line bundle $K_C (3\cdot \{ q_1 \} 
+ \{q_2\})$ has the trivializations $\kappa_i$ ($i=1,2$).
\begin{align*}
  \kappa_1 &= \frac{\d z_1}{z_1^3}, \\
	\kappa_2 &= \frac{\d z_2}{z_2}.
\end{align*}
The conversion from $\kappa_1$ to $\kappa_2$ is the following:
\begin{equation}
	\label{eq:31_triv_conv}
	\kappa_1 = \frac{\mathrm{d}z_1}{z_1^{3}}=-z_2 \d z_2 = -z_2^2 \kappa_2.
\end{equation}
The trivialization $\kappa_i$ induces a trivialization $\kappa_i^2$ on 
$K_C(3\cdot \{ q_1 \} + \{q_2\})^{\otimes 2}$, $i=1,2$.

We proceed as in Section~\ref{sec:22_local_form}.  Consider the two
charts $V_i$ on Hirzebruch surface which are the complements of the
section at infinity in preimages of $U_i$ under $p$ ($i=1,2$) (and
disjoint from the section at infinity).  Let $\zeta \in \Gamma
\left(\mathbb{F}_2,p^* K_C(D)\right)$ be the canonical section, and
introduce $w_{i} \in \Gamma (V_i,\O)$ by
\begin{equation*}
	\zeta = w_{i} \otimes \kappa_i.
\end{equation*}
The conversion between $w_1$ to $w_2$ is $w_2=-z_2^2 w_1$ again.

Consider an irregular Higgs
bundle $(\E, \theta)$ in some trivializations; the local forms of
$\theta$ near $q_1$ or $q_2$ are the following
\begin{equation}
	\label{eq:31_theta} 
	\theta = \sum_{n\geq -3} A_n z_1^n  \otimes \d z_1, \mbox{~or~} \quad \theta = \sum_{n\geq -1} B_n z_2^n \otimes \d z_2,
\end{equation}
where $A_n, B_n \in \gl (2,\C )$.

Recall the definition of the characteristic polynomial $\chi_{\theta} (\zeta )$ 
from \eqref{eq:characteristic_polynomial}. Let us set
$\vartheta_1 = \sum_{n\geq 0} A_{n-3} z_1^n$ or $\vartheta_2 =
\sum_{n\geq 0} B_{n-1} z_2^n$ so that we have
\begin{equation*}
	\theta  = \vartheta_i \otimes \kappa_i.
\end{equation*}

If we divide by $\kappa_i$ in \eqref{eq:characteristic_polynomial}
then the characteristic polynomial may be rewritten as 
\begin{equation}
	\label{eq:31_char-poly2}
	\chi_{\vartheta_i} (w_i ) = \det (w_i \mbox{I}_{\E} - \vartheta_i) = w_i^2 + w_i f_i + g_i, 
\end{equation}
with locally defined functions $f_i, g_i$ such that
\begin{equation*}
  {s_1} = f_i \kappa_i, \quad {s_2} = g_i \kappa_i^2.
\end{equation*}
Now, as $K(3\cdot \{ q_1 \} + \{q_2\}) \cong \O(2)$, the coefficients
$f_1$ and $g_1$ in the $\kappa_1$ trivialization
are polynomials in $z_1$ of degree $2$ and $4$, respectively:
\begin{align*}
	f_1(z_1) &= - (p_2 z_1^2 + p_1 z_1 + p_0), \\
	g_1(z_1) &= - (q_4 z_1^4 + q_3 z_1^3 + q_2 z_1^2 + q_1 z_1 + q_0),
\end{align*}
where all coefficients are elements of $\C$. 

Similarly, in the $\kappa_2$ trivialization 
we get
\begin{align*}
	f_2(z_2) &= p_0 z_2^2+p_1 z_2+p_2,\\
	g_2(z_2) &= -\left( q_0 z_2^4+q_1 z_2^3+q_2 z_2^2+q_3 z_2+q_4 \right).
\end{align*}

Finally, from (\ref{eq:31_char-poly2}) we get two polynomials in
variables $z_1, w_1$ or $z_2, w_2$. These polynomials are the local
  forms of the spectral curves in $Z_C (D)$:
\begin{align}
	\label{eq:31_char-poly2_z1}
	\chi_{\vartheta_1} (z_1, w_1) & = w_1^2 - \left(p_2 z_1^2+p_1 z_1+p_0\right) w_1 -\left(q_4 z_1^4+q_3 z_1^3+q_2 z_1^2+q_1 z_1+q_0\right), \\
	\label{eq:31_char-poly2_z2}
	\chi_{\vartheta_2}  (z_2, w_{2}) & = w_2^2 + \left(p_0 z_2^2+p_1 z_2+p_2\right) w_2 - (q_0 z_2^4+q_1 z_2^3+q_2 z_2^2+q_3 z_2+q_4).
\end{align}

The spectral curve has an expansion near $q_1$ and $q_2$ in which
these parameters have a geometric meaning.  The lowest index matrices
$A_n$ and $B_n$ in (\ref{eq:31_theta}) encode the base locus of a
pencil. The matrices $A_{-2}$ and $A_{-1}$ encode the tangent and the
second derivative of the curve of a pencil.  As indicated in
Subsection~\ref{subsec:result_painleve_IV} the letters $(S), (N), (s),
(n)$ refer to the following cases which may occur independently of
each other.

Near $z_1 = 0$:
\begin{itemize}
  \item[$(S)$] the semisimple case
   \begin{equation} \label{eq:31_ss3}
   \t = \left[ \begin{pmatrix}
         a_+ & 0 \\
         0 & a_-
        \end{pmatrix} z_1^{-3}  
				+
				\begin{pmatrix}
         b_+ & 0 \\
         0 & b_- 
        \end{pmatrix} z_1^{-2} 
				+
				\begin{pmatrix}
        \lambda_+ & 0 \\
         0 & \lambda_- 
        \end{pmatrix} z_1^{-1} 
				+  O(1)
   \right] \otimes \d z_1,
   \end{equation}
  \item[$(N)$] the non-semisimple case 
   \begin{equation} \label{eq:31_nil3}
   \t = \left( \begin{pmatrix}
             b_{-6} & 1 \\
              0	& b_{-6}
            \end{pmatrix} z_1^{-3} 
            +
         \begin{pmatrix}
             0 & 0 \\
              b_{-5} & b_{-4}
            \end{pmatrix} z_1^{-2}
               +
         \begin{pmatrix}
             0 & 0 \\
              b_{-3} & b_{-2}
            \end{pmatrix} z_1^{-1} 
            + O(1)
            \right) \otimes \d z_1.
   \end{equation}
\end{itemize}

Near $z_2 = 0$: 
\begin{itemize}
  \item[$(s)$] the semisimple case
   \begin{equation} \label{eq:31_ss1}
   \t = \left[ \begin{pmatrix}
         \mu_+ & 0 \\
         0 & \mu_- 
        \end{pmatrix} z_2^{-1}  +  O(1) 
   \right] \otimes \d z_2,
   \end{equation}
  \item[$(n)$] the non-semisimple case 
   \begin{equation}\label{eq:31_nil1}
   \t = \left( \begin{pmatrix}
                b_{-1} & 1 \\
              0	& b_{-1}
            \end{pmatrix} z_2^{-1} 
            + O(1)
            \right) \otimes \d z_2.
   \end{equation}
\end{itemize}

We again observe 
\begin{equation*}
  \tr \Res_{z_1=0} \t + \tr \Res_{z_1=\infty} \t = 0,
\end{equation*}
which, according to the various cases, implies
\begin{subequations}
 \begin{align}
	\label{eq:residuum_31_ss}
  \lambda_+ + \lambda_- + \mu_+ + \mu_- & = 0, \\
	\label{eq:residuum_31_sn}
  \lambda_+ + \lambda_- + 2 b_{-1} & = 0, \\
	\label{eq:residuum_31_ns}
  b_{-2} + \mu_+ + \mu_- & = 0, \\ 
	\label{eq:residuum_31_nn}
  b_{-2} + 2 b_{-1} & = 0 .
 \end{align}
\end{subequations}

The roots of the characteristic polynomial (\ref{eq:31_char-poly2}) 
in $w_{i}$ have Puiseux expansions with respect to $z_i$. 
The first several terms of these expansions are equal to the 
eigenvalues of the matrices (\ref{eq:31_ss3}) -- (\ref{eq:31_nil1}).
In concrete terms:

\begin{itemize}
\item[$(S)$] The polar part of $\theta$ near $q_1$ has semisimple 
leading-order term. The power series of the 'negative' 
root of $\chi_{\vartheta_1} (z_1, w_{1})$ in (\ref{eq:31_char-poly2_z1}) 
up to second order is equal to $a_- + b_-z_1 + \lambda_- z_1^2$
and the 'positive' root up to second order is equal to 
$a_+ + b_+ z_1 + \lambda_+ z_1^2$. We get the following equations:
\begin{subequations}
\label{eq:31_ss_z1_eq}
\begin{align}
a_-=&\frac{1}{2} \left(p_0-\sqrt{p_0^2+4 q_0}\right), \\
b_-=&\frac{1}{2} \left(p_1-\frac{p_0 p_1+2 q_1}{\sqrt{p_0^2+4 q_0}}\right), \\
\lambda _-=&\frac{1}{2} \left(p_2-\frac{2 p_0^2 q_2+p_0 \left(4 p_2 q_0-2 p_1 q_1\right)+2 p_1^2 q_0+p_2 p_0^3-2 q_1^2+8 q_0 q_2}{\left(p_0^2+4 q_0\right){}^{3/2}}\right), \\
a_+=&\frac{1}{2} \left(\sqrt{p_0^2+4 q_0}+p_0\right), \\
b_+=&\frac{1}{2} \left(\frac{p_0 p_1+2 q_1}{\sqrt{p_0^2+4 q_0}}+p_1\right),\\
\lambda _+ =& \frac{1}{2}\left(\frac{2 p_0^2 q_2+p_0 \left(4 p_2 q_0-2 p_1 q_1\right)+2 p_1^2 q_0+p_2 p_0^3-2 q_1^2+8 q_0 q_2}{\left(p_0^2+4 q_0\right){}^{3/2}}+p_2\right).
\end{align}
\end{subequations}

\item[$(N)$] The polar part of $\theta$ near $q_1$ has non-semisimple
leading-order term. This means that the polynomial 
$\chi_{\vartheta_1} (z_1, w_{1})$ has one ramified root $w_1$ with branch point $z_1=0$.
This leads to the formula $p_0^2+4 q_0=0$. Simplify the roots of $\chi_{\vartheta_1} (z_1, w_{1})$ 
using this condition. The series of two roots up to second order 
are equal to the series of eigenvalues of the matrix 
(\ref{eq:31_nil3}) up to second order. Notice that the expansions 
allow half integer powers of the variable $z_1$. The corresponding terms are:
\begin{subequations}
\label{eq:31_nil_z1_eq}
\begin{align}
b_{-6}=&\frac{p_0}{2}, \\
\sqrt{b_{-5}}=&\sqrt{\frac{p_0 p_1}{2}+q_1}, \\
\frac{b_{-4}}{2}=&\frac{p_1}{2}, \\
\frac{b_{-4}^2+4 b_{-3}}{8 \sqrt{b_{-5}}}=&\frac{p_1^2+2 p_0 p_2+4 q_2}{4 \sqrt{2 p_0 p_1+4 q_1}}, \\
\frac{b_{-2}}{2}=&\frac{p_2}{2}.
\end{align}
\end{subequations}

\item [$(s)$] The polar part of $\theta$ near $q_2$ has semisimple 
leading-order term. The zero order terms of the series of the 
roots of $\chi_{\vartheta_2} (z_2, w_{2})$ in (\ref{eq:31_char-poly2_z2})
are equal to $\mu_-$ and $\mu_+$:
\begin{subequations}
\label{eq:31_ss_z2_eq}
\begin{align}
\mu _-=&-\frac{1}{2} \sqrt{p_2^2+4 q_4}-\frac{p_2}{2},\\
\mu _+=&\frac{1}{2} \sqrt{p_2^2+4 q_4}-\frac{p_2}{2}.
\end{align}
\end{subequations}

\item[$(n)$] The polar part of $\theta$ near $q_2$ has non-semisimple 
leading-order term, 
implying that the polynomial $\chi_{\vartheta_2} (z_2, w_{2})$ 
has one ramified root $w_2$ with branch point $z_2=0$. 
This leads to the formula $p_2^2+4 q_4=0$. 
Use this equation to simplify the roots of $\chi_{\vartheta_2} (z_2, w_{2})$.
The zero order terms of the series of the roots of 
$\chi_{\vartheta_2} (z_2, w_{2})$ are equal to the eigenvalues 
of the matrix in (\ref{eq:31_nil1}), that is
\begin{equation}
\label{eq:31_nil_z2_eq}
	b_{-1}=-\frac{p_2}{2}.
\end{equation}
\end{itemize}

As in Section~\ref{sec:22}, fix the polar part of $\theta$. Again, 
the choices of the semisimple or the non-semisimple cases near $q_1$ 
and near $q_2$ 
are independent. We have four possibility for fixing the polar part, 
namely: $(Ss), (Sn), (Ns), (Nn)$.

It turns out that
Equations~(\ref{eq:31_ss_z1_eq})--(\ref{eq:31_nil_z2_eq}) do not
depend on $q_3$ (the coefficient of $g_1(z_1)$ and $g_2(z_2)$); thus
we set
\begin{equation}
	\label{eq:tq3}
	t=q_3.
\end{equation}
The above equations corresponding to the pairs of letters determine
the coefficients of $s_1$ and $s_2$ in any possibly choice. The given
complex parameters (i. e. $a_{\pm}$, $b_{\pm}$ etc.) define the pencil
of spectral curve of $(\E,\t)$ parametrized by $t$, namely, the
characteristic polynomial $\chi_{\t}(\zeta)$.

According to the introduction in Section~\ref{sec:ell_penc}, the 
pencil $\chi_{{\theta}}$ gives rise to an elliptic fibration on 
$\CP2\# 9 {\overline {\CP{}}}^2$ with some singular fibers. Let us 
denote the spectral curves by $\chi_{\vartheta_1}(z_1, w_{1},t)$ in 
the $\kappa_1$ trivialization, and by $\chi_{\vartheta_2}(z_2, 
w_{2},t)$ in the $\kappa_2$ trivialization.

According to the remark before Subsection~\ref{sec:sing_fib}, the
fiber component of the curve at infinity with multiplicity $3$
intersects every curve in the pencil in its smooth points. Therefore it 
suffices to consider the $\kappa_2$ trivialization, i. e. the chart 
$(z_2, w_{2})$. In identifying the singular fibers in the pencil,
we look for triples $(z_2, w_{2},t)$ such that $(z_2, w_{2})$ fits the
curve with parameter $t$, and the partial derivatives below vanish:
\begin{subequations}
\label{eq:31_partials}
\begin{align}
	\label{eq:31_partials_first}
	\chi_{\vartheta_2} (z_2, w_{2}, t) &= 0, \\
	\frac{\partial \chi_{\vartheta_2}(z_2, w_{2}, t)}{\partial w_{2}} &=0, \\
	\frac{\partial \chi_{\vartheta_2}(z_2, w_{2}, t)}{\partial z_2} &=0.
\end{align}
\end{subequations}
If no triple $(z_2, w_{2}, t)$ lies on the fiber component of the
curve at infinity with multiplicity~$1$ then the triples are in
one-to-one correspondence with the singular points in the singular
fibers (for the same reasons after Equations~\eqref{eq:22_partials} in
Subsection~\ref{sec:22_local_form}).  If such a triple lies on the
fiber component of the curve at infinity with multiplicity~$1$ then we
will have to blow up this point and we will have to compute some
singular points on a new chart.

Before the discussion of the cases, we consider a lemma about sections.
\begin{lem}
\label{lem:31_section_exist}
Consider the second Hirzebruch surface ${\mathbb {F}}_2$ with a pencil
of type $(3,1)$ having four or three base points, two in the fiber with 
multiplicity three and two or one in the other fiber. 
The pencil contains a section if and only if
\begin{itemize}
	\item $L\pm M = 0$ in the case of four base points,
	\item $L = 0$ in the case of three base points.
\end{itemize}
\end{lem}

\begin{proof}
The proof is very similar to the proof of
Lemma~\ref{lem:22_section_exist}. The two sections in the $\kappa_1$
trivialization are
\begin{equation*}
	\theta_{\pm,1}=(a_{\pm}+ b_{\pm}z_1 + \lambda_{\pm}z_1^2 + \dots) \kappa_1.
\end{equation*}
If the pencil has four base points, the sections in the $\kappa_2$
trivialization are: $\theta_{\pm,2} =(\mu_{\pm} + \dots) \kappa_2$.
If the pencil has three base points, the sections in the $\kappa_2$
trivialization are: $\theta_{\pm,2} =(b_{-1} + \dots) \kappa_2$.  Four
base points give the same two equations as in
Lemma~\ref{lem:22_section_exist}. Three base points give the
following:
\begin{align*}
	b_{-1}&=-\lambda_+, \\ 
	b_{-1}&=-\lambda_-.
\end{align*}
Equation~(\ref{eq:residuum_31_sn}) ensures that both rows are
fulfilled simultaneously and give $L=0$. 

The converse is also true, namely, if $L=0$ then above two equations
hold and the sections $\theta_{\pm}$ satisfy these two conditions.
\end{proof}


\subsection{The discussion of cases appearing in Proposition~\ref{thm:main_31_ss}}
\label{sec:31_ss}
The polar part of the Higgs field is semisimple near both $q_i$ 
$(i=1,2)$. Namely, the spectral curve $Z_{t}$ and the pencil 
(specified by this spectral curve, together with the 
curve at infinity) are determined by Equations~(\ref{eq:31_ss_z1_eq}) and 
(\ref{eq:31_ss_z2_eq}). The base locus of the pencil consists of 
four points: $(0, a_-)$ and $(0,a_+)$ in the chart $(z_1, w_{1})$ (on 
the fiber with multiplicity three) and $(0, \mu_-)$ and $(0,\mu_+)$ in 
the chart $(z_2, w_{2})$. 
The fibration has a singular fiber of type $\Et6$.

Express the coefficients $p_i$ and $q_j$ from the equations listed in 
cases $(S)$ and $(s)$, $i=0,1,2$ and $j=0,1,2,4$.
The characteristic polynomial (\ref{eq:31_char-poly2}) in the $\kappa_2$ 
trivialization becomes 
\begin{align*}
	\begin{split}
	\chi_{\vartheta_2}(z_2, w_{2},t)=& w_{2}^2 + \left(\left(a_-+a_+\right) z_2^2+
	\left(b_-+b_+\right) z_2 + \lambda_-+\lambda_+ \right) w_{2} + a_- a_+ z_2^4 +\\
	&+ \left(a_+ b_-+a_- b_+\right)z_2^3 +\left(a_+ \lambda_- +
	a_- \lambda_+ +b_- b_+\right) z_2^2 -t z_2+\mu_- \mu_+.
	\end{split}
\end{align*}

Recall that if the polar part is semisimple near both poles then the 
fiber component of the curve at infinity with multiplicity $1$ has 
no singular point. Hence the equations of (\ref{eq:31_partials}) 
determine the singular points.  Express $w_{2}$ and $t$ from the 
second and the third equations by $z_2$.
\begin{align} 
	w_{2} (z_2)=& -\frac{1}{2} \left( \left(a_-+a_+\right) z_2^2+\left(b_-+b_+\right) z_2+\lambda _-+\lambda _+ \right), \nonumber \\
	\begin{split}
	\label{eq:31_ss_t} 
	t (z_2) =&  - \frac{1}{2} \left(2 \left(a_--a_+\right){}^2 z_2^3+3 \left(a_--a_+\right) \left(b_--b_+\right) z_2^2+ \right.\\
	&+ \left.\left( 2 \left(a_--a_+\right) \left(\lambda _--\lambda _+\right)+\left(b_--b_+\right){}^2 \right) z_2 
	+\left(b_-+b_+\right) \left(\lambda _-+\lambda _+\right) \right).
	\end{split}
\end{align}
We will need the $t$ value later, so we simplify (\ref{eq:31_ss_t}) using the notation of (\ref{eq:31_notation}):
\begin{equation}
	\label{eq:31_ss_t2}
	t (z_2)= -A^2 z_2^3 -\frac{3}{2} A B z_2^2 - \frac{1}{2} \left(2 A L+B^2\right) z_2
	-\frac{1}{2} \left(b_-+b_+\right) \left(\lambda _-+\lambda _+\right).
\end{equation}

Now, substitute the resulting expressions into the Equation~\eqref{eq:31_partials_first} and get
\begin{align*}
	\begin{split}
	0 =& 3 \left(a_--a_+\right){}^2 z_2^4 + 4 \left(a_--a_+\right) \left(b_--b_+\right) z_2^3 + \\
	&+ \left(2 \left(a_--a_+\right) \left(\lambda _--\lambda _+\right)+\left(b_--b_+\right){}^2\right) z_2^2 
	-\left (\lambda _-+\lambda _+\right)^2 + 4 \mu _- \mu _+.
	\end{split}
\end{align*}
Rewrite the equation with the notation in (\ref{eq:31_notation}) 
with Condition~(\ref{eq:residuum_31_ss}):
\begin{equation} 
\label{eq:31_ss_quartic}
	0= 3A^2 z_2^4 + 4AB z_2^3 + (2AL+B^2)z_2^2 -M^2.
\end{equation}

The roots of the quartic give the $z_1$ values of the corresponding singular points in the singular curves. 
The roots are in one-to-one correspondence with singular points in the semisimple case. 
Since (\ref{eq:31_ss_quartic}) is a quartic polynomial, generally 
it has four distinct roots, and this corresponds to the fact that 
there are at most four singular fibers in the fibration.

The quartic polynomial of (\ref{eq:31_ss_quartic}) with variable 
$z_2$ has multiple roots if and only if its discriminant 
\begin{equation*}
	-16 A^2 M^2  \cdot \left(48 A^4 \left(L^2+3 M^2\right)^2+64 A^3 B^2 L \left(L^2-9 M^2\right)+24 A^2 B^4 \left(L^2+3 M^2\right)-B^8\right)
\end{equation*}
vanishes. 

\begin{remark}
\label{rem:31_aamumu}
A similar calculation as in Lemma~\ref{lem:22_aabb} gives 
that the cases $a_-=a_+$ and $\mu_-= \mu_+$ lead to the non-regular 
semisimple case and do not give elliptic fibrations.
\end{remark}

Since $A \neq 0$ and $M\neq 0$ by Assumption~\ref{assn:elliptic}, 
we can simplify the above expression by $-16 A^2 M^2$
and set
\begin{equation}
\label{eq:31_ss_discr}
\Delta = 48 A^4 \left(L^2+3 M^2\right)^2+64 A^3 B^2 L \left(L^2-9 M^2\right)+24 A^2 B^4 \left(L^2+3 M^2\right)-B^8.
\end{equation}

We will see that the further expressions below are connected to the
fact whether the quartic has double, triple or quadruple roots.
\begin{align*}
	\Delta_0 &= \left(2 A L+B^2\right)^2-36 A^2 M^2,\\
	\Delta_1 &= -48 A^3 \left(12 A^3 \left(L^2+3 M^2\right)-20 A^2 B^2 L-13 A B^4+4 B^3\right).
\end{align*}

\subsubsection{One root}
The fibration has one root in the chart $(z_2, w_{2})$ 
if and only if $\Delta=\Delta_0=\Delta_1=0$. 
If we solve these three equations for any three 
variables (e. g. $B, L, M$ or $A, B, L$) then we
get $A=0$ or $M=0$ or we do not get solution.
Hence one root in the semisimple case is not possible.

\subsubsection{Two roots (one triple root)}
\label{sec:31_ss_triple}
The pencil has two singular points and 
according to Proposition~\ref{prop:E6-E7_sing_fib} there 
are two possible cases, namely type $III+I_1$ and type $2 II$.

On the other hand, the quartic has two roots in two different cases.

In the first case, the quartic has a triple root. 
This is equivalent to $\Delta=\Delta_0=0$ and 
$\Delta_1\neq 0$. 
Start with $\Delta_0=0$:
\begin{equation*}
\left(2 A L+B^2\right)^2 =36 A^2 M^2.
\end{equation*}
Express $B^2$, then $\Delta$ may be rewritten as
\begin{align*}
	B^2=&-2 A L \pm 6 A M, \\
	\Delta=&1728 A^4 M^2 (L \mp M)^2.
\end{align*}
Since $A\neq 0 \neq M$ but $\Delta=0$ we get $L= \pm M$.  This implies
$B^2=\pm 4AM$.

Substitute $B$ and $L$ into $\Delta_1$ with the assumption $\Delta_0=0$, 
and get $\Delta_1=16A^2 M^2$. This implies that  $\Delta_1$ never vanishes. 

Since $L\pm M=0$, Lemma~\ref{lem:31_section_exist} implies that the
pencil contains a section. 
Due to Lemma~\ref{lem:section_to_fiber} only a fiber of type $III$ can appear in the fibration.

For the converse
direction, we suppose the pencil has a 
singular fibers of type $III$ and an $I_1$ fiber. Obviously 
the pencil has two singular points and $\Delta=0$.
Lemma~\ref{lem:ss_section} provides that the pencil 
has a section, consequently $L=\pm M$. The discriminant becomes 
$-\left(B^2 \mp 4 A L\right)^3 \left(12 A L \pm B^2\right)$, which 
vanishes if and only if $B^2 =\mp 12 A L$ or $B^2 =\pm 4 A L$.
In the first case substitute 
$B^2$ and $L=\pm M$ to $\Delta_0$ and $\Delta_1$. Since none of these
vanish, the fibration has three distinct roots and that is a contradiction. 
The second case verifies the part of the proposition. 

\subsubsection{Two roots (two double roots)}
\label{sec:31_ss_2double}
In the second case, the quartic has two double roots, 
hence the shape of the quartic polynomial is
\begin{equation}\label{eq:BegSubSec}
	0=c_1 (z-c_2)^2(z-c_3)^2.
\end{equation}
Expand the equation and denote the coefficients by 
$r_0, r_1, r_2, r_3,r _4$ in ascending order. There 
are two relations among the coefficients:
\begin{align*}
64 r_4^3 r_0 &= \left(r_3^2-4 r_2 r_4\right){}^2, \\
r_3^3+8 r_1 r_4^2 &= 4 r_2 r_3 r_4.
\end{align*}
Replace the $r_i$'s by the coefficients of the quartic~\ref{eq:31_ss_quartic} and get
\begin{align*}
0 &= 108 A^2 M^2+\left(B^2-6 A L\right)^2,\\
0 &= 6 A^2 B L-A B^3.
\end{align*}
Solve these equations for $B$ and $L$ with the 
assumptions $A \neq 0$ and $M \neq 0$, and get
\begin{align*}
B &= 0, \\
L &=\pm i \sqrt{3}M.
\end{align*}
We see that the pencil has no section (due to $L\pm M\neq 0$). 
Proposition~\ref{prop:E6-E7_sing_fib} and 
Lemma~\ref{lem:ss_section} provide that the fibration has two cusps. 

In the converse direction, if the fibration has two fibers of type 
$II$ then the quartic has two double roots because triple root leads 
to a fiber of type $III$ as shown above. The fact that the quartic 
has two double roots implies Equation~\eqref{eq:BegSubSec}, implying 
$B=0$ and $L=\pm i \sqrt{3}M$. The conditions $B = 0$ and $L =\pm 
i \sqrt{3}M$ are equivalent to $B=0$ and $\Delta=0$, the condition 
listed in the proposition.

\subsubsection{Three roots with section}
The quartic has three roots if and only if $\Delta=0$ but 
$\Delta_0$ and $\Delta_1$ do not vanish. 
According to Proposition~\ref{prop:E6-E7_sing_fib} we have two cases 
again, which can be distinguished by the existence of a section. 

First, we suppose the pencil has a section, i. e. $L=\pm M$. 
The discriminant becomes 
\begin{equation}
	\label{eq:31_ss_discr_red}
	0=-\left(B^2 \mp 4AM\right)^3 \left(B^2 \pm 12AM\right).
\end{equation}
We already met the $B^2 =\pm 4AM$ case, hence the one possible case is
$B^2=\mp 12AM$. Above we computed $\Delta_0$ and $\Delta_1$ in this
case and they do not vanish.
The appearance of an $I_2$ fiber is the consequence of 
Lemma~\ref{lem:section_to_fiber}. Proposition~\ref{prop:E6-E7_sing_fib} 
provides that the fibration contains an additional cusp fiber.

In the reverse direction $\Delta=0$ and $L=\pm M$ follows immediately. 
We excluded the $B^2 =\pm 4AM$ case, thus the former equality leads to $B^2=\mp 12AM$.

\subsubsection{Three roots without section}
Again $\Delta=0$, $\Delta_0 \neq 0$, $\Delta_1 \neq 0$ but we 
suppose the pencil has no section. We will be using the 
process of elimination. 

In Subsection~\ref{sec:31_ss_2double} we saw that the
$\Delta=0$, $\Delta_0 \neq 0$, $\Delta_1=0$ case is equivalent 
to $\Delta=0$, $B=0$ and $L \neq \pm M$.
In Subsection~\ref{sec:31_ss_triple} we deduced that
$\Delta=0$, $\Delta_0 = 0$, $\Delta_1 \neq 0$ implies the existence 
of a section.
Hence if $L \neq \pm M$ then 
$\Delta=0$, $\Delta_0 \neq 0$, $\Delta_1 \neq 0$ is 
equivalent to $\Delta=0$, $B \neq 0$.
Proposition~\ref{prop:E6-E7_sing_fib} 
and Lemma~\ref{lem:ss_section} guarantee that
the pencil has a cusp and two fishtail fibers next to the $\Et6$ fiber.

Conversely, if the pencil has a $II$ and two $ I_1$ fibers, then 
$\Delta=0$ and the pencil contains no section. 
Since $B=0$ leads to the case of two cusps, we conclude that $B\neq 0$.

\subsubsection{Four roots}
The fibration has four singular points if and only if 
$\Delta \neq 0$. The characterization of singular 
fibers in Proposition~\ref{prop:E6-E7_sing_fib} give three possible 
cases which differ by the number of singular curves in the pencil:
$4 I_1, I_2+2I_1, I_3+I_1$.

Let us denote the four singular distinct roots of 
(\ref{eq:31_ss_quartic}) by $y_i$ ($i=1,...,4$). Denote $t$ by $t_i$ 
after the substitution of $z_2$ with $y_i$ in Equation 
(\ref{eq:31_ss_t2}). Two roots (say $y_1$ and $y_2$) provide 
singularities on the same curve {if and only if} $t_1=t_2$. 
Equivalently:
\begin{equation}
	\label{eq:31_ss_t1-t2}
	\begin{split}
	0=& t_1-t_2= \\
	 =& \left(y_1-y_2\right) \left( 2 A^2 y_1^2+2 A^2 y_2^2+2 A^2 y_1 y_2
	+3 A B y_1+3 A B y_2+2 A L+B^2 \right).
	\end{split}
\end{equation}
We can simplify with $\left(y_1-y_2\right)$. 
Similarly, we can express all $(t_i-t_j)$ factor, where $i<j$ and $i, j \in \{1,...,4\}$. 
Obviously, the four distinct roots provide three or less 
values for $t$ if and only if 
\begin{equation}
	\label{eq:31_ss_teegy}
	T_1:=(t_1-t_2)(t_1-t_3)(t_1-t_4)(t_2-t_3)(t_2-t_4)(t_3-t_4)
\end{equation}
vanishes.
Plugging the simplified $t_i-t_j$ factors (expression in 
Equation~\eqref{eq:31_ss_t1-t2} and similar others) to \eqref{eq:31_ss_teegy}. 
The expression $T_1$ became a symmetric polynomial in variables $y_1, y_2, y_3, y_4$, 
hence can be written as a polynomial of the elementary symmetric polynomials 
as in case $(2,2)$ in Subsection~\ref{sec:22_ss_fourroots}. 

The vanishing condition $T_1=0$ would yield a long expression, but $\sigma_1,
\sigma_2, \sigma_3, \sigma_4$ can be determined from the coefficients of 
polynomial~(\ref{eq:31_ss_quartic}) by Vieta's formulas. The relations 
between the symmetric polynomials and the coefficients are the following: 
\begin{align*}
	\sigma_1 =& -\frac{4 B}{3 A},\\
	\sigma_2 =& \frac{2 A L+B^2}{3 A^2},\\
	\sigma_3 =& 0,\\
	\sigma_4 =& -\frac{M^2}{3 A^2}.
\end{align*}
This transforms Equation~\eqref{eq:31_ss_teegy} to
\begin{equation*}
\begin{split}
	T_1=&\frac{4}{729} A^2 (L-M) (L+M) \left(48 A^4 \left(L^2+3
        M^2\right)^2+ \right.\\ &\left.+64 A^3 B^2 L \left(L^2-9
        M^2\right)+24 A^2 B^4 \left(L^2+3
        M^2\right)-B^8\right)=\\ &\left. =\frac{4}{729} A^2 (L-M)
        (L+M) \Delta \right.
	\end{split}
\end{equation*}

Consider a similar expression
\begin{equation*}
	T_2:=\sum_{\substack{i, j=1 \\ i<j}}^4 \frac{1}{t_i-t_j} 
	\prod_{\substack{k, l=1 \\ k<l}}^4 (t_k-t_l).	
\end{equation*}
This expression vanishes exactly when three $t_i$ values are 
equal ($i=1,...4$) or $t_i=t_j$ and $t_k=t_l$ for distinct indices 
($i,j,k,l \in \{1,2,3,4\}$). 
Executing on $T_2$ the same simplification as on $T_1$, we get 
\begin{equation*}
\begin{split}
 T_2=& \frac{8}{729} \left(192 A^5 \left(L^5+3 L^3 M^2\right)+16 A^4 B^2 \left(13 L^4-72 L^2 M^2+27 M^4\right)+\right.\\
&\left. +12 A^3 B^4 L \left(5 L^2+3 M^2\right)-A B^8 L\right).
\end{split}
\end{equation*}

It is easy to conclude that $T_1 \neq 0$ and $\Delta\neq 0$ is equivalent to 
$L \neq\pm M$. 
By Lemma~\ref{lem:31_section_exist} the 
latter condition  means that the pencil has no section. 
By Lemma~\ref{lem:ss_section} the fiber of type $I_2$ or $I_3$ 
is excluded, hence the only possibility is four fishtail fibers. 

Consider the case when the pencil has a section (that is, $L=\pm M$). 
$T_1=0$ immediately follows, and 
\begin{equation*}
	T_2=\pm\frac{8}{729} A M \left(4 A M \mp B^2\right)^3 \left(12 A M \pm B^2\right).
\end{equation*}
The condition $T_1=0$ (and $\Delta\neq 0$) ensures that the pencil has no four 
distinct singular fibers, hence the $4I_1$ case is not possible. On the 
other hand, if $T_1=T_2=0$ we exactly get back the $III+I_1$ and $II+I_2$ 
cases,  therefore the  $I_3$ fiber cannot appear in the 
fibration. 
$T_2 \neq 0$ (and $L=\pm M$ holds) is equivalent to
$B^2\neq\pm 4AM$ and $B^2\neq\mp 12AM$, and the latter is equivalent to 
$\Delta \neq 0$ (see Equation~(\ref{eq:31_ss_discr_red})). $T_1=0$ 
gives rise to three distinct $t$ values.  This means the fibration has an 
$I_2$ fiber and two $I_1$'s.

By Proposition~\ref{prop:E6-E7_sing_fib} 
and Lemma~\ref{lem:section_to_fiber} the converse is also obvious.

\begin{proof}[Proof of Proposition~\ref{thm:main_31_ss}]
It is not hard to see that the cases discussed in the subsection above 
exhaust all possibilities, and hence verify 
Proposition~\ref{thm:main_31_ss}. For the summary of the results, see 
Table~\ref{tab:31_ss}.
\end{proof}


\subsection{The discussion of cases appearing in Proposition~\ref{thm:main_31_sn}}
\label{sec:31_sn}
The polar part of the Higgs field is semisimple near $q_1$ and 
non-semisimple near $q_2$. Equations~(\ref{eq:31_ss_z1_eq}) and 
(\ref{eq:31_nil_z2_eq}) determine the spectral curve and the 
pencil. The base locus consists of three points: $(0,a_-)$ and 
$(0,a_+)$ in the chart $(z_1, w_{1})$ and $(0,b_{-1})$ in the chart 
$(z_2, w_{2})$.  Consequently, the fibration has a singular fiber of 
type $\Et6$.

Again we express $p_i$ and $q_j$ from Equations~(\ref{eq:31_ss_z1_eq}) and 
(\ref{eq:31_nil_z2_eq}) ($i=0,1,2$ and $j=0,1,2,4$).
The characteristic polynomial (\ref{eq:31_char-poly2}) in the $\kappa_2$ 
trivialization become the following:
\begin{align}
	\begin{split} 
	\label{eq:31_sn_chi}
	\chi_{\vartheta_2}(z_2, w_{2},t)=& w_{2}^2 + \left(\left(a_-+a_+\right)
	z_2^2+\left(b_-+b_+\right) z_2+\lambda _-+\lambda _+\right) w_{2} + a_- a_+ z_2^4 +\\
	& + \left(a_+ b_-+a_- b_+\right)z_2^3 + \left(a_+ \lambda _-
	+a_- \lambda _++b_- b_+\right)z_2^2 -t z_2+b_{-1}^2.
	\end{split}
\end{align}

As before, we identify the partial derivatives (\ref{eq:31_partials}) 
to determine the singular points. Not necessarily all of 
them, because some singular points can come from the blow-up. We 
express $w_{2}$ and $t$ from the second and the third equations and 
simplify using Condition~(\ref{eq:residuum_31_sn}): 
\begin{align*}
	w_{2} (z_2)=& -\frac{1}{2} \left( \left(a_-+a_+\right) z_2^2+\left(b_-+b_+\right) z_2+\lambda _-+\lambda _+ \right), \\
	\begin{split}
	t (z_2) =& -\frac{1}{2} \left( 2 \left(a_--a_+\right){}^2 z_2^3+3 \left(a_--a_+\right) \left(b_--b_+\right) z_2^2+ \right. \\
	&+\left.\left(2 \left(a_--a_+\right) \left(\lambda _--\lambda _+\right)+\left(b_--b_+\right){}^2\right)z_2+\left(b_-+b_+\right) \left(\lambda _-+\lambda _+\right)\right).
	\end{split}
\end{align*}
We will need the $t$ value later, so we simplify that using the notation of (\ref{eq:31_notation}):
\begin{equation}
	\label{eq:31_sn_t}
	t(z_2) = -\frac{1}{2} \left( 2A^3 z_2^3 + 3AB z_2^2 + \left(2 A L+B^2\right)z_2 + \left(b_-+b_+\right) \left(\lambda _-+\lambda _+\right)\right).
\end{equation}
Now, substitute $w_2$ and $t$ into Equation~(\ref{eq:31_partials_first}) and get
\begin{equation*}
	\begin{split}
	0 =& 3 \left(a_--a_+\right){}^2 z_2^4+4 \left(a_--a_+\right) \left(b_--b_+\right) z_2^3+ \\
	&+ \left(2 \left(a_--a_+\right) \left(\lambda _--\lambda _+\right)+\left(b_--b_+\right){}^2\right)z_2^2 +4 b_{-1}^2- 
	\left(\lambda _- + \lambda _+\right)^2.
	\end{split}
\end{equation*}
Reformulate the equation with the notation of (\ref{eq:31_notation}) 
and use Condition~(\ref{eq:residuum_31_sn}):
\begin{equation} 
	\label{eq:31_sn_quartic}
	0= 3 A^2 z_2^4+4 A B z_2^3+ \left(2 A L+B^2\right)z_2^2.
\end{equation}

A quartic polynomial generally has four distinct roots, but now 
(\ref{eq:31_sn_quartic}) has two $z_2=0$ roots, thus it has at most 
three distinct roots. Moreover, by Lemma~\ref{lem:sphere_in_fiber} 
we will  blow up the point $z_2=0, w_2=b_{-1}$  
and compute singular points in a new chart. The number of singular points may 
eventually be four.

The quartic Equation~(\ref{eq:31_sn_quartic}) can be reduced by $z_2^2$ 
and we get a quadric polynomial: 
\begin{equation} 
	\label{eq:31_sn_quadric}
	0= 3 A^2 z_2^2+4 A B z_2+ \left(2 A L+B^2\right).
\end{equation}

The discriminant of the quadric is 
\begin{equation}
	\label{eq:31_sn_discr}
	\Delta= 4 A^2 \left(B^2-6 A L\right),
\end{equation}
where $A\neq 0$ by Assumption~\ref{assn:elliptic} (see
Remark~\ref{rem:31_aamumu}).

\subsubsection{Blow-up on the chart $(z_2, w_{2})$}
\label{sec:31_sn_blowup}
The root $z_2=0$  of the quartic 
creates the possibility of existence of 
some possible singular points. 
According to Lemma~\ref{lem:sphere_in_fiber}, the fibration may 
contain a fiber with a component which appears in the blow-up only. 
For this reason we would like to compute the blow-up of 
$\chi_{\vartheta_2}$ in  Equation~(\ref{eq:31_sn_chi}). 
(See the lower diagram in the Figure~\ref{fig:blowup31B}.)
The appropriate blow-up procedure is the following. 
First we blow up the $(z_2=0, w_{2}=b_{-1})$ point in the chart $(z_2, w_{2})$. 
The exceptional divisor is
\begin{equation*}
	E_1 = \left\{z_2=0, w_2=b_{-1}, \left[\alpha:\beta\right]\right\}.
\end{equation*}
Choose the local chart given by $\alpha \neq 0$;  the variables 
in $\chi_{\vartheta_2}$ should be replaced as follows: $z_2=\alpha$, 
$w_2=\alpha \beta$. 
Next we blow-up the origin $(\alpha=0, \beta=0)$. 
The exceptional divisor is
\begin{equation*}
	E_2 = \left\{\alpha=0, \beta=0, \left[u:v\right]\right\}.
\end{equation*}
Choose the local chart given by $u \neq 0$; the variables 
in $\chi_{\vartheta_2}$ should be replaced as follows: $\alpha=u v$, 
$\beta=u$. The $u$ axis will be part of the possible singular 
fiber.
Let us denote by $\chi_{b}$ the pencil $\chi_{\vartheta_2}$ in the chart 
$u \neq 0$. Reduce the expression by Condition~(\ref{eq:residuum_31_sn}), 
and divide by the exceptional divisor (which is $u^2 v$), and get
\begin{equation*}
	\begin{split}
	\chi_{b}(u,v,t) =& a_- a_+ u^6 v^3+\left(a_-+a_+\right) u^3 v^2+\left(a_+ b_-+a_- b_+\right)u^4 v^2 + \\
	&+ \left(\left(a_--a_+\right)	b_{-1}+a_- \lambda _+-a_+ \lambda _++b_- b_+\right)u^2 v + 
	\left(b_-+b_+\right) u v+v+\\
	&+\left(b_-+b_+\right) b_{-1}-t.
	\end{split}
\end{equation*}

Let us compute the singular points of this family of curves on the chart 
$(u, v)$. As before, we search the triples $(u, v, t)$ such that the
point $(u, v)$ lies on the curve with parameter $t$, and the partial
derivatives vanish (as in Equation~(\ref{eq:31_partials})). The second and
third equations do not contain the variable $t$, thus we solve them in the
variable $u$ and $v$, and we get four roots (hence four is the maximal
number of singular points in this case). Two of these lie on the
$u$ axis:
\begin{equation}
	\label{eq:31_sn_blowup}
	u=-\frac{2}{(b_-+b_+) \pm \sqrt{2 A L+B^2}}, v=0.
\end{equation}

The corresponding $t$ values are 
\begin{equation}
	\label{eq:31_sn_t_inblowup}
	\left(b_-+b_+\right) b_{-1}
\end{equation}
in both cases.  This means that the possible two singular points on
the $u$ axis lie on the same curve.

\subsubsection{One root}
\label{sec:31_sn_oneroot}
Since the point $(z_2=0, w_2=b_{-1})$ is always the preimage of a 
singular point in ${\mathbb {F}}_2$, the necessary condition of the 
appearance of the fiber of type $IV$ is that there is no other 
singular point in the chart $(z_2, w_{2})$. This requires that the 
quadric~(\ref{eq:31_sn_quadric}) has one $z_2=0$ root, thus $\Delta=0$ and the 
linear term vanishes. These are equivalent to $L=B=0$, and also 
equivalent to $\Delta=L=0$. 

For the other direction we need the blow-up of a point $(0,b_{-1})$ in the chart 
$(z_2, w_{2})$ with the condition $L=B=0$. Previously we computed 
the singular points in the blow-up. By substituting $L=B=0$ to 
(\ref{eq:31_sn_blowup}), it turns out that the pencil has one singular 
point: $u=-\frac{1}{b_+}, v=0$. 
Proposition~\ref{prop:E6-E7_sing_fib} implies that the fibration contains a 
singular fiber of type $IV$.

\subsubsection{Two roots (one triple root)}
\label{subsub:TripleRoots}
If the quartic has one root in the chart $(z_2, w_{2})$ then there are no two
singular points in the blowing up. Thus one of the roots is not on
$z_2=0$. There are two possible cases.

The first case is when $(0,b_{-1})$ with any $t$ values is a singular point with
multiplicity three, i. e. $z_2=0$ is a triple root of the quartic
(\ref{eq:31_sn_quartic}). 
The second case will be the case of two double roots in the next subsection.
In the case of triple root, the quadric
(\ref{eq:31_sn_quadric}) has no constant term (i. e. $2AL+B^2=0$) and
its discriminant does not vanish. The quadric and the
discriminant~(\ref{eq:31_sn_discr}) become 
\begin{align}
	\label{eq:31_sn_quadric_2}
	0 = &A z_2 \left(3 A z_2+4 B\right), \\
	\notag
	0 \neq &-32 A^3 L.
\end{align}
Obviously $L=0$ leads to the case of one root, hence we can assume now 
that $L\neq 0$.

We need to determine how many singular points come from the 
point $(0,b_{-1})$ during the blow-up.  Instead of computing on a 
chart in the blow-up, we consider the tangents of the curves of the 
pencil of \eqref{eq:31_sn_chi} at the point $(0,b_{-1})$. There are 
three cases 

\begin{itemize}
\item If the tangent, as function $z_2(w_2)$, is zero 
(i. e. the tangent consists of the $w_2$ axis), 
then the curve is a smooth curve in the pencil or a smooth point of a singular fiber. 
\item If the computation leads two different tangents for one curve, then 
the curve has a fishtail singularity at the point $(0,b_{-1})$.
\item If the tangent is unique and it is not zero, then the curve has a cusp or 
higher order singularity at the point $(0,b_{-1})$.
\end{itemize}

We saw in Subsection~\ref{sec:31_sn_blowup} that the $z_2=0$ root of 
the quartic of~\eqref{eq:31_sn_quartic} leads to the singular points 
on the $u$-axis on the chart $(u,v)$. Now, $z_2=0$ is a triple root 
and it means that the $u$-axis has at most three distinct singular 
points.  Suppose that the $u$-axis has exactly three distinct singular 
points. Since the $u$-axis is part of a singular fiber (see 
Lemma~\ref{lem:sphere_in_fiber}), we get a fiber with three singular 
points.  The only possible case is an $I_3$ fiber. In this case, 
according to Lemma~\ref{lem:31_sn_section} the pencil contains a section. 
Then Lemma~\ref{lem:31_section_exist} guarantees $L=0$, contradicting 
our assumption $L\neq 0$.

Hence the $u$-axis has at most two singular points. 
Let us compute the tangents of the curves of the pencil of \eqref{eq:31_sn_chi}.
In Subsection~\ref{sec:31_sn_blowup} it was shown that the singular 
points in the $u$-axis fit on the same fiber with $t$ value given 
by~\eqref{eq:31_sn_t_inblowup}. Substitute 
Equation~\eqref{eq:31_sn_t_inblowup} to the equation of the pencil 
$\chi_{\vartheta_2}$ and use condition 
\eqref{eq:residuum_31_sn}. Denote the curve by $Z_{\mathrm{sing}}$ in 
the $\kappa_2$ trivialization:
\begin{align*}
	Z_{\mathrm{sing}} (z_2,w_2)=&w_2^2 + \left(\left(a_-+a_+\right) z_2^2+\left(b_-+b_+\right) z_2+\lambda _-+\lambda _+\right) w_2 +a_- a_+ z_2^4+\\
	& +\left(a_+ b_-+a_- b_+\right)z_2^3+ \left(a_+ \lambda _-+a_- \lambda _++b_- b_+\right)z_2^2 +\\
	&+\frac{1}{2} \left(b_-+b_+\right) \left(\lambda_-+\lambda_+\right)z_2+\frac{1}{4} \left(\lambda _-+\lambda _+\right).
\end{align*}
Solve the equation $Z_{\mathrm{sing}}=0$ in variable $w_2$
\begin{align*}
w_2^{(1,2)}=&\frac{1}{2} \left(\left(a_-+a_+\right) z_2^2 -\left(b_-+b_+\right)z_2 -\lambda _- -\lambda _+ \pm \right.\\
	& \left. \pm \sqrt{\left(\left(a_--a_+\right) z_2+b_--b_+\right){}^2+2 \left(a_--a_+\right) \left(\lambda _--\lambda _+\right)} z_2  \right).
\end{align*}
Taking the derivatives of $w_2^{(1,2)}$ with respect to the variable
$z_2$ at $z_2=0$ and using the notation of (\ref{eq:31_notation}) we get:
\begin{equation*}
-\frac{1}{2}{(b_-+b_+) \pm \sqrt{2 A L+B^2}}.
\end{equation*}

If $2 A L+B^2=0$ (when the quartic \eqref{eq:31_sn_quartic} has a
triple root) then the two derivatives are equal and the tangent of the
curve $Z_{\mathrm{sing}}$ is unique. It leads to the case when the
curve $Z_{\mathrm{sing}}$ has a cusp singularity at the point
$(0,b_{-1})$ and the pencil has an unique singular point on the
$u$-axis in the blow-up. We note that the slope of the tangent of the 
  function $z_2(w_2)$ is not zero, because $b_-+b_+\neq \infty$.

In conclusion, the pencil has altogether two singular points: one of 
them comes from the blow-up of the point $(0,b_{-1})$, the other one 
comes from the root $z_2=-\frac{4 B}{3 A}$ of the 
quadric~\eqref{eq:31_sn_quadric_2}. 
By Propositions~\ref{prop:E6-E7_sing_fib} and 
Lemma~\ref{lem:sphere_in_fiber} the only possible case is that the 
fibration has a type $III$ and a type $I_1$ singular fibers.

Conversely, we suppose that the fibration has a type $III$ and an
$I_1$ fibers.  Due to Subsection~\ref{sec:31_sn_oneroot} two singular
points come from two roots of the quartic. As at the beginning of
Subsection~\ref{subsub:TripleRoots}, we have two cases but only the
first case gives two singular points, thus $B^2+2AL=0$ and $L\neq 0$.

\subsubsection{Two roots (two double roots)}
If the quartic has a double root in a point $z_2 \neq 0$, then the 
discriminant (\ref{eq:31_sn_discr}) of the quadric 
vanishes:
\begin{equation*}
	0=B^2 - 6 A L.
\end{equation*} 
If $L=B=0$, we could get the fiber $IV$. Thus $L \neq 0$, consequently
$B\neq 0$.  The blow-up of the point $(0, b_{-1})$ is
\begin{equation*}
	u=-\frac{2}{(b_-+b_+) \pm \sqrt{8 A L}}, v=0.
\end{equation*}
These points never coincide, hence the pencil has three singular points. 
Moreover, the latter two points lie on the same curve, hence 
the pencil has altogether two singular curves.
Proposition~\ref{prop:E6-E7_sing_fib} leads to one possible case: 
the fibration has a type $II$ and an $I_2$ fibers.

Suppose now that the fibration has a type $II$ and an $I_2$ singular
fibers. Then we have three singular points and three possible cases
which produce these. In the first and second cases the quartic has
two roots, as we discussed above. We saw that if $z_2=0$ is a triple
root, then the fibration has a type $III$ fiber. If a point $z_2 \neq
0$ is a double root, we get a cusp and an $I_2$ fibers.  This happens
exactly when the discriminant vanishes and $L\neq 0$. In the third case
the quartic has three distinct roots, so that the blow-up of the point
$(0, b_{-1})$ must contain one singular point. This condition means
that the two points in (\ref{eq:31_sn_blowup}) coincide. This happens
exactly when $0=2 A L+B^2$, which leads to a triple root and
fiber of type $III$.

\subsubsection{Three roots with section}
The quartic has three distinct roots if and only if the discriminant 
(\ref{eq:31_sn_discr}) does not vanish, and $B^2 \neq -2AL$. 
If $B^2\neq -2AL$, then the two singular points
in the blow-up are distinct. Thus three roots give four singular points. 
Proposition~\ref{prop:E6-E7_sing_fib} offers three possibilities: 
$4 I_1, I_3+I_1, I_2+2I_1$.
The two singular points which came from the blow-up lie on the same fiber. 
Hence the case of four fishtails is not possible.

We will separate the remaining cases by the existence of section.  If
the pencil has a section, then Lemma~\ref{lem:31_section_exist}
provides $L=0$.  The roots of the quadric (\ref{eq:31_sn_quadric}) and
the roots which came from the blow-up are 
\begin{equation*}
	-\frac{B}{A},-\frac{B}{3A}, -\frac{1}{b_+}, -\frac{1}{b_-}.
\end{equation*}
We have $B \neq 0$, since $\Delta \neq 0$ and $L=0$ (i. e. $B=0$ leads to 
type $IV$ fiber), so these roots are distinct. It is enough to show that
the fibration has two singular fibers only. Substitute the 
two roots of the quadric to the expression of~(\ref{eq:31_sn_t}):
\begin{align*}
	t_1:=&-\frac{1}{2} \left(b_-+b_+\right) \left(\lambda _-+\lambda _+\right), \\
	t_2:=&\frac{1}{2} \left(\frac{2 B^3}{27 A}-\left(b_-+b_+\right) \left(\lambda _-+\lambda _+\right)\right).
\end{align*}

Otherwise, as we have seen in Equation~(\ref{eq:31_sn_t_inblowup}),
the $t$ value which appeared in the blow-up was
\begin{equation*}
	t_3:=\left(b_-+b_+\right) b_{-1}. 
\end{equation*}
Condition~(\ref{eq:residuum_31_sn}) gives $-2b_{-1}=\lambda_+ +
\lambda_-$ and this makes $t_1$ and $t_3$ equal.  The result gives a
singular fiber with three singular points. The only possible case is
an $I_3$ fiber.  Finally, due to Proposition~\ref{prop:E6-E7_sing_fib}, 
the other singular fiber is $I_1$.

Assume that there is an $I_3$ and a $I_1$ singular fibers in the fibration. 
$\Delta \neq 0$ is obvious, and  Lemma~\ref{lem:31_sn_section} provides 
that there is a section, and furthermore $L=0$. 

\subsubsection{Three roots without section}
If the pencil has no section then $L\neq 0$. The equations give four
distinct roots as above and at least two singular fibers.
Proposition~\ref{prop:E6-E7_sing_fib} and the process of elimination
ensure that we will get the case $I_2+2 I_1$, because the appearance
of an $I_3$ fiber leads to $L=0$.

Conversely, suppose that the fibration has an $I_2$ fiber 
and two fishtail fibers. Four singular points require 
$\Delta \neq 0$. The $L\neq 0$ condition comes from the fact that the  
fibration has no $I_3$ or $IV$ fiber. Finally, if
$B^2 = -2AL$ then all of this leads to $III+I_1$ case, so 
$B^2 \neq -2AL$.

\begin{proof}[Proof of Proposition~\ref{thm:main_31_sn}]
Once again, the case-analysis above discussed all possible cases, providing the proof of
the claim given in Proposition~\ref{thm:main_31_sn}. 
\end{proof}


\subsection{The discussion of cases appearing in Proposition~\ref{thm:main_31_ns}}
\label{sec:31_ns}
The polar part of the Higgs field is non-semisimple near $q_1$ and 
semisimple near $q_2$. Equations~(\ref{eq:31_nil_z1_eq}) and 
(\ref{eq:31_ss_z2_eq}) give the spectral curve and the pencil. The 
base locus consists of three points: $(0, b_{-6})$ in the chart $(z_1, 
w_{1})$ (on the fiber with multiplicity three) and $(0, \mu_-)$ and 
$(0,\mu_+)$ in the chart $(z_2, w_{2})$.  The fibration has a singular 
fiber of type $\Et7$.

We will need the characteristic polynomial (\ref{eq:31_char-poly2}) in both 
trivializations:
\begin{align}
	\label{eq:31_ns_chi}
	\begin{split}
	\chi_{\vartheta_1}(z_1, w_{1}, t)=& w_{1}^2 - \left(b_{-2} z_1^2+b_{-4} z_1+2 b_{-6}\right) w_{1} 
	+\mu _- \mu _+ z_1^4 - t z_1^3 + \\
	&+ \left(b_{-6} b_{-2}-b_{-3}\right) z_1^2+\left(b_{-6} b_{-4}-b_{-5}\right) z_1+b_{-6}^2,
	\end{split}\\
	\begin{split} \nonumber
	\chi_{\vartheta_2}(z_2, w_{2}, t)=& w_{2}^2 + \left(2 b_{-6} z_2^2+b_{-4} z_2+b_{-2}\right) w_{2} 
	+b_{-6}^2 z_2^4+\left(b_{-6} b_{-4}-b_{-5}\right) z_2^3 + \\
	&+ \left(b_{-6} b_{-2}-b_{-3}\right) z_2^2 -t z_2 + \mu _- \mu _+.
	\end{split}
\end{align}

It is enough to analyze the fibration in $\kappa_2$ trivialization, 
but later we need to use the equation of the pencil in $\kappa_1$ trivialization.

We solve the equations of~(\ref{eq:31_partials}) 
using Condition~(\ref{eq:residuum_31_ns}). We express $w_{2}$ and $t$ 
from the second and the third equations and substitute them into the first.
\begin{align*} 
	w_{2} (z_2)=& -\frac{1}{2} \left( 2 b_{-6} z_2^2+b_{-4} z_2+b_{-2}\right), \\
	t (z_2)=&  -\frac{1}{2} \left(6 b_{-5} z_2^2+\left(b_{-4}^2+4 b_{-3}\right) z_2+b_{-4} b_{-2}\right), \\
	0 =& 8 b_{-5} z_2^3+\left(b_{-4}^2+4 b_{-3}\right) z_2^2-b_{-2}^2+4 \mu _- \mu _+.
\end{align*}
We can rewrite the last expression using the notation in 
(\ref{eq:31_notation}) and again use the residuum 
condition~(\ref{eq:residuum_31_ns}) to get 
\begin{equation} 
	\label{eq:31_ns_cubic}
	0= Q z_2^3+R z_2^2-M^2.
\end{equation}

The roots of this cubic give the $z_2$ values of singular points in the singular fibers. 
Now, the roots are in one-to-one correspondence with singular points, 
because the fiber component of the curve at infinity with multiplicity $1$ has two base points. 
Cubic polynomials generally have three distinct roots, and this 
corresponds to the fact that there are at most three singular fibers 
in the elliptic fibration.

The discriminant of the cubic of~(\ref{eq:31_ns_cubic}) is
\begin{equation*}
	\Delta= M^2 \left(27 M^2 Q^2-4 R^3\right).
\end{equation*}
$M=0$ leads to the non-regular semisimple case (see 
Remark~\ref{rem:31_aamumu}), and according to Assumption~\ref{assn:elliptic} we have $M \neq 0$ in the
following.

One more expression characterizes the number of the roots:
\begin{equation*}
	\Delta_0 = R^2.
\end{equation*}

\subsubsection{One root}
The cubic polynomial~(\ref{eq:31_ns_cubic}) has one root 
if and only if the discriminant $\Delta$ and $\Delta_0$ vanish. 
This leads to $R=Q=0$, furthermore $M=0$, which  is 
contradiction.

\subsubsection{Two roots}
If $Q=0$, the polynomial of~(\ref{eq:31_ns_cubic}) becomes quadratic, 
and hence it has two roots. 

\begin{lem}
\label{lem:31_ns_no_fibration}
The equation $0=R z_2^2-M^2$ does not give an elliptic fibration. 
\end{lem}
\begin{proof}
The $Q=0$ condition is equivalent to $b_{-5}=0$. Consider the tangents
of the curves of the pencil of~(\ref{eq:31_ns_chi}) in $(z_1=0,
w_1=b_{-6})$.  Compute the implicit derivative of $\chi_{\vartheta_1} (z_1,
w_1)$ in the points $(0,b_{-6})$ with the condition $b_{-5}=0$ and get
\begin{equation*}
	\frac{\frac{\partial \chi_{\vartheta_1}}{\partial w_1}}
	{\frac{\partial \chi_{\vartheta_1}}{\partial z_1}} = -\frac{2}{b_{-4}}.
\end{equation*}
Since $b_{-4} \neq \infty$, the curves intersect the $z_1=0$ axis
transversally in $(0,b_{-6})$. This is a singular point and
$(0,b_{-6})$ lies on all curves in the pencil. The pencil has no
smooth curves, hence the resulting fibration is not elliptic. (See
remark before Subsection~\ref{sec:sing_fib}.)
\end{proof}

If $Q\neq 0$ and $\Delta=0$, the cubic polynomial has two distinct roots, 
one double root which gives a cusp fiber and a single root which gives a 
fishtail fiber;
by Proposition~\ref{prop:E6-E7_sing_fib} and Lemma~\ref{lem:31_ns_nextE7} 
there is no other possibility.

Conversely, the existence of a fiber of type $II$ and $I_1$ implies $\Delta=0$.

\subsubsection{Three roots}
Finally, if $\Delta \neq 0$ and $Q\neq 0$ the cubic has three 
distinct roots. Then by 
Proposition~\ref{prop:E6-E7_sing_fib} and Lemma~\ref{lem:31_ns_nextE7} 
all singular fibers are $I_1$. The converse direction is obvious.

\begin{proof}[Proof of Proposition~\ref{thm:main_31_ns}]
The cases enlisted above now prove  
Proposition~\ref{thm:main_31_ns}: to have an elliptic fibration we need $Q\neq 0$,
and the fiber at infinity is $\Et7$, while the further fibers are 
\begin{itemize}
\item if $\Delta \neq 0$, then three $I_1$-fibers, and 
\item if $\Delta =0$, then a type $II$ fiber and an $I_1$ fiber. 
\end{itemize}
\end{proof}


\subsection{The discussion of cases appearing in Proposition~\ref{thm:main_31_nn}}
\label{sec:31_nn}
The polar part of the Higgs field is non-semisimple near both $q_i$ 
($i=1,2$). Equations~(\ref{eq:31_nil_z1_eq}) and 
(\ref{eq:31_nil_z2_eq}) determine the spectral curve and the 
pencil. The base locus of the pencil consists of only two points: 
$(0, b_{-6})$ in the chart $(z_1, w_{1})$ and $(0, b_{-1})$ in the 
chart $(z_2, w_{2})$.  Consequently the fibration has a singular fiber 
of type $\Et7$.

We will need the characteristic polynomial (\ref{eq:31_char-poly2}) in both 
trivialization:
\begin{align}
	\label{eq:31_nn_chi_z1}
	\begin{split}
	\chi_{\vartheta_1}(z_1, w_{1}, t)=& w_{1}^2 + \left(-b_{-2} z_1^2-b_{-4} z_1-2 b_{-6}\right) w_{1}
	+b_{-1}^2 z_1^4-t z_1^3 + \\
	&+\left(b_{-6} b_{-2}-b_{-3}\right) z_1^2	+\left(b_{-6} b_{-4}-b_{-5}\right) z_1+b_{-6}^2,
	\end{split}\\
	\label{eq:31_nn_chi_z2}
	\begin{split} 
	\chi_{\vartheta_2}(z_2, w_{2}, t)=& w_{2}^2 +\left(2 b_{-6} z_2^2+b_{-4} z_2+b_{-2}\right) w_{2}
	+b_{-6}^2 z_2^4+\left(b_{-6} b_{-4}-b_{-5}\right) z_2^3+ \\
	&+ \left(b_{-6} b_{-2}-b_{-3}\right) z_2^2-t z_2+b_{-1}^2.
	\end{split}
\end{align}
It is enough to analyze the fibration in
$\kappa_2$ trivialization.

Again consider the equations of (\ref{eq:31_partials}) to identify 
singular points; once again some singular points can come from the blowing up. 

Express $w_{2}$ and $t$ from the second and the third 
equations by $z_2$:
\begin{align*} 
	w_{2} (z_2)=& -\frac{1}{2} \left( 2 b_{-6} z_2^2+b_{-4} z_2+b_{-2} \right), \\
	t (z_2) =& -\frac{1}{2} \left( 6 b_{-5} z_2^2+\left(b_{-4}^2+4 b_{-3}\right) z_2+b_{-4} b_{-2}\right).
\end{align*}
Substitute all these into the equation of~(\ref{eq:22_partials_first}):
\begin{equation*}
	0 =8 b_{-5} z_2^3+\left(b_{-4}^2+4 b_{-3}\right) z_2^2-b_{-2}^2+4 b_{-1}^2.
\end{equation*}
Rewrite this polynomial with the notation in (\ref{eq:31_notation}) 
and use Condition~(\ref{eq:residuum_31_nn}):
\begin{equation} 
	\label{eq:31_nn_cubic}
	0=Q z_2^3 + R z_2^2.
\end{equation}

This polynomial has at most two roots: $z_2=0$ and $z_2=-\frac{R}{Q}$.
$z_2=0$ is a double root and it lies on the fiber component of the
curve at infinity with multiplicity $1$, thus we will need to blow up
this point, similarly to Subsection~\ref{sec:31_sn} (see
Lemma~\ref{lem:sphere_in_fiber}).

\subsubsection{Blow-up on the chart $(z_2, w_{2})$}
The $z_2=0$ root of the cubic provides the possibility of the
existence of some possible singular points. In the same way as
in Subsection~\ref{sec:31_sn_blowup}, we compute the blow-up of
$\chi_{\vartheta_2}$ in Equation~(\ref{eq:31_nn_chi_z2}). 
(See the lower diagram in the Figure~\ref{fig:blowup31B}.)
The blow-up operation is the same as before: we blow-up the point $(z_2=0,
w_2=b_{-1})$ and then we blow up the new origin.  The exceptional
divisors are also the same as in previous case. Finally, we get a
pencil in the chart $u \neq 0$, which  we denote by $\chi_{b}$.  We simplify
the expression by Condition~(\ref{eq:residuum_31_nn}) and divide by
the exceptional divisor (which is $u^2 v$): 
\begin{equation*}
	\chi_{b}(u,v,t)= b_{-6}^2 u^6 v^3+\left(b_{-6} b_{-4}-b_{-5}\right) u^4 v^2+
	2 b_{-6} u^3 v^2-b_{-3} u^2 v+b_{-4} u v+v+b_{-4} b_{-1}-t.
\end{equation*}

Let us compute the singular points of this family of curves on the chart
$(u, v)$. As before, we search for triples $(u, v, t)$ such that the
point $(u, v)$ lies on the curve with parameter $t$, and the partial
derivatives vanish (as in
Equation~(\ref{eq:31_partials})). The second and third equations do not
contain the variable $t$, thus we solve them for $u$ and $v$,
and we get three roots (three is the maximal number of singular
points in this case).  The $(u,v)$ coordinates of the singular points are
the following:
\begin{equation}
	\label{eq:31_nn_blowup}
	\left(\frac{Q}{b_{-6} R-4 b_{-5} b_{-4}},-\frac{R \left(b_{-6} R-4 b_{-5} b_{-4}\right){}^2}{Q^3}\right),
	\left(-\frac{2}{b_{-4}+\sqrt{R}}, 0\right), \left(\frac{b_{-4}+\sqrt{R}}{2 b_{-3}},0\right).
\end{equation}

The corresponding $t$ values are $b_{-4} b_{-1}$ in second and third cases.
This  means that 
the possible two singular points on the 
$u$ axis lie on the same curve, and 
implies  that the fibration cannot have  three $I_1$ fibers.

\subsubsection{One root}
The equation of~(\ref{eq:31_nn_cubic}) can
 have only one solution (which is $z_2=0$) 
in two different ways.

First we consider the case of $Q=b_{-5}=0$.
As in Lemma~\ref{lem:31_ns_no_fibration}, 
we analyze the spectral curves~(\ref{eq:31_nn_chi_z1}) in the point 
$(z_1=0, w_1=b_{-6})$. 
We take the implicit derivative of $\chi_{\vartheta_1} (z_1, w_{1})$:
\begin{equation*}
	\frac{\frac{\partial \chi_{\vartheta_1}}{\partial w_1}}
	{\frac{\partial \chi_{\vartheta_1}}{\partial z_1}} = -\frac{2}{b_{-4}}.
\end{equation*}
Since $b_{-4} \neq \infty$, the curves intersect the $z_1=0$ axis
transversally in $(0,b_{-6})$.  This is a singular point and it lies
on all curves in the pencil.  The pencil has no smooth curve, hence it
is not an elliptic fibration in $\CP2\# 9 {\overline {\CP{}}}^2$. (See
remark before Subsection~\ref{sec:sing_fib}.)

Now, we consider the $Q\neq 0$ case. The
polynomial~(\ref{eq:31_nn_cubic}) has a triple root $z_2=0$ if and
only if $R=0$. Substitute $R=0$ in the $(u,v)$ coordinates of the 
singular points~(\ref{eq:31_nn_blowup}) on the blow-up and use the notation
$R=b_{-4}^2+4 b_{-3}$. As a result, we get one singular point:
\begin{equation*}
	 (u,v,t)=\left(-\frac{2}{b_{-4}},0,b_{-4} b_{-1}\right).
\end{equation*}
By the classification in Proposition~\ref{prop:E6-E7_sing_fib}, 
the unique singular point belongs to a fiber of type $III$. 

Conversely, if we have a type $III$ fiber, 
then the cubic of~(\ref{eq:31_nn_cubic}) 
has one root. This implies  that $R=0$. 

\subsubsection{Two roots}
The cubic of~(\ref{eq:31_nn_cubic}) has two distinct roots if and only 
if $R \neq 0$ (since $Q \neq 0$ holds). As we have shown above, the 
fibration has at most three singular points in the chart $(u, v)$. The 
singular point which comes from cubic's root $z_2=-\frac{R}{Q}$ 
never coincides with the other singular point which comes from the root 
$z_2=0$. Hence we only analyze the second and third singular 
points from~(\ref{eq:31_nn_blowup}). Solve the following system of 
equations in variables $R$ and $b_{-3}$: 
\begin{align*}
	-\frac{2}{b_{-4}+\sqrt{R}}=&\frac{b_{-4}+\sqrt{R}}{2 b_{-3}}, \\
	R=&b_{-4}^2+4 b_{-3}.
\end{align*}
It has only one solution, where $R=0$. This is a contradiction, 
hence the second and the third singular points do not coincide. 
The fibration has three distinct singular points on two singular fibers 
because the last two singular points lie on the fiber component which came 
from the blow-up (see Lemma~\ref{lem:sphere_in_fiber}). 
According to Proposition~\ref{prop:E6-E7_sing_fib} 
these  are an $I_2$ and an $I_1$ fibers.

In the converse direction, we suppose that the fibration has an $I_2$ 
and an $I_1$ fibers. Then the cubic of~(\ref{eq:31_nn_cubic}) must have two 
distinct roots because the blow-up operation brings up two singular points 
from one root. This implies  $R\neq 0$.

\begin{proof}[Proof of Proposition~\ref{thm:main_31_nn}]
Once again, the cases analyzed above provide a complete proof of 
Proposition~\ref{thm:main_31_nn},
showing that (since the pencil is assumed to have smooth curves) we have $Q\neq 0$ and
the fibers next to the $\Et7$ fiber are 
\begin{itemize}
\item if $R\neq 0$, then an $I_2$  and an $I_1$ fiber, and 
\item if $R=0$, then a type $III$ fiber. 
\end{itemize}
\end{proof}


\section{Sheaves on curves of type $I_1$ and $II$}\label{sec:I1_II}

In Propositions \ref{thm:main_22_ss}, \ref{thm:main_22_sn},
\ref{thm:main_22_nn}, \ref{thm:main_31_ss}, \ref{thm:main_31_sn},
\ref{thm:main_31_ns} and \ref{thm:main_31_nn} we have separated cases
according to the multiplicities of fibers of the Hirzebruch surface
appearing in the pencil and determined the various possibilities for
the remaining singular fibers of the fibration.  From now on, we will
study torsion-free rank-$1$ sheaves on the various fibers of the
fibration.  Therefore, by virtue of Theorem \ref{thm:spectral}, it
will be more convenient to consider the various singular curves that
occur, and study torsion-free rank-$1$ sheaves on each one of them.

In this section, we will study torsion-free rank-$1$ sheaves on curves
of type $I_1$ and $II$.  Although the classifications of torsion-free
sheaves on curves of these types are well-known~\cite{AK, Cook}, we reproduce
them in this section for sake of completeness.  This analysis will
finish the proof of parts \ref{thm:PIII(D6)_II_II},
\ref{thm:PIII(D6)_II_I1}, \ref{thm:PIII(D6)_II_I1valtozat} and
\ref{thm:PIII(D6)_I1} of Theorem \ref{thm:PIII(D6)}, Theorems
\ref{thm:PIII(D7)} and \ref{thm:PIII(D8)}, parts \ref{thm:PIV_II_II},
\ref{thm:PIV_II_I1} and \ref{thm:PIV_I1} of Theorem \ref{thm:PIV} and
Theorem \ref{thm:PII}.  Indeed, Propositions \ref{thm:main_22_ss},
\ref{thm:main_22_sn}, \ref{thm:main_22_nn}, \ref{thm:main_31_ss} and
\ref{thm:main_31_ns} respectively show that in these cases all
singular fibers of the elliptic surface are of type $I_1$ or $II$, so
the proofs of the above parts of the theorems follows from a simple
application of Theorem \ref{thm:spectral} and the results of this
section.

\subsection{Sheaves on curves of type $I_1$}
\begin{lem}\label{lemma:I1}
Assume that $X_t$ is any curve of type $I_1$ and let $\delta \in \Z$
be given.  {(Actually, we will only use that $X_t$ is a fishtail
  curve, not its self-intersection number.)}  Then isomorphism classes
of invertible sheaves of degree $\delta$ on $X_t$ are parameterized by
$\C^{\times}$, and isomorphism classes of non-invertible torsion-free
sheaves of rank $1$ and degree $\delta$ are parameterized by a
point. In particular, in case the moduli space is known to
  be a smooth elliptic surface the corresponding Hitchin fiber is of
  type $I_1$.
\end{lem}

\begin{proof}
 We will be sketchy, as we will see similar ideas in the proofs of Lemmas \ref{lemma:types_of_sheaves} and \ref{lemma:invertible_sheaves}. 

 The first statement follows 
 using the cohomology long exact sequence induced by the short exact sequence of sheaves 
 $$
  0 \to \O_{X_t}^{\times} \to \O_{\tilde{X}_t}^{\times} \to \C^{\times} \to 0, 
 $$
 where $\tilde{X}_t$ stands for the normalization of $X_t$. 
 
 The length of $\O_{\tilde{X}_t}$ as an $\O_{X_t}$-module is 
 $$
  l ( \O_{\tilde{X}_t} ) = 1. 
 $$
 Any torsion-free sheaf of rank $1$ on $X_t$ is either invertible or the direct image of an invertible sheaf on $\tilde{X}_t$. 
 The second statement then follows since $\tilde{X}_t$ is of genus $0$. 
 
 For the third statement, simply notice that by Kodaira's classification, $I_1$ is the only degeneration of an elliptic curve in the class of $\Lef$. 
\end{proof}

\subsection{Sheaves on curves of type $II$}
\begin{lem}\label{lemma:II}
Assume that $X_t$ is any curve of type $II$ {(more precisely, any
  cuspidal rational curve)} and let $\delta \in \Z$ be given. Then
isomorphism classes of invertible sheaves of degree $\delta$ on $X_t$
are parameterized by $\C$, and isomorphism classes of non-invertible
torsion-free sheaves of rank $1$ and degree $\delta$ are parameterized
by a point. In particular, in case the moduli space is known
  to be a smooth elliptic surface the corresponding Hitchin fiber is
  of type $II$.
\end{lem}

\begin{proof}
 Let $\tilde{X}_t$ stand for the normalization of $X_t$. 
 We claim that there is a short exact sequence of sheaves of abelian groups 
 $$
   0 \to \O_{X_t}^{\times} \to \O_{\tilde{X}_t}^{\times} \to \C \to 0.
 $$
 Indeed, in suitable coordinates on affine charts we have 
 $$
  X_t = \Spec ( \C [x,y] / (x^3 - y^2) ), \quad \tilde{X}_t = \Spec \C [t]
 $$
 and the normalization morphism $\tilde{p}$ is induced by 
 $$
  t \mapsto (t^2, t^3). 
 $$
 The cokernel of this morphism is the $\C$-vector space spanned by the monomial $t$. 
 The first statement then follows from the induced cohomology long exact sequence. 
 
 Again, a torsion-free sheaf of rank $1$ on $X_t$ is either invertible or the direct image of an invertible sheaf on $\tilde{X}_t$.  This implies the second statement because $g(\tilde{X}_t) = 0$. 
 
 Finally, by Kodaira's classification, $II$ is the only degeneration of an elliptic curve in the class of $\Lef + \Pt$. 
\end{proof}


\section{Sheaves on curves of type $III$, non-degenerate case}
\label{sec:III}

In this section we will study torsion-free rank-$1$ sheaves on curves
of type $III$ and prove part \eqref{thm:PIII(D6)_III} of Theorem
\ref{thm:PIII(D6)} and part \eqref{thm:PIV_III} of Theorem
\ref{thm:PIV}. 
Along the way, we will prove the assertion of 
Theorem~\ref{thm:wall-crossing} in the cases where the hyperplane in the 
weight space is induced by the existence of two irreducible 
components of a fiber $X_t$ of type $III$.

Indeed, by Propositions \ref{thm:main_22_ss} and
\ref{thm:main_31_ss}, in these cases the elliptic fibration has a
singular fiber $X_t$ of type $III$ (and possibly an $I_1$ fiber).
In particular, in these cases we have parabolic weights $\alpha_i^j$
for $i \in \{ \pm \}$ and $j \in \{ q_1, q_2 \}$. 
According to the proof
of Propositions \ref{thm:main_22_ss} and \ref{thm:main_31_ss}, all
components of the type $III$ curve cover simply the base $\CP1$
(i.e. none of them lie in some fiber of $p$).

We will denote the components of $X_t$ by $X_+$ and $X_-$. 
Up to a permutation of $X_{\pm}$ we see that there are two essentially different cases for the intersection of the exceptional divisors of $X$ and the components of $X_t$:  
\begin{enumerate}
 \item either $X_+$ intersects $E_+^1, E_+^2$ and $X_-$ intersects $E_-^1, E_-^2$  \label{szelesek_egyenesen}
 \item or $X_+$ intersects $E_+^1, E_-^2$ and $X_-$ intersects $E_-^1, E_+^2$. \label{szelesek_keresztben}
\end{enumerate}
Actually, these cases are told apart by the conditions on the parameters: according to the discussion in the proof of Lemma \ref{lem:22_section_exist}, 
\eqref{szelesek_egyenesen} happens if and only if $L = - M \neq 0$ and \eqref{szelesek_keresztben} happens if and only if $L = M \neq 0$. 
(The statement analogous to Lemma \ref{lem:22_section_exist} in the $(3,1)$ case is Lemma \ref{lem:31_section_exist}.)
In the rest of this section we will assume that we have $L = - M \neq 0$; 
in the case $L = M \neq 0$ the same analysis continues to be correct up to exchanging the roles of $\alpha_+^2$ and $\alpha_-^2$.

\subsection{Normalization, partial normalization, length, bidegree}
We begin by introducing some notation. We let $X_t$ be the singular fiber of type $III$. 
In suitable coordinates on an open affine set $U$ containing its singular point, it is given by 
$$
  X_t \cap U = \Spec ( R  ) 
$$
with 
\begin{align*}
  R & = \C [x,y] / I \\
  I & = ((x-y^2) (x+y^2) ). 
\end{align*}
We let $(0,0)$ stand for the only singular point of $X_t$, given by $x = 0 = y$. 
We denote by $\tilde{X}_t$ the normalization of $X_t$. It is easy to see that there exists a 
partial normalization $X'_t$ inbetween $X_t$ and $\tilde{X}_t$. Affine open sets of these 
curves may respectively be written as 
\begin{align*}
 \tilde{X}_t \cap U & = \Spec ( \tilde{R} ) & \tilde{R} = \C [\tilde{x},\tilde{y}] / \tilde{I} \\
 X'_t \cap U & = \Spec ( R' ) & R' =  \C [ x' , y' ] / I', 
\end{align*}
with 
\begin{align*}
 \tilde{I} & = ((\tilde{x} - 1) (\tilde{x} + 1)) \\
 I' & = ((x' - y')(x' + y')).
\end{align*}
We then have natural morphisms 
\begin{equation}\label{eq:normalization}
 \tilde{X}_t  \xrightarrow{\tilde{p}} X'_t  \xrightarrow{p'} X_t 
\end{equation}
induced over Zariski open sets $V$ such that $(0,0)\notin V$ by the identity and over $U$ by 
\begin{align*}
 p'(x) & = x'y' ,  & p'(y) = y'; \\
 \tilde{p} (x') & = \tilde{x}\tilde{y} ,   & \tilde{p}(y') = \tilde{y}. 
\end{align*}

For any torsion-free coherent sheaf $\Ft$ of $\O_{X_t\cap U}$-modules let $\Ft_{(0,0)}$ denote the fiber of $\Ft$ at $(0,0)$. 
\begin{defn}
If $\Ft$ satisfies 
$$
  \O_{X_t,(0,0)} \subseteq \Ft_{(0,0)}
$$
then the \defin{length} of $\Ft$ at $(0,0)$ is defined as 
$$
  l(\Ft) = \dim_{\C}( \Ft_{(0,0)} /\O_{X_t,(0,0)}). 
$$
\end{defn}

\begin{example}
It may be checked that $R' / R$ is the $\C$-vector space with basis $x'$, so 
$$
  l (p'_* \O_{X'_t} ) = 1.
$$
Similarly, the vector space $\tilde{R}/R$ has basis $\tilde{x}, \tilde{x}\tilde{y}$, hence 
$$
  l (( p' \circ  \tilde{p})_*\O_{\tilde{X}_t}) = 2.
$$
\end{example}
 

Recall from \eqref{eq:projection} that we have denoted by $p$ the ruling of the Hirzebruch surface $Z_{\CP1} (D) = {\mathbb {F}}_2$. 
We also have the map 
$$
  \sigma : X \to Z_{\CP1} (D)
$$
obtained by blowing up $8$ infinitely close points, studied in detail in Section \ref{sec:ell_penc}. 
The specific form of the map $\sigma$ depends on the configuration that we fix (i.e., the orders of 
the poles and whether or not the polar parts are semi-simple), but for sake of simplicity we will 
lift this dependence from the notation. We may thus consider the map 
$$
  p \circ \sigma : X \to \CP1. 
$$
Furthermore, we will use the same notation for the restriction of this map to any subscheme $X_t$ of $X$. 
We now assume that 
$$
  \E = (p \circ \sigma)_* (\Ft ), 
$$
with $\Ft$ a torsion-free coherent sheaf of rank $1$ and degree $\delta$ over $X_t$. 
A simple argument using the Riemann--Roch formula then shows that 
\begin{equation}\label{eq:delta-d}
  \delta = d + 2. 
\end{equation}
The curve $X_t$ has a single singular point $(0,0)$ which is a tacnode (an $A_3$-singularity), 
and it has two reduced irreducible components $X_+, X_-$ that are rational curves. 
We let $\mathcal{L}(\Ft)_{\pm}$ denote the line bundle associated to the restriction of 
$\Ft$ to $X_{\pm}$ and we set 
$$
  \delta_{\pm} (\Ft ) = \deg (\mathcal{L}(\Ft)_{\pm} ). 
$$
\begin{defn}
The invariant 
\begin{equation}\label{eq:bidegree}
 (\delta_+ (\Ft ), \delta_- (\Ft ))\in \Z^2
\end{equation}
is called the \defin{bidegree} of $\Ft$. 
\end{defn}

For any coherent sheaf $\Ft'$ on $X'_t$ (or $\tilde{\Ft}$ on
$\tilde{X}_t$) we define its bidegree as the bidegree of the coherent
sheaf $p'_* (\Ft')$ (respectively, $( p' \circ \tilde{p})_*
(\tilde{\Ft})$) on $X_t$.

\subsection{Algebraic description of rank $1$ torsion-free sheaves}

We first give a local description. 

\begin{lem}\label{lemma:types_of_sheaves}
 Any rank-$1$ torsion-free sheaf $\Ft$ of regular modules on $X_t\cap U$ is isomorphic to exactly one of the three following sheaves: 
 \begin{enumerate}
  \item $\O_{X_t\cap U}$ \label{lemma:types_of_sheaves1}
  \item $p'_* ( \O_{X'_t\cap U} )$ \label{lemma:types_of_sheaves2}
  \item $( p' \circ  \tilde{p})_*  ( \O_{\tilde{X}_t\cap U} )$. \label{lemma:types_of_sheaves3}
 \end{enumerate}
\end{lem}

\begin{proof}
 This is a special case of a more general result for arbitrary
 Cohen--Macaulay modules, for details see
 \cite[Proposition~2.2.1]{Cook} and \cite{Greuel-Knoerrer}.  Since the
 scheme $X_t\cap U$ is affine, sheaves of $\O_{X_t\cap U}$-modules
 correspond to modules over $R$.  Let $\Ft$ correspond to the
 $R$-module $M$. Then,
 $$
  \tilde{M} = ( M \otimes_R \tilde{R} ) / \Torsion_1^{\tilde{R}} (\tilde{R} / (\tilde{x}) , M \otimes_R \tilde{R} )
 $$
 is a torsion-free $\tilde{R}$-module of rank $1$. 
 Since $\tilde{R}$ is regular, we see that (possibly up to restricting $U$) $\tilde{M}$ is in fact a free rank-$1$ $\tilde{R}$-module,
 so choosing a generator $m\in M$ results in an $\tilde{R}$-module isomorphism 
 $$
  \tilde{M} \cong \tilde{R}. 
 $$
 Notice that this is also an $R$-module isomorphism. 

 On the other hand, the natural $R$-module morphism $M \to \tilde{M}$ is a monomorphism, for its kernel is torsion and $M$ is by assumption torsion-free. 
 To sum up, we obtain the sequence of $R$-modules 
 $$
  R \subseteq M \subseteq \tilde{M} \cong \tilde{R}, 
 $$
 the first morphism being $r \mapsto rm$. 
 The cases $M = R$ and $M = \tilde{R}$ clearly imply parts \eqref{lemma:types_of_sheaves1} and \eqref{lemma:types_of_sheaves3} respectively. 
 So, assume that 
 $$
  R \subset M \subset \tilde{R}, 
 $$ i.e. that $\dim_{\C} (M / R) = 1$. Up to subtracting an element
    of $R$, the generator $m$ of the $R$-module $M$ is of the form
 $$
  m = a \tilde{x}
 $$
 for some $a \in R \setminus \{ 0 \}$. 
 Now, the only relations of $\tilde{R}$ as an $R$-module are $\tilde{y} = y, y^2 \tilde{x} = x$. 
 The dimension condition then implies that 
 $$
  a = b y 
 $$\
 for some $b \in R^{\times}$. In particular, this implies that 
 $$
  x' = b^{-1} m,
 $$
 generates $M$, i.e. $M = R'$. 
 
 To show that the sheaves
 \eqref{lemma:types_of_sheaves1}--\eqref{lemma:types_of_sheaves3} are
 not isomorphic to each other, simply observe that length is an
 invariant for modules over the local ring, and that the lengths of
 these sheaves are all different.  This finishes the proof.
\end{proof}

We will now describe one by one the moduli of sheaves of the above three types having given bidegree. 

\begin{lem}\label{lemma:invertible_sheaves_normalization}
 For any fixed $(\delta_+, \delta_-) \in \Z^2$, there exists a unique isomorphism class of invertible sheaves $\tilde{\Ft}$ on $\tilde{X}_t$ 
 of bidegree $(\delta_+, \delta_-)$. 
\end{lem}
\begin{proof}
Since 
$$
  \tilde{X}_t = X_+ \coprod X_-
$$
with each of $X_{\pm}$ isomorphic to $\CP1$, the statement follows from the Grothendieck--Birkhoff theorem about line bundles of given degree on $\CP1$. 
\end{proof}

We will denote the unique sheaf of bidegree $(\delta_+, \delta_-)$ provided by Lemma \ref{lemma:invertible_sheaves_normalization} by $\tilde{\Ft}_{(\delta_+, \delta_-)}$. 

\begin{lem}\label{lemma:invertible_sheaves}
 For any fixed $(\delta_+, \delta_-) \in \Z^2$, isomorphism classes of invertible sheaves $\Ft$ on $X_t$ 
 such that 
 $$
  (\delta_+ (\Ft ), \delta_- (\Ft )) = (\delta_+, \delta_-)
 $$
 are parameterized by $\C$. 
\end{lem}
\begin{proof}
We follow 
\cite[Lemma~4.1]{Gothen_Oliveira}. 
Let us use the notation $\O_{\tilde{X}_t, (0,0)}^{\times}$ for the stalk of the sheaf $\O_{\tilde{X}_t}^{\times}$ at $(0,0)$. 
Since $\tilde{X}_t \cap U$ has two connected components, each isomorphic to an affine line, we see that 
$$ \O_{\tilde{X}_t, (0,0)}^{\times} = \{ (f_+ , f_-) \in \C\{ t_+ \}
\oplus \C\{ t_- \} \mid \quad f_+(0) \neq 0 \neq f_-(0) \}.
$$
Since $(0,0)$ is the only singular point, the normalization map $p' \circ  \tilde{p}$ is an isomorphism over affine open sets $V$ 
such that $(0,0) \notin V$. 
Let $\C_{(0,0)}$ denote the sky-scraper sheaf of abelian groups with fiber $\C$ placed at the singular point $(0,0)$, and 
similarly for $\C_{(0,0)}^{\times}$. Then, we have a short exact sequence of sheaves of abelian groups 
$$
  0 \to \O_{X_t}^{\times} \to ( p' \circ  \tilde{p})_*  ( \O_{\tilde{X}_t}^{\times} ) \to \C_{(0,0)}^{\times} \times \C_{(0,0)} \to 0 
$$
induced by 
$$
  (f_+, f_-) \mapsto \left( \frac{f_+(0)}{f_-(0)}, f_+'(0) - f_-'(0) \right) \in \C^{\times} \times \C
$$
on the stalk at $(0,0)$ and by the identity over open sets $V$ such that $(0,0) \notin V$. 
Here $f_i'$ stands for the differential of $f_i$ with respect to $t_i$. 
The associated long exact sequence of cohomology groups reads as 
\begin{align}
 0 \to & H^0 (X_t, \O_{X_t}^{\times}) \to H^0 (X_t, ( p' \circ  \tilde{p})_*  ( \O_{\tilde{X}_t}^{\times} )) \to \C^{\times} \times \C \to \notag \\
  \to & H^1 (X_t, \O_{X_t}^{\times}) \to H^1 (X_t, ( p' \circ  \tilde{p})_*  ( \O_{\tilde{X}_t}^{\times} )) \to 0. \label{eq:LES}
\end{align}
We have 
\begin{align*}
 H^0 (X_t, \O_{X_t}^{\times}) & = \C^{\times} \\
 H^0 (X_t, ( p' \circ  \tilde{p})_*  ( \O_{\tilde{X}_t}^{\times} )) & = \C^{\times} \times \C^{\times},
\end{align*}
and the map between them is the diagonal embedding 
$$
  \epsilon \in \C^{\times} \mapsto (\epsilon, \epsilon ) \in \C^{\times} \times \C^{\times}, 
$$
with cokernel given by 
$$
  (\epsilon_+, \epsilon_-) \mapsto \frac{\epsilon_+}{\epsilon_-} \in \C^{\times}. 
$$
Thus, taking into account that $p' \circ  \tilde{p}$ is finite, the cohomology long exact sequence simplifies into 
$$
  0 \to \C \to H^1 (X_t, \O_{X_t}^{\times}) \to H^1 (\tilde{X}_t, \O_{\tilde{X}_t}^{\times} ) \to 0.
$$
As $H^1 (X_t, \O_{X_t}^{\times})$ parameterizes isomorphism classes of invertible sheaves on $X_t$, the lemma follows 
from Lemma \ref{lemma:invertible_sheaves_normalization}. 
\end{proof}

Given $\lambda \in \C$, we will denote the corresponding sheaf of bidegree $(\delta_+, \delta_-)$ constructed in Lemma 
\ref{lemma:invertible_sheaves} by $\Ft_{(\delta_+, \delta_-)}(\lambda )$, and the family of sheaves of bidegree 
$(\delta_+, \delta_-)$ will be denoted by $\C_{(\delta_+, \delta_-)}$.

\begin{lem}\label{lemma:partially_normalized_sheaves}
For given $(\delta_+, \delta_-) \in \Z^2$, there exists a unique isomorphism class of invertible sheaves $\Ft'$ on $X'_t$ of bidegree $(\delta_+, \delta_-)$. 
\end{lem}
\begin{proof}
 This is similar to Lemma \ref{lemma:invertible_sheaves}, so we only give a sketch. 
 Since $X'_t$ has a unique singular point, which is of type $A_1$, we get the long exact sequence of cohomology groups 
 \begin{align*}
 0 \to & H^0 (X_t', \O_{X_t'}^{\times}) \to H^0 (\tilde{X}_t,  \O_{\tilde{X}_t}^{\times} ) \to \C^{\times} \to \\
  \to & H^1 (X_t', \O_{X_t'}^{\times}) \to H^1 (\tilde{X}_t,  \O_{\tilde{X}_t}^{\times} ) \to 0. 
 \end{align*}
 We infer that the morphism 
 $$
  H^1 (X_t', \O_{X_t'}^{\times}) \to H^1 (\tilde{X}_t,  \O_{\tilde{X}_t}^{\times} )
 $$
 is an isomorphism. We conclude using Lemma \ref{lemma:invertible_sheaves_normalization}. 
\end{proof}

We will denote the unique sheaf of bidegree $(\delta_+, \delta_-)$ provided by Lemma \ref{lemma:partially_normalized_sheaves} by $\Ft'_{(\delta_+, \delta_-)}$.

\subsection{Limits of vector bundles}

In this section we will study limits of the direct images with respect to $p$ of the sheaves introduced in Lemma \ref{lemma:invertible_sheaves}. 
Given $(\delta_+, \delta_-) \in \Z^2$, let $\CP1_{(\delta_+, \delta_-)}$ be the compactification of $\C_{(\delta_+, \delta_-)}$ by a point: 
$$
  \CP1_{(\delta_+, \delta_-)} = \Spec (\C [\lambda ]) \cup \Spec (\C [\mu ]), \quad \lambda^{-1} = \mu. 
$$
We will denote by $C$ the base curve $\CP1$. 

\begin{lem}
  Let $(\delta_+, \delta_-) \in \Z^2$, and assume $\delta_+ > \delta_-$. There exists a relative vector bundle $\E_{(\delta_+, \delta_-)}$ over 
  $$
  C \times \CP1_{(\delta_+, \delta_-)} \to \CP1_{(\delta_+, \delta_-)}
  $$ 
  such that
 \begin{itemize}
  \item for any $\lambda \in \C_{(\delta_+, \delta_-)}$, the restriction of $\E_{(\delta_+, \delta_-)}$ to $C \times \{ \lambda \}$ is isomorphic to 
  $$
    (p \circ \sigma)_* \Ft_{(\delta_+, \delta_-)}(\lambda ), 
  $$ 
  \item the restriction of $\E_{(\delta_+, \delta_-)}$ to $C \times \{ \infty \}$ is isomorphic to $(p \circ \sigma \circ p')_* \Ft'_{(\delta_+ - 1, \delta_-)}$.
 \end{itemize}
\end{lem}
\begin{proof}
 We may interpret the long exact sequence \eqref{eq:LES} as follows. Let $U_+, V_+$ and $U_-, V_-$ be affine coverings of $X_+$
 and $X_-$ respectively, so that $( p' \circ  \tilde{p})^{-1} (0,0)$ consists set-theoretically of one point in each of $U_+ \cap V_+, U_- \cap V_-$. 
 Following the notation of Lemma \ref{lemma:invertible_sheaves}, we denote the corresponding points in both intersections by $t_{\pm} = 0$. 
 For $i \in \{ \pm \}$ let $\varphi_i$ be a \v{C}ech $1$-cocycle associated to a line bundle $\L_i$ of degree $\delta_i$ over $X_i$. 
 The sheaf $\Ft_{(\delta_+, \delta_-)}(\lambda )$ arises by suitably identifying the pull-backs of the line bundles $\L_i$ to $\C [t_i] / (t_i^2)$ with each other. 
 Namely, modifying $(\varphi_+ , \varphi_- )$ by a $\C^{\times}$-valued locally constant \v{C}ech $1$-coboundary, we may arrange that $\varphi_+ (0) = \varphi_- (0) = 1$. 
 By Taylor's formula, on the vector spaces $\C [t_i] / (t_i^2)$ with their respective bases given by $(1, t_i)$ the coordinates of $\varphi_i$ are 
 $$
  \begin{pmatrix}
    1 \\ \varphi'_i (0)
  \end{pmatrix}.
 $$
 Then, in order to define $\Ft_{(\delta_+, \delta_-)}(\lambda )$ we need an isomorphism 
  $$
    \C [t_+] / (t_+^2)  \cong \C [t_-] / (t_-^2) \\
  $$
 so that 
  $$
   \lambda = \varphi'_+ (0) - \varphi'_- (0). 
  $$
 This shows that the matrix of the isomorphism between the two-dimensional vector spaces $\C [t_i] / (t_i^2)$ with respect to their above-mentioned bases must be 
 \begin{equation}\label{eq:change_of_bases}
  \begin{pmatrix}
   1 & 0 \\
   - \lambda & 1
  \end{pmatrix}.
 \end{equation}
 The limit of this matrix as $\lambda \to \infty$ does not make sense immediately.
 However, if we apply the abelian \v{C}ech $1$-coboundary $\mu\in C^1 (U_+, V_+)$ to $\varphi_+$ then \eqref{eq:change_of_bases} transforms into 
 $$
 \begin{pmatrix}
   \mu & 0 \\
   - 1 & \mu
  \end{pmatrix}.
 $$
 The limit
 \begin{equation}\label{eq:limit_change}
  \lim_{\mu \to 0}  
  \begin{pmatrix}
   \mu & 0 \\
   - 1 & \mu
  \end{pmatrix} 
  = 
  \begin{pmatrix}
   0 & 0 \\
   -1 & 0
  \end{pmatrix}
 \end{equation}
 is not invertible, therefore the limit 
 $$
  \lim_{\lambda \to \infty} \Ft_{(\delta_+, \delta_-)}(\lambda )
 $$ does not arise as the identification of any line bundles on
  $X_{\pm}$ along the subschemes given by $\C [t_i] / (t_i^2)$.  Let
  us denote by $M_+$ the $\C [t_+] / (t_+)$-module $t_+ \C [t_+] /
  (t_+)$ and $M_-$ be the $\C [t_-] / (t_-)$-module $\C [t_-] /
  (t_-)$.  Identification of $X_+$ with $X_-$ along the scheme
  morphism $\C [t_+] / (t_+) \to \C [t_-] / (t_-)$ induced by $t_-
  \mapsto t_+$ produces $X_t'$.  Formula \eqref{eq:limit_change}
  induces an identification of $M_+$ with $M_-$ by
 $$
  t_+ \mapsto -1
 $$
 along the above scheme morphism, giving rise to an invertible sheaf on $X_t'$. 
 The module $M_-$ is the stalk at $(t_- = 0) \in X_-$ of the invertible sheaf $\L_- = \O_{X_-} ( \delta_-)$. 
 On the other hand, the module $M_+$ is the stalk at $(t_+ = 0) \in X_+$ of the invertible sheaf 
 $$
  \L_+ \otimes_{\O_{X_+}} \O_{X_+} ( - 1 ) \cong \O_{X_+} ( \delta_+ - 1 ). 
 $$ 
\end{proof}

\subsection{The stability condition}\label{subsec:stability}
By assumption, we have the formula 
\begin{equation}\label{eq:0=pardeg}
   0 = \pardeg (\E) = \deg (\E ) + \alpha_+^1 + \alpha_-^1 + \alpha_+^2 + \alpha_-^2 .
\end{equation}
Set 
\begin{equation}\label{eq:alpha_i}
   \alpha_i = \alpha_i^1 + \alpha_i^2;
\end{equation}
the introduction of these parameters is justified by the fact that stability only depends on these sums of the weights $\alpha_i^j$. 
Because $\alpha_i^1, \alpha_i^2 \in [0,1)$ we see that for $i \in \{ +, - \} $ we have 
\begin{equation}
   \alpha_i \in [0,2).
\end{equation}
On the other hand, \eqref{eq:0=pardeg} shows that 
$$
  \alpha_+ + \alpha_- \in \{ 0, 1, 2, 3 \}. 
$$

Now there exists a short exact sequence of sheaves 
$$
  0 \to \Ft \to \L(\Ft)_+ \oplus \L(\Ft)_- \to \C_{(0,0)}^{2-l(\Ft)}\to 0, 
$$
hence 
$$
  \chi (\Ft) + 2 - l(\Ft) = \chi (\L(\Ft)_+) + \chi (\L(\Ft)_-). 
$$
Applying this to $\Ft = \O_{X_t}$ we get 
$$
  \chi (\O_{X_t}) + 2 = \chi (\O_{X_+}) + \chi (\O_{X_-}).  
$$
Subtracting the second formula from the first we infer 
$$
  \delta - l(\Ft) = \delta_+ + \delta_-.
$$
Using this formula and (\ref{eq:delta-d}) we can rewrite (\ref{eq:0=pardeg}) as 
\begin{equation}\label{eq:0=pardeg-delta}
   0 = \delta_+ + \delta_- + l(\Ft) - 2 + \alpha_+ + \alpha_-.
\end{equation}
Let $\theta$ be a Higgs field on $\E$ with spectral curve $X_t$. The canonical restriction morphisms 
$$
  \Ft \to \mathcal{L}(\Ft)_{i}
$$
for $i \in \{ \pm \}$ give quotient irregular parabolic Higgs bundles $(\E_i, \theta_i)$ of $(\E, \theta)$ of rank $1$ and degree 
$$
  d_{i} = \delta_{i}.
$$ Furthermore, it follows that these are the only non-trivial
  quotient objects of $(\E, \theta)$.  Indeed, the spectral scheme of
  any non-trivial quotient Higgs bundle is a subscheme of dimension
  $1$ of the spectral curve of $(\E, \theta)$, flat over $\CP1$, and
  this latter scheme has only two such non-trivial one-dimensional
  subschemes, that precisely correspond to the above quotient Higgs
  bundles.  The parabolic weights at $q_1$ and $q_2$ associated to
  $\E_i$ are respectively $\alpha_i^1$ and $\alpha_i^2$, so the
  parabolic degree of $\E_{i}$ is
$$
  \pardeg (\E_i) =  \delta_{i} + \alpha_i . 
$$
By definition, parabolic stability of $(\E,\theta)$ is then equivalent to the inequalities 
\begin{equation}\label{eq:stability-III}
  0 < \delta_{i} + \alpha_i
\end{equation}
for $i \in \{ \pm \}$, and semi-stability is equivalent to 
\begin{equation}\label{eq:semi-stability-III}
  0 \leq \delta_{i} + \alpha_i.
\end{equation}
Condition~\eqref{eq:stability-III} and
Equation~\eqref{eq:0=pardeg-delta} immediately imply that there exist
no stable Higgs bundles with spectral sheaf $\Ft$ of length $2$.

\subsection{Hecke transformations}\label{subsec:Hecke}

Using the notations of Section \ref{sec:ell_penc} 
for any $i \in \{ +, - \}$ and $j\in \{ 1, 2 \}$ we 
denote by $P_i^j \in X_i$ the intersection point $X_i\cap E_i^j$.   
It follows from Section \ref{sec:ell_penc} that $P_i^j$ is a smooth point of $X_i$. 
\begin{defn}
Given a rank-$1$ torsion-free sheaf $\Ft$ on $X_t$, its \defin{Hecke transform} corresponding to $i,j$ is 
\begin{equation}\label{eq:Hecke_transform}
 \Hecke{i}{j} (\Ft ) = \Ft \otimes \O_{X_t} (P_i^j).
\end{equation}
\end{defn}
Obviously, Hecke transformations for various choices of $i,j$ commute
with each other.  If $\E = (p \circ \sigma)_*(\Ft )$ then we set
$$
  \Hecke{i}{j} (\E ) = (p \circ \sigma)_* (\Hecke{i}{j} (\Ft )).
$$
The action of $\Hecke{i}{j}$ on the bidegree is clearly 
$$
  (\delta_+, \delta_-) \mapsto (\delta_+ + \delta_{i+}, \delta_-+ \delta_{i-})
$$
with $\delta_{ii'}$ standing for the Kronecker symbol. Furthermore, we set 
$$
  \Hecke{i}{j} (\alpha_{i'}) = \alpha_{i'} - \delta_{ii'} .
$$
\begin{lem}
The parabolic degree of $\Hecke{i}{j} (\E )$ with respect to the
weights $\Hecke{i}{j}(\alpha_i^j)$ is the same as the parabolic degree
of $\E$ with respect to the weights $(\alpha_i^j)$. The same holds for
any quotient (or sub-)object. In particular, $\Hecke{i}{j}$
preserves stability.
\end{lem}
\begin{proof}
 Immediate from the definitions. 
\end{proof}

\begin{remark}
We do not specify the action of $\Hecke{i}{j}$ on the individual
weights $\alpha_{i'}^{j'}$, only on their sums over $j' \in {\{ 1, 2
  \}}$ for $i'$ fixed.  Again, this is justified because stability
only depends on these sums, and in the sequel we will make use of this
freedom of choice to make sure that the individual weights
$\Hecke{i}{j} (\alpha_{i'}^{j'})$ all lie 
in the interval $[0,1)$. In
  particular, the action of $\Hecke{i}{j}$ on the weights is
  independent of $j$, so we may omit the superscript $j$ from the
  notation of the action of Hecke transformations on $(\alpha_i^j)$.
\end{remark}

For $d \in \{ 0,1,2,3 \}$ let us introduce 
\begin{equation}\label{eq:weight_space}
   W_d = \{ \vec{\alpha} \in [0,1)^4 : \quad \alpha_+^1 + \alpha_-^1 + \alpha_+^2 + \alpha_-^2 = d \}. 
\end{equation}
Clearly, we have $W_0 = \{ (0,0,0,0) \}$, and the inverse of $\Hecke{+}{}$ maps this vector to 
$$
  (\varepsilon , 1- \varepsilon ,0,0)
$$
for some $\varepsilon \in ( 0, 1 )$. 

\begin{lem}\label{lemma:Hecke}
\begin{enumerate}
 \item Given $\vec{\alpha} \in W_2$, there exists a Hecke transformation $\Hecke{i}{}$ such that $\Hecke{i}{}\vec{\alpha} \in W_1$. 
 \item Given $\vec{\alpha} \in W_3$, there exists a composition of two Hecke transformations 
$$ 
  \Hecke{-}{} \circ \Hecke{+}{}
$$
such that 
$$
    \Hecke{-}{} \circ \Hecke{+}{} (\vec{\alpha}) \in W_1.
$$
\end{enumerate}
\end{lem}
\begin{remark}
 We do not claim that these procedures are canonical, and we will see that they depend on choices. 
\end{remark}

\begin{proof}
Assume first that we have $\alpha_+ + \alpha_- = 2$. Then at least one of $\alpha_+\geq 1, \alpha_- \geq 1$ holds. 
We may assume that $\alpha_+\geq 1$, the other case being similar. 
Then, there exists $\varepsilon  \in ( 0, 1 )$ such that 
$$
  \alpha_+^1 - \varepsilon \geq 0 , \quad \alpha_+^2 - (1 -\varepsilon) \geq 0; 
$$
indeed, one may pick for instance $\varepsilon = \alpha_+^1$. We may then use the Hecke transformation $\Hecke{+}{}$ with 
the action 
\begin{align}
 \alpha_+^1 & \mapsto \alpha_+^1 - \varepsilon \geq 0 \notag \\
 \alpha_+^2 & \mapsto \alpha_+^2 - (1 -\varepsilon) \label{eq:Hecke_parabolic_weights} \\
 \alpha_-^j & \mapsto \alpha_-^j. \notag
\end{align}

Assume now that we have $\alpha_+ + \alpha_- = 3$. 
Taking into account the inequalities \eqref{eq:alpha_i} we then see that 
$$
  \alpha_+, \alpha_- \in ( 1,2 ) .
$$
Then, there exists $\varepsilon  \in ( 0, 1 )$ such that 
$$
  \alpha_+^1 - \varepsilon > 0 , \quad \alpha_+^2 - (1 -\varepsilon) > 0. 
$$
We let $\Hecke{+}{}$ act on $\vec{\alpha}$ by \eqref{eq:Hecke_parabolic_weights}. 
Then we have $\Hecke{+}{} (\vec{\alpha}) \in W_2$ and 
$$
  \Hecke{+}{} (\alpha_+^1) + \Hecke{+}{} (\alpha_+^2) < 1. 
$$
We may then apply the first statement with $i = -$.  
\end{proof}

\subsection{Degree $-1$, generic weights}\label{subsec:generic_weights}

Lemma~\ref{lemma:Hecke} shows that the other degree conditions can all be reduced to the analysis of the case 
$$
  \alpha_+ + \alpha_- = 1 
$$
by means of suitable Hecke transformations. By assumption, we have 
$$
  \alpha_i \in [ 0, 2 ), 
$$
which in this case implies 
$$
  \alpha_i \in [ 0, 1 ].
$$
\begin{defn}\label{defn:generic_parabolic_weights}
A weight vector $\vec{\alpha} \in W_1$ is called \defin{generic} if 
\begin{equation}\label{eq:generic_parabolic_weights}
   \alpha_i \in ( 0, 1 )
\end{equation}
for both $i \in \{ + , - \}$ and \defin{special} otherwise. 
\end{defn}
In this section, we will determine the stable and semi-stable sheaves on the singular curve $X_t$ in the case of generic weights. 
Clearly, in this case stability is equivalent to semi-stability. 

Assume first that $l(\Ft) = 0$, i.e. $\Ft$ is an invertible sheaf. 
Equation~\eqref{eq:0=pardeg-delta} reads 
$$
  (\delta_+ + \alpha_+) + (\delta_- + \alpha_-) = 2.
$$
Condition~\eqref{eq:semi-stability-III} then implies 
\begin{equation}\label{eq:stable_bidegrees}
 \delta_+ = 0, \delta_- = 1 \quad \mbox{or} \quad \delta_+ = 1, \delta_- = 0.
\end{equation}
Conversely, it is easy to see that under the assumption 
of~\eqref{eq:generic_parabolic_weights} the bidegree
conditions~\eqref{eq:stable_bidegrees} imply~\eqref{eq:stability-III}.
Lemma~\ref{lemma:invertible_sheaves} implies that under
Condition~\eqref{eq:generic_parabolic_weights} stable irregular
parabolic Higgs bundles with $l(\Ft) = 0$ are parameterized by
$$
  \C_{(0,1)} \coprod \C_{(1,0)}.
$$

Let us now come to the study of sheaves with $l(\Ft) = 1$. Condition~\eqref{eq:0=pardeg-delta} then reads as 
$$
  \delta_+ + \delta_- = 0. 
$$
It is easy to see that assuming \eqref{eq:generic_parabolic_weights} the only bidegree condition giving rise to stable sheaves is $\delta_+ = 0 = \delta_-$. 
In addition, by virtue of Lemma \ref{lemma:partially_normalized_sheaves} this bidegree condition gives rise to a unique stable sheaf $\Ft'_{(0,0)}$. 

Finally, in the case $l(\Ft) = 2$ we have already seen that there may exist no stable sheaves $\Ft$.

To sum up, we have found that for generic weight vectors $\vec{\alpha}\in W_1$ we have 
\begin{equation}
\label{eq:moduli_generic}
    [ \Mod_t^{s} (\vec{\alpha}) ] = [ \Mod_t^{ss} (\vec{\alpha}) ] = 2 \Lef + \Pt.
\end{equation}
As the moduli space is a complete elliptic fibration, it then follows
from Kodaira's list that the fiber corresponding to the point $t\in
\CP1$ is of type $III$.  This, combined with Lemma \ref{lemma:I1}
finishes the proof of part \eqref{thm:PIII(D6)_III} of Theorem
\ref{thm:PIII(D6)} and part \eqref{thm:PIV_III} of Theorem
\ref{thm:PIV} in the case of generic weights.

Now, observe that without the restrictions $\alpha_i^j \in
  [0,1 )$ we could have arbitrary values $\alpha_i \in \R$ subject to
    the only condition $\alpha_+ + \alpha_- = 1$.  Said differently,
    the single quantity $\alpha_+\in \R$ determines the relevant
    stability condition. It is clear that under these more relaxed
    conditions semi-stability is equivalent to stability if and only
    if $\alpha_+ \notin \Z$. Again, we call such weights
    \defin{generic}.  It is straightforward to check that if $\alpha_+
    \notin \Z$ then instead of Equation~\eqref{eq:stable_bidegrees}
    stability would be equivalent to
\begin{equation}\label{eq:wall-crossing}
    \delta_+ = - \lceil \alpha_+ \rceil + 1 \quad \mbox{or} \quad \delta_+ = - \lceil \alpha_+ \rceil + 2, 
\end{equation}
with $\delta_- = 1 - \delta_+$. The corresponding Higgs bundles are therefore parameterized by 
\begin{equation}\label{eq:stable_components}
   \C_{(- \lceil \alpha_+ \rceil + 1,\lceil \alpha_+ \rceil)} \coprod \C_{(- \lceil \alpha_+ \rceil + 2,\lceil \alpha_+ \rceil - 1)}.
\end{equation}
This shows the assertion of Theorem~\ref{thm:wall-crossing} in the
particular case of a fiber $X_t$ of type $III$, hence verifies it for
case~\eqref{thm:PIII(D6)_III} of Theorem~\ref{thm:PIII(D6)} and
case~\eqref{thm:PIV_III} of Theorem~\ref{thm:PIV}.

\subsection{Degree $-1$, special weights}\label{subsec:special_weights}

In this section we treat the case $\alpha_+ = 1, \alpha_- = 0$. 
Of course, the same considerations hold with $+$ and $-$ swapped. 

Assume first that $l(\Ft) = 0$. 
In this case, the extensions of invertible sheaves of bidegree $\delta_+ = 1, \delta_- = 0$ are all strictly semi-stable, 
so the component $\C_{(1,0)}$ of \eqref{eq:moduli_generic} is not in the stable moduli space, 
hence the corresponding subset 
of $\Mod_t^s$ is $\C_{(0,1)}$. 
However, the above component is in the semi-stable moduli space, parameterizing strictly semi-stable Higgs bundles 
\begin{equation}\label{eq:strictly_semistable_(1,0)}
  (\E_{(1,0)}(\lambda), \theta_{(1,0)}(\lambda))
\end{equation}
for $\lambda \in \C$. 
Moreover, the bidegree condition $\delta_+ = - 1, \delta_- = 2$ also leads to strictly semi-stable sheaves 
\begin{equation}\label{eq:strictly_semistable_(-1,2)}
  (\E_{(-1,2)}(\mu), \theta_{(-1,2)}(\mu))
\end{equation}
for $\mu \in \C$. 
{Let us set 
\begin{equation}\label{eq:decomposable_Higgs}
   (\tilde{\E}_{(\delta_+, \delta_-)}, \tilde{\theta}_{(\delta_+, \delta_-)}) = (p \circ \sigma \circ p' \circ \tilde{p})_* \tilde{\Ft}_{(\delta_+, \delta_-)}.
\end{equation}
}

\begin{lem}\label{lemma:S-equivalence1}
For any $\lambda, \mu \in \C$, the Higgs bundles \eqref{eq:strictly_semistable_(1,0)} and \eqref{eq:strictly_semistable_(-1,2)} are S-equivalent to 
\[
  (\tilde{\E}_{(-1,0)}, \tilde{\theta}_{(-1,0)}). 
\]
\end{lem}
\begin{proof}
The destabilizing quotient Higgs bundle for the family
\eqref{eq:strictly_semistable_(-1,2)} is of degree $-1$, with spectral
scheme $X_+$.  By additivity of the Euler-characteristic, we see that
the destabilizing Higgs subbundle of this family is then of degree
$0$, with spectral scheme $X_-$.  According to
\eqref{eq:0=pardeg-delta} the associated graded Higgs bundle with
respect to the Jordan--H\"older filtration of
\eqref{eq:strictly_semistable_(-1,2)} corresponds to sheaves $\Ft$ on
$X_t$ with $l(\Ft) = 2$ and bidegree $(-1,0)$.  We deduce from Lemma
\ref{lemma:invertible_sheaves_normalization} that the associated
graded Higgs bundles are isomorphic to ${(p \circ \sigma \circ p' \circ
  \tilde{p})_*} \tilde{\Ft}_{(-1,0)}$.  We get the same conclusion for
the family \eqref{eq:strictly_semistable_(1,0)} along the same lines,
as their destabilizing quotient Higgs bundles are of degree $0$ with
spectral scheme $X_-$.
\end{proof}

We infer from the lemma that in the semi-stable moduli space the subset parameterizing invertible sheaves is 
$$
  \C_{(0,1)} \coprod \{ {(\tilde{\E}_{(-1,0)}, \tilde{\theta}_{(-1,0)})} \} .
$$

Let us now come to the case $l(\Ft) = 1$. 
Here, the bidegree conditions $\delta_+ = 0 = \delta_-$ and $\delta_+ = - 1, \delta_- = 1$ 
give rise to strictly semi-stable sheaves $\Ft'_{(0,0)}$ and $\Ft'_{(-1,1)}$. 
Let us denote by 
$$
  (\E'_{(0,0)}, \theta'_{(0,0)}), (\E'_{(-1,1)}, \theta'_{(-1,1)})
$$
the corresponding Higgs bundles. 
\begin{lem}\label{lemma:S-equivalence2}
The Higgs bundles $(\E'_{(0,0)}, \theta'_{(0,0)})$ and $(\E'_{(-1,1)}, \theta'_{(-1,1)})$ are S-equivalent to 
\[
  (\tilde{\E}_{(-1,0)}, \tilde{\theta}_{(-1,0)}). 
\]
\end{lem}
\begin{proof}
According to Subsection \ref{subsec:stability}, the destabilizing
quotient of $(\E'_{(0,0)}, \theta'_{(0,0)})$ is $(\E_-, \theta_-)$,
with $\E_- \cong \O_{\CP1}$.  By additivity of the degree the
corresponding destabilizing Higgs subbundle $(\F_-, \theta_{\F_-})$
must then be of degree $-1$ on $\CP1$, so we have $\F_- \cong
\O_{\CP1} (-1)$.  The spectral schemes of the Higgs fields $\theta_-$
and $\theta_{\F_-}$ are respectively $X_-$ and $X_+$.  Now, as both
$\E_-$ and $\F_-$ are rank-$1$ bundles, there exists a unique Higgs
field on them with given spectral scheme.  The Jordan--H\"older
filtration of the Higgs bundle $(\E_{(0,0)}, \theta)$ is thus given by
$$
  0 \subset (\F_-, \theta_{\F_-}) \subset (\E'_{(0,0)}, \theta'_{(0,0)}). 
$$

A similar analysis shows that the destabilizing sub- and quotient objects $(\F_+, \theta_{\F_+})$ and $(\E_+, \theta_+)$ of $(\E'_{(-1,1)}, \theta'_{(-1,1)})$ have 
degrees $-1$ and $0$ respectively, and the Higgs fields on them have spectral schemes $X_-$ and $X_+$ respectively. 
It follows from unicity of such rank-$1$ Higgs bundles that 
\begin{align*}
 (\F_-, \theta_{\F_-}) & = (\E_+, \theta_+) \\
 (\F_+, \theta_{\F_+}) & = (\E_-, \theta_-).
\end{align*}
Therefore, the graded Higgs bundles of the Jordan--H\"older filtrations of $(\E'_{(0,0)}, \theta'_{(0,0)})$ and $(\E'_{(-1,1)}, \theta'_{(-1,1)})$ agree with ${(p \circ \sigma \circ p' \circ \tilde{p})_*} \tilde{\Ft}_{(-1,0)}$. 
\end{proof}

To sum up, we have found that if $\alpha_+ = 1, \alpha_- = 0$ then 
\begin{align*}
   [\Mod_t^{s} (\vec{\alpha}) ] & = \Lef \\ 
   [\Mod_t^{ss} (\vec{\alpha}) ] & = \Lef + \Pt.
\end{align*}
This, combined with Lemma \ref{lemma:I1} finishes the proof of part
\eqref{thm:PIII(D6)_III} of Theorem \ref{thm:PIII(D6)} and part
\eqref{thm:PIV_III} of Theorem \ref{thm:PIV} in the case of special
weights.

\subsection{Degree $-2$}

We now briefly sketch how Lemma \ref{lemma:Hecke} allows us to reduce the case 
$$
  \alpha_{{+}} + \alpha_{{-}} = 2, \quad 0 < \alpha_{{+}}, \alpha_{{-}} <2. 
$$
to the study of Subsections \ref{subsec:generic_weights} and \ref{subsec:special_weights}. 
Let us distinguish between three cases according to the values of the parabolic weights. 
\subsubsection{Case $\alpha_+ = 1 = \alpha_-$.} 
This case gets reduced by either $\Hecke{+}{}$ or $\Hecke{-}{}$ to the case of degree $-1$ with special weights treated in Subsection \ref{subsec:special_weights}. 
More precisely, $\Hecke{+}{}$ reduces it to 
$$
  \alpha_+ = 0, \quad  \alpha_- = 1
$$
while $\Hecke{-}{}$ reduces it to 
$$
  \alpha_+ = 1, \quad  \alpha_- = 0.
$$

\subsubsection{Case $\alpha_+ < 1 < \alpha_-$.} 
This case gets reduced by $\Hecke{-}{}$ to the case of degree $-1$
with generic weights treated in
Subsection~\ref{subsec:generic_weights}.  Observe that we may not
apply $\Hecke{+}{}$.

\subsubsection{Case $\alpha_- < 1 < \alpha_+$.} 
This is analogous to the previous case with $+$ and $-$ swapped. 

To sum up, there exists a wall 
$$
  \{ \alpha_+ = \alpha_- \} \subset W_2,
$$
on which $\Mod_t^{ss}(\vec{\alpha})$ is isomorphic to the moduli space of degree $-1$ Higgs bundles with special weights, and on the two sides of the wall 
$\Mod_t^{ss}(\vec{\alpha})$ is isomorphic to the moduli space of degree $-1$ Higgs bundles with generic weights with different bidegree conditions.

\subsection{Degree $-3$}
According to Lemma \ref{lemma:Hecke}, a composition of two Hecke transformations shows that the moduli space $\Mod_t^{ss}(\vec{\alpha})$ 
is isomorphic to the moduli space of degree $-1$ Higgs bundles with generic weights.

\subsection{Degree $0$}
A Hecke transformation shows that the moduli space $\Mod_t^{ss}(\vec{\alpha})$ is isomorphic to the moduli space of degree $-1$ Higgs bundles with special weights.

\section{Sheaves on curves of type $I_2$, non-degenerate case}
\label{sec:I2}

In this section we will study torsion-free rank-$1$ sheaves on curves of type $I_2$, and prove parts~\eqref{thm:PIII(D6)_I2_I2} and~\eqref{thm:PIII(D6)_I2_I1} of Theorem~\ref{thm:PIII(D6)} 
and parts~\eqref{thm:PIV_I2_II} and~\eqref{thm:PIV_I2_I1} of Theorem~\ref{thm:PIV}. 
{Along the way, we will prove the assertion of Theorem~\ref{thm:wall-crossing} in the cases of a fiber $X_t$ of type $I_2$. }
Indeed, by Propositions~\ref{thm:main_22_ss} and~\ref{thm:main_31_ss}, in these cases the elliptic fibration has a singular fiber $X_t$ of type $I_2$ (and other singular fibers that have already been studied). 
It follows from the proof of Propositions~\ref{thm:main_22_ss} and~\ref{thm:main_31_ss} that no component of $X_t$ lies in any fiber of $p$. 
Again, we have parabolic weights $\alpha_i^j$ for $i \in \{ \pm \}$ and $j \in \{ q_1, q_2 \}$. 
Much of the analysis follows the one carried out in \cite{ISS} and in Section~\ref{sec:III}, so we will content ourselves with sketching the 
differences in the proofs as compared to Section~\ref{sec:III}, and mainly focus on the motivic wall-crossing phenomenon that has not yet been described explicitely. 

Again, let us denote the components of $X_t$ by $X_+$ and $X_-$.
{Just as in the beginning of Section \ref{sec:III}, we have two
  combinatorial possibilities for the intersections of $X_{\pm}$ with
  the exceptional divisors $E_i^j$.  Again, Lemma
  \ref{lem:22_section_exist} and its counterpart Lemma
  \ref{lem:31_section_exist} apply and give the same conditions for
  these possibilities in terms of the parameters as in Section
  \ref{sec:III}. As in Section \ref{sec:III}, here we will content
  ourselves with analyzing case \eqref{szelesek_egyenesen}, noting
  that the arguments of this section remain correct in case
  \eqref{szelesek_keresztben} as well, up to exchanging $\alpha^2_+$
  and $\alpha^2_-$.  However, in the current setup we need to point
  out the existence of a third case with $I_2$ fibers, not covered by
  the previous conditions.  Namely, for $L = M = 0, \Delta \neq 0$
  (part \eqref{thm:PIII(D6)_I2_I2} of Theorem \ref{thm:PIII(D6)}) $X$
  admits two $I_2$ fibers: one of them (let us temporarily call it
  $X_{b_1}$) will intersect the exceptional divisors along the scheme
  of point \eqref{szelesek_egyenesen} and the other one (say
  $X_{b_2}$) as in point \eqref{szelesek_keresztben} from the
  beginning of Section \ref{sec:III}. In this special case, we need to
  apply the below analysis for the fiber $X_{b_1}$, and redo the same
  analysis with $\alpha^2_+$ and $\alpha^2_-$ exchanged for
  $X_{b_2}$. In particular, in the degree $-2$ case the set of special
  weights for these values of the parameters is the union of the two
  walls defined by the equations
\begin{align}
 \alpha^1_+ + \alpha^2_+ & = 1 \label{eq:hyperplane_1} \\
 \alpha^1_+ + \alpha^2_- & = 1 \label{eq:hyperplane_2};
\end{align}
crossing \eqref{eq:hyperplane_1} only alters the Hitchin fiber over $b_1$, whereas crossing \eqref{eq:hyperplane_2} only alters the Hitchin fiber over $b_2$, 
in the way described later in this section. 
}

We denote by {$x_1,x_2$ the singular points of $X_t$ and by}
$$
  \tilde{p}: \tilde{X}_t \to X_t
$$
the normalization of $X_t$. We assume that 
$$
  \E = (p \circ \sigma)_* (\Ft ) 
$$ for some torsion-free sheaf $\Ft$ of $\O_{X_t}$-modules of rank
  $1$.  We consider the line bundles $\L_i$ induced by $\Ft$ on $X_i$,
  and we denote by $\delta_i$ their degrees.  We use the notation
  \eqref{eq:alpha_i}, and by assumption we have \eqref{eq:0=pardeg}.
  The local classification of rank $1$ torsion-free sheaves $\Ft$ on
  $X_t$ is simple. For $x \in \{ x_1,x_2 \}$, let $\Ft_{{x}}$
  denote the stalk of $\Ft$ at ${x}$.
\begin{lem}\label{lemma:types_of_sheavesI2}
For any rank-$1$ torsion-free sheaf $\Ft$ over $X_t$ and ${x \in \{ x_1, x_2 \}}$
\begin{enumerate}
 \item either $\Ft_{{x}}$ is a rank-$1$ locally free $\O_{X_t,{x}}$-module, 
 \item or $\Ft_{{x}} = \tilde{p}_* \tilde{\Ft}_{{x}}$ for some rank-$1$ locally free $\O_{\tilde{X}_{b,{x}}}$-module $\tilde{\Ft}_{{x}}$. 
\end{enumerate} 
\end{lem}
\begin{proof}
 Similar to Lemma \ref{lemma:types_of_sheaves}. 
\end{proof}

The notion of bidegree is defined exactly as in \ref{eq:bidegree}. 
The global classification starts similarly as in Section \ref{sec:III}. 
\begin{lem}\label{lemma:invertible_sheaves_normalizationI2}
 For any fixed $(\delta_+, \delta_-) \in \Z^2$, there exists a unique isomorphism class of invertible sheaves $\tilde{\Ft}_{(\delta_+, \delta_-)}$ on $\tilde{X}_t$ 
 of bidegree $(\delta_+, \delta_-)$. 
\end{lem}
\begin{proof}
 Similar to Lemma \ref{lemma:invertible_sheaves_normalization}. 
\end{proof}
However, there is a small difference concerning invertible sheaves. 
\begin{lem}\label{lemma:invertible_sheaves_I2}
 For any fixed $(\delta_+, \delta_-) \in \Z^2$, isomorphism classes of invertible sheaves $\Ft$ on $X_t$ 
 such that 
 $$
  (\delta_+ (\Ft ), \delta_- (\Ft )) = (\delta_+, \delta_-)
 $$
 are parameterized by $\C^{\times}$. 
\end{lem}
\begin{proof}
 It follows from the cohomology long exact sequence associated to the
 short exact sequence of sheaves
 $$
    0 \to \O_{X_t, {x}}^{\times} \to  \tilde{p}_*  ( \O_{\tilde{X}_t}^{\times} )_{{x}} \to \C^{\times} \to 0 .
 $$
\end{proof}
This family of sheaves will be denoted by $\Ft_{(\delta_+, \delta_-)}(\lambda )$ with $\lambda \in \C^{\times}$. 
Following \cite{Oda-Seshadri}, we let 
$$
  J (\Ft) = \{ {x \in \{ x_1, x_2 \}} \mid \quad \Ft_{{x}} \mbox{ is locally an invertible sheaf} \}. 
$$
The invariant $2 - |J(\Ft )|$ plays much the same role as $l(\Ft )$ in Section \ref{sec:III}, as they both measure the defect of $\Ft$ from being an invertible sheaf. 
In particular, we have 
$$
  d = \delta_+ + \delta_- - 2 + |J(\Ft )|,
$$
and consequently \eqref{eq:0=pardeg} reads as 
\begin{equation}\label{eq:pardeg=0-delta_I2}
   0 = (\delta_+ + \alpha_+ ) + (\delta_- + \alpha_- ) - |J(\Ft )| .
\end{equation}
For any $i \in \{ \pm \}$ restriction to the quotient sheaf 
$$
  \Ft \to \Ft_{X_i} \to \L_i
$$
gives rise to a quotient Higgs bundle $(\Qt_i, \t|_{\Qt_i} )$ of parabolic degree 
$$
  \delta_i + \alpha_i . 
$$ Just as in Section \ref{sec:III}, these are the only non-trivial
  quotient Higgs bundles of $(\E, \theta)$, so $(\E, \theta)$ is
  $\vec{\alpha}$-semi-stable if and only if the two inequalities
\begin{equation}\label{eq:parabolic_stability_I2}
   \delta_i + \alpha_i \geq 0 
\end{equation}
for $i \in \{ \pm \}$ hold, and the condition for stability is given by the corresponding strict inequalities. 
Now, a straightforward modification of Lemma \ref{lemma:Hecke} shows that Hecke transformations allow us to reduce the other degree conditions to the 
case $d=-1$, i.e. 
$$
  \alpha_+ + \alpha_- = 1 .
$$
We will distinguish the cases 
$$
  \alpha_i \in (0,1)
$$
(case of generic weights) and 
$$
  \alpha_i \in \{ 0,1 \} 
$$
(special weights).

\subsection{Degree $-1$, generic weights}\label{subsec:generic_weights_I2}
In this case, semi-stability is equivalent to stability. If $| J (\Ft) | = 2$ then stability implies either 
$$
  \delta_+ =  1, \delta_- = 0 
$$
or 
$$
  \delta_+ = 0, \delta_- = 1 , 
$$
so by virtue of Lemma \ref{lemma:invertible_sheaves_I2} stable invertible sheaves are parameterized by 
$$
    \C_{(0,1)}^{\times} \coprod \C_{(1,0)}^{\times}.
$$
If $J (\Ft) = 1$, then $\Ft_{{x}}$ is locally free for exactly one of ${x \in \{ x_1, x_2 \}}$. 
It is easy to see that stability is equivalent to 
$$
  \delta_+ = 0 = \delta_-, 
$$
so such stable sheaves are parameterized by a point for each of ${x \in \{ x_1, x_2 \}}$. 
We infer that 
\begin{align*}
      [ \Mod_t^{s} (\vec{\alpha}) ] & = [ \Mod_t^{ss} (\vec{\alpha}) ] \\ 
      & = 2 [ \C^{\times} ] + 2 \cdot \Pt \\
      & = 2 \Lef. 
\end{align*}
As the moduli space is a smooth elliptic surface, it follows from Kodaira's list that the singular Hitchin fiber is of type $I_2$. 
This combined with Lemmas \ref{lemma:I1} and \ref{lemma:II} finishes the proof of parts 
\eqref{thm:PIII(D6)_I2_I2} and \eqref{thm:PIII(D6)_I2_I1} of Theorem \ref{thm:PIII(D6)} and 
parts \eqref{thm:PIV_I2_II} and \eqref{thm:PIV_I2_I1} of Theorem \ref{thm:PIV} in the case of generic $\vec{\alpha} \in W_1$.

Just as in Subsection~\ref{subsec:generic_weights}, dropping the assumptions $\alpha_i^j \in [0,1)$ 
any value $\alpha_+ \in \R$ determines $\alpha_-$ by $\alpha_+ + \alpha_- = 1$, and it is a generic weight (in the sense 
that $\vec{\alpha}$-semistability is equivalent to $\vec{\alpha}$-stability) if and only if $\alpha_+ \notin \Z$. 
The stable bidegrees are again given by \eqref{eq:wall-crossing}, while the corresponding Higgs bundles are now parameterized by 
\begin{equation}\label{eq:stable_components_I2}
  \C_{(- \lceil \alpha_+ \rceil + 1,\lceil \alpha_+ \rceil)}^{\times} \coprod \C_{(- \lceil \alpha_+ \rceil + 2,\lceil \alpha_+ \rceil - 1)}^{\times}. 
\end{equation}
This shows the assertion of Theorem~\ref{thm:wall-crossing} in the particular case of a fiber $X_t$ of type $I_2$, and provides the proof for 
the 
 cases~\eqref{thm:PIII(D6)_I2_I2}
    and~\eqref{thm:PIII(D6)_I2_I1} of Theorem~\ref{thm:PIII(D6)},
   and of  cases~\eqref{thm:PIV_I2_II}
    and~\eqref{thm:PIV_I2_I1} of Theorem~\ref{thm:PIV}.

\subsection{Degree $-1$, special weights}\label{subsec:special_weights_I2}
We need to distinguish between stability and semi-stability. For sake of concreteness, we assume $\alpha_+ = 1, \alpha_- = 0$. 
If $|J (\Ft)| = 2$, then stability implies 
$$
    \delta_+ = 0, \delta_- = 1 ;
$$
on the other hand, semi-stability allows for bidegrees 
$$
  (-1, 2), (0,1), (1,0). 
$$
Let us denote by 
\begin{equation}\label{eq:strictly_semistable_(1,0)I2}
  (\E_{(1,0)}(\lambda), \theta_{(1,0)}(\lambda))
\end{equation}
the family of strictly semi-stable Higgs bundles corresponding to the family $\Ft_{(1,0)}(\lambda )$.  
Moreover, the family of sheaves $\Ft_{(-1,2)}(\lambda )$ also induces strictly semi-stable Higgs bundles denoted 
\begin{equation}\label{eq:strictly_semistable_(-1,2)I2}
  (\E_{(-1,2)}(\mu), \theta_{(-1,2)}(\mu))
\end{equation}
for $\mu \in \C^{\times}$. 
Finally, let us denote by  
\begin{equation}\label{eq:decomposable_Higgs_I2}
   (\tilde{\E}_{(\delta_+, \delta_-)}, \tilde{\theta}_{(\delta_+, \delta_-)}) 
\end{equation}
the Higgs bundle associated to $\tilde{p}_* \tilde{\Ft}_{(\delta_+, \delta_-)}$. 
\begin{lem}
For any $\lambda, \mu \in \C^{\times}$ the Higgs bundles
\eqref{eq:strictly_semistable_(1,0)I2} and
\eqref{eq:strictly_semistable_(-1,2)I2} are S-equivalent to $
(\tilde{\E}_{(-1,0)}, \tilde{\theta}_{(-1,0)})$.
\end{lem}
\begin{proof}
 Similar to Lemma \ref{lemma:S-equivalence1}. 
\end{proof}

If $| J (\Ft) | = 1$, then again $\Ft_{{x}}$ is locally free for
exactly one of ${x \in \{ x_1, x_2 \}}$, and semi-stability is
equivalent to
$$
  \delta_+ = 0 = \delta_-. 
$$
This produces two further strictly semi-stable Higgs bundles 
\begin{equation}\label{eq:E00j}
   (\E_{(0,0)}^{{x}}, \theta_{(0,0)}^{{x}})
\end{equation}
with $J (\Ft) = \{ {{x}} \}$. 
\begin{lem}
For both $x \in \{ x_1, x_2 \}$ the Higgs bundles \eqref{eq:E00j}
are S-equivalent to $ (\tilde{\E}_{(-1,0)}, \tilde{\theta}_{(-1,0)})$.
\end{lem}
\begin{proof}
 Similar to Lemma \ref{lemma:S-equivalence2}.
\end{proof}

From these lemmas we deduce that 
\begin{align*}
      [ \Mod_t^{s} (\vec{\alpha}) ] & =  \Lef - \Pt \\
      [ \Mod_t^{ss} (\vec{\alpha}) ] & =  \Lef. 
\end{align*}
This combined with Lemmas \ref{lemma:I1} and \ref{lemma:II} finishes the proof of parts \eqref{thm:PIII(D6)_I2_I2} and \eqref{thm:PIII(D6)_I2_I1} of Theorem \ref{thm:PIII(D6)} and 
parts \eqref{thm:PIV_I2_II} and \eqref{thm:PIV_I2_I1} of Theorem \ref{thm:PIV} in the case of special weights.

\section{Sheaves on singular fibers in the degenerate cases}
\label{sec:I2_I3_III_IV}

In this section we will prove Theorems \ref{thm:PIV_degenerate} and
\ref{thm:PII_degenerate}: we assume that the Higgs field is required
to have a first order pole with non-semisimple residue at the marked
point $q_2$, and study sheaves on the singular fibers of the fibration
described in Propositions \ref{thm:main_31_sn} and
\ref{thm:main_31_nn} that give rise to Higgs fields of the required
local form.  By virtue of Propositions \ref{thm:main_31_sn} and
\ref{thm:main_31_nn}, in this situation independently of the values of
the parameters $A,B,L,M,Q,R$ the elliptic fibration has a singular
fiber $X_t$ of type $I_2, I_3, III$ or $IV$.  In the above singular
cases (as well as in the smooth case) one needs a thorough
understanding of the relationship between torsion-free rank-$1$
sheaves on the singular curves in the pencil and singularity behaviour
of Higgs bundles; this will be the content of Lemma
\ref{lem:adjoint_orbit}.  Indeed, due to the possible degeneration of
nilpotent endomorphisms into semi-simple ones, a compactification
phenomenon appears that has not yet been observed previously.  Roughly
speaking, Lemma \ref{lem:adjoint_orbit} tells us that in order to find
the Higgs bundles with non-semisimple residue we need to extract the
invertible sheaves from the families found in Sections
\ref{sec:I1_II}, \ref{sec:III} and \ref{sec:I2}.  The torsion-free
sheaves on the spectral curve that are not invertible give rise to
Higgs bundles with diagonal residue at $q_2$.
Again, Hecke transformations allow one to reduce any degree condition to degree $-1$, so we will only be interested in this most general case. 

One needs to consider sheaves $\Ft$ on ${\mathbb {F}}_2$ whose support
is a ramified double cover of $\CP1$ by the map $p \circ \sigma$.  Now, it follows
from the results of Lemma~\ref{lem:sphere_in_fiber} that in the cases
considered here one of the components of $X_t$ is the exceptional
divisor $E$ of the first blow-up 
at the base point $P$ given by
\[
  P = (z_2 = 0, w_2 = b_{-1})
\]
corresponding to the non-semisimple residue at the point $q_2$, see the dashed 
curve in Figure~\ref{fig:blowup31B}.  In
  particular, the exceptional divisor $E_1$ then belongs to the fiber
  $(p \circ \sigma)^{-1} (q_2)$. Remember that we denote by 
  $$
    {Z}_t = \omega (X_t)
  $$ 
  the singular curve in the pencil whose proper transform is $X_t$:
\begin{itemize}
 \item if $X_t$ is of type $I_2$ then ${Z}_t$ is a nodal rational curve with a single node on the fiber of multiplicity $1$; 
 \item if $X_t$ is of type $I_3$ then ${Z}_t$ is composed of two rational curves intersecting each other transversely in two distinct points, one of them being on the fiber of multiplicity $1$; 
 \item if $X_t$ is of type $III$ then ${Z}_t$ is a cuspidal rational curve with a single cusp on the fiber of multiplicity $1$; 
 \item if $X_t$ is of type $IV$ then ${Z}_t$ is composed of two
   rational curves tangent to each other to order $2$ on the fiber of
   multiplicity $1$.
\end{itemize}
Let $(\E, \theta)$ be an irregular Higgs bundle such that 
$$
  {Z}_t = (\det (\zeta - p^* \theta  )). 
$$
Let us denote by $\Ft$ the spectral sheaf of $(\E, \theta)$: 
\begin{equation}\label{eq:coker}
  0 \to p^* \E \otimes K^{\vee} ( - 3 \{ q_1 \} - \{ q_2 \}) \xrightarrow{\zeta - p^*\theta } p^* \E \to \Ft  \to 0 .
\end{equation}
By assumption, $\Ft$ is supported on ${Z}_t$. 
We will use the notations of Subsection \ref{subsec:31_local_forms}. In particular, near $q_2$ we have $\zeta = w_2 \kappa_2$ and $\theta = \vartheta_2 \kappa_2$.

\begin{lem}\label{lem:adjoint_orbit}
\begin{enumerate}
 \item If $\Res_{q_2} \theta$ is in the adjoint
   orbit~\eqref{eq:31_nil1} and ${Z}_t$ is smooth then ${Z}_t$ is
   ramified over $q_2$.
 \item For any curve ${Z}_t$ in the corresponding pencil (smooth or
   singular), the endomorphism $\Res_{q_2} \theta$ is in the adjoint
   orbit~\eqref{eq:31_nil1} if and only if $\Ft$ is a locally free
   sheaf on $Z_t$ near $P$.
\end{enumerate}
\end{lem}

\begin{proof}
 Up to a transformation $\zeta \mapsto \zeta + b_1$ we may assume $b_1 = 0$. 
 Let us denote by $\mathfrak{m} = (z_2)$ the maximal ideal of $\CP1$ at $q_2$. 

 For the first statement, with respect to the trivialization $\kappa_2$ we write 
 $$
  \vartheta_2 = \begin{pmatrix}
            a(z_2) & b(z_2) \\
            c(z_2) & d(z_2)
           \end{pmatrix}.
 $$
 Up to a suitable constant change of basis, our assumption on the adjoint orbit means that 
 \begin{equation}\label{eq:b=1,a=c=d=0}
  b(0) = 1, \quad a(0) = c(0) = d(0) = 0.   
 \end{equation}
 We then have 
 $$
 {f_2(z_2)} = - (a(z_2) + d(z_2)) , \quad {g_2(z_2)} = (a(z_2) d(z_2) - b(z_2) c(z_2)) . 
 $$
 Because of \eqref{eq:b=1,a=c=d=0} we have $f_2 (0) = 0$ and 
 \begin{equation}\label{eq:congruence}
    a(z_2) d(z_2) - b(z_2) c(z_2) \equiv - c(z_2) \pmod{\mathfrak{m}^2}. 
 \end{equation}
 It follows that the characteristic polynomial $\chi_{\vartheta_2}(z_2)$
 is an Eisenstein polynomial if and only if 
 $$
  c(z_2) \not\equiv 0 \pmod{\mathfrak{m}^2}.
 $$
 On the other hand, we have 
 \begin{align*}
  \frac{\partial \chi_{\vartheta_2}}{\partial {w_2}} & = 2 {w_2} - (a(z_2) + d(z_2)) \\
  \frac{\partial \chi_{\vartheta_2}}{\partial z_2} & = - \frac{ \mbox{d} (a(z_2) + d(z_2))}{\mbox{d} z_2} {w_2} + 
      \frac{ \mbox{d} (a(z_2) d(z_2) - b(z_2) c(z_2))}{\mbox{d} z_2}.
 \end{align*}
 Plugging $z_2 = 0 = {w_2}$ into these expressions and using \eqref{eq:b=1,a=c=d=0} we get  
 \begin{align*}
  \frac{\partial \chi_{\vartheta_2}}{\partial {w_2}} (0,0) & = 0 \\
  \frac{\partial \chi_{\vartheta_2}}{\partial z_2} (0,0) & = \frac{ \mbox{d} (a(z_2) d(z_2) - b(z_2) c(z_2))}{\mbox{d} z_2}.
 \end{align*}
 Now, because of \eqref{eq:congruence} the curve defined by $\chi_{\vartheta_2}$ is singular at $(0,0)$ if and only if 
 $$
  c(z_2) \equiv 0 \pmod{\mathfrak{m}^2}.
 $$
 To sum up, if the curve defined by $\chi_{\vartheta_2}$ is smooth then $\chi_{\vartheta_2}$ is an Eisenstein polynomial, hence ramified over $0$. 
 As for the second statement, tensoring the identity \eqref{eq:coker} with 
 \begin{equation}\label{eq:local_ring}
    p^* ( \O_{\CP1} / \mathfrak{m}) = \O_{{\mathbb {F}}_2} / p^* \mathfrak{m} 
 \end{equation}
 over $\O_{{\mathbb {F}}_2}$ yields 
 \begin{align*}
  \cdots \to \Tor_1^{\O_{{\mathbb {F}}_2}} (\Ft, \O_{{\mathbb {F}}_2} / p^* \mathfrak{m}) & \to \\ 
  \to p^* (\E \otimes K^{\vee} ( - 3 \{ q_1 \} - \{ q_2 \} ) / \mathfrak{m}) \xrightarrow{\zeta - p^* \theta  \pmod{\mathfrak{m}} } p^* ( \E  / \mathfrak{m}) \to \Ft \otimes \O_{{\mathbb {F}}_2}/ p^* \mathfrak{m} & \to 0 ,
 \end{align*}
 where $\pmod{\mathfrak{m}}$ stands for the morphism induced on the reduction modulo $\mathfrak{m}$. 
 By the symmetry of the $\Tor$ functor we have 
 $$
  \Tor_1^{\O_{{\mathbb {F}}_2}} (\Ft, \O_{{\mathbb {F}}_2} / p^* \mathfrak{m}) = \Tor_1^{\O_{{\mathbb {F}}_2}} (\O_{{\mathbb {F}}_2} / p^* \mathfrak{m} , \Ft) = 0 
 $$
 because $\Ft$ is a torsion-free $\O_{{\mathbb {F}}_2}$-module. We infer that 
 $$
  \Ft\otimes \O_{{\mathbb {F}}_2}/ p^* \mathfrak{m} = \coker ( \zeta - p^* \theta  \pmod{\mathfrak{m}} ). 
 $$
 Locally near $z_2 = 0$ the ring $\O_{{\mathbb {F}}_2}$ is given by $\C[z_2, w_2]$ and \eqref{eq:local_ring} is given by $\C [w_2]$. 
 In different terms, we have the equality of $\C [w_2]$-modules 
 \begin{equation}\label{eq:spectral_sheaf_coker}
    \Ft\otimes_{\C[z_2, w_2]} \C [w_2] =  \coker ( w_2 - \vartheta_2 (0) ).
 \end{equation}
 This (and the assumption $b_1 = 0$) implies that the fiber of $\Ft$ vanishes over $(0,w_2)$ for any $w_2 \neq 0$. 
 On the other hand, $\vartheta_2 (0)$ can be identified with the residue 
of 
 $$
  \theta \colon  \E \otimes K^{\vee} ( - \{ q_2 \} ) \to \E  
 $$ at $q_2$.  Reducing \eqref{eq:spectral_sheaf_coker} modulo $w_2$
  we find that the cokernel of $\Res_{q_2} (\theta )$ is of dimension
 $1$ if and only if the fiber of $\Ft$ at $(0,0)$ is of dimension $1$,
 i.e. if and only if $\Ft$ is locally free near $(0,0)$.
\end{proof}

Recall the definition of generic and special weights from Definition \ref{defn:generic_parabolic_weights}. 
We will use the same notion of genericity, up to the convention \eqref{eq:par_wt_nil} in the case of a Higgs field with non-semisimple polar part or non-semisimple residue at $q_j$. 

\begin{lem}
 The Hitchin fibers of the moduli spaces $\Mod_t^{s}, \Mod_t^{ss}$
 over $t$ are given by
 \begin{enumerate}
  \item if $X_t$ is of type $I_2$ then $[ \Mod_t^{s} (\vec{\alpha}) ] = [ \Mod_t^{ss} (\vec{\alpha}) ] = \Lef - \Pt$,
  \item if $X_t$ is of type $I_3$ then 
    \begin{enumerate}
     \item for generic weights we have $[ \Mod_t^{s} (\vec{\alpha}) ] = 2 \Lef -  \Pt$,  
     \item for special weights we have  
     \begin{align*}
      [\Mod_t^{s} (\vec{\alpha}) ] & =  \Lef \\
      [\Mod_t^{ss} (\vec{\alpha}) ] & =  \Lef + \Pt .
     \end{align*} 
    \end{enumerate}
  \item if $X_t$ is of type $III$ then $[ \Mod_t^{s} (\vec{\alpha}) ] = \Lef$,
  \item if $X_t$ is of type $IV$ then 
  \begin{enumerate}
     \item for generic weights we have $[ \Mod_t^{s} (\vec{\alpha}) ] = 2 \Lef$,  
     \item for special weights we have 
     \begin{align*}
      [\Mod_t^{s} (\vec{\alpha}) ] & = \Lef \\
      [\Mod_t^{ss} (\vec{\alpha}) ] & = \Lef + \Pt.
     \end{align*} 
    \end{enumerate}
 \end{enumerate}
\end{lem}

\begin{proof}
\begin{enumerate}
 \item If $X_t$ is of type $I_2$ then $Z_t$ is a nodal rational curve with a single node on the fiber of multiplicity $1$. 
 As $Z_t$ has a unique component, stability is automatic, thus S-equivalence is the same relation as isomorphism. 
 The degree of $\Ft$ must satisfy \eqref{eq:delta-d}, and Lemma \ref{lem:adjoint_orbit} implies that $\Ft$ must be a locally free sheaf on $Z_t$. 
 The result follows from Lemma \ref{lemma:I1}. 
 \item 
 If $X_t$ is of type $I_3$ then $Z_t$ is composed of two rational
 curves intersecting each other transversely in two distinct points,
 one of them being on the fiber of multiplicity $1$.  By Lemma
 \ref{lem:adjoint_orbit}, $\Ft$ must be locally free on the fiber over $q_2$,
  i.e. $J(\Ft) \supseteq \{ q_2 \}$. 
 Equation~\eqref{eq:pardeg=0-delta_I2} is valid, and the analysis
 closely follows the one detailed in Section \ref{sec:I2},
 hence we content ourselves with sketching it.  In the case
   of generic weights and everywhere locally free sheaves $J(\Ft) = \{ q_1,  q_2 \}$, 
   there exist two bidegree conditions~\eqref{eq:wall-crossing} compatible 
   with stability, giving rise to a family of sheaves parameterized by two copies of
   $\C^{\times}$, specifically~\eqref{eq:stable_components_I2}. 
   In the case of generic weights and $J(\Ft) = \{  q_2 \}$, we get a 
   further S-equivalence class of sheaves. 
  This argument finishes the proof of Theorem~\ref{thm:wall-crossing} in the
 case corresponding to~\eqref{thm:PIV_degenerate_I2} of
 Theorem~\ref{thm:PIV_degenerate}.  If the weights are special, then
 stability is equivalent to a unique value of the bidegree, giving
 rise to a family of locally free sheaves parameterized by $\C^{\times}$, 
 plus one further sheaf with $J(\Ft) = \{  q_2 \}$.  On the
 other hand, semi-stability is equivalent to any one of three bidegree
 conditions: two extremal ones in addition to the central one given by
 stable sheaves.  The semi-stable sheaves with one of the two extremal
 bidegrees are all S-equivalent to each other, so the family of
 semi-stable Higgs bundles coming from locally free sheaeves is parameterized 
 by $\C$. Again, there is one further sheaf with $J(\Ft) = \{  q_2 \}$. 
 \item 
 If $X_t$ is of type $III$ then $Z_t$ is a cuspidal rational curve with a single cusp on the fiber of multiplicity $1$. 
 Stability is again automatic, $\delta$ satisfies \eqref{eq:delta-d}, and $\Ft$ must be a locally free sheaf on $Z_t$. 
 We conclude using Lemma \ref{lemma:II}. 
 \item If $X_t$ is of type $IV$ then $Z_t$ is composed of two
   rational curves tangent to each other to order $2$ on the fiber of
   multiplicity $1$. This analysis closely follows the one
     carried out in Subsection~\ref{subsec:generic_weights}, hence we
     only give a sketch. For generic weights, stability is equivalent
     to~\eqref{eq:wall-crossing} and
     Lemma~\ref{lemma:invertible_sheaves} shows that stable Higgs
     bundles are parameterized by \eqref{eq:stable_components}.  This
   verifies Theorem~\ref{thm:wall-crossing} in
   case~\eqref{thm:PIV_degenerate_III} of
   Theorem~\ref{thm:PIV_degenerate}.  In the case of special weights,
   the $l(\Ft ) = 0$ part of the analysis of Subsection
   \ref{subsec:special_weights} can be repeated verbatim. Namely,
   there exists a unique bidegree condition compatible with stability,
   while there are three bidegree conditions compatible with
   semi-stability. Thus, the stable Higgs bundles are parameterized by
   $\C$.  The semi-stable ones having one of the two extremal
   bidegrees are all S-equivalent to each other (and actually, to all
   semi-stable ones with $l(\Ft ) \in \{ 1, 2 \}$), so they give rise
   to a point in the corresponding Hitchin fiber.
\end{enumerate}
\end{proof}


The lemma combined with Lemmas~\ref{lemma:I1} and~\ref{lemma:II}
finishes the proof of Theorems~\ref{thm:PIV_degenerate},~\ref{thm:PII_degenerate} and~\ref{thm:wall-crossing}.

\bibliography{2szing}

\begin{thebibliography}{10}

\bibitem{AK}
A.~Altman and S.~Kleiman.
\newblock The presentation functor and the compactified {J}acobian.
\newblock In {\em The Grothendieck Festschrift}, volume~86 of {\em Progress in
  Mathematics}, pages 15--32. Birkhauser, 1990.

\bibitem{Biq-Boa}
O.~Biquard and Ph. Boalch.
\newblock Wild non-abelian {H}odge theory on curves.
\newblock {\em Compos. Math.}, 140(1):179--204, 2004.

\bibitem{Cook}
P.~Cook.
\newblock {\em {Local and Global aspects of the Module Theory of Singular
  Curves}}.
\newblock PhD thesis, University of Liverpool, 1993.

\bibitem{GMN}
D.\;Gaiotto, G.\;Moore, and A.\;Neitzke.
\newblock Wall-crossing, {H}itchin systems, and the {WKB} approximation.
\newblock {\em Adv. Math.}, 234:239--403, 2013.

\bibitem{Donagi_Pantev}
R.~Donagi and T.~Pantev.
\newblock Langlands duality for {H}itchin systems.
\newblock {\em Invent. Math.}, 189(3):653--735, 2012.

\bibitem{Gothen_Oliveira}
P.~Gothen and A.~Oliveira.
\newblock The singular fiber of the {H}itchin map.
\newblock {\em Int. Math. Res. Not. IMRN}, 5:1079--1121, 2013.

\bibitem{Greuel-Knoerrer}
G.-M. Greuel and H.~Kn\"orrer.
\newblock Einfache {K}urvensingularit\"aten und torsionsfreie {M}oduln.
\newblock {\em Math. Ann.}, 270(3):417--425, 1985.

\bibitem{HarerKasKirby}
J.~Harer, A.~Kas, and R.~Kirby.
\newblock Handlebody decompositions of complex surfaces.
\newblock {\em Mem. Amer. Math. Soc.}, 62(350):iv+102, 1986.

\bibitem{Hausel}
T.~Hausel, A.~Mellit, and D.~Pei.
\newblock Mirror symmetry with branes by equivariant verlinde formulae.
\newblock arXiv:1712.04408.

\bibitem{Hausel_Thaddeus}
T.~Hausel and M.~Thaddeus.
\newblock Mirror symmetry, {L}anglands duality, and the {H}itchin system.
\newblock {\em Invent. Math.}, 153(1):197--229, 2003.

\bibitem{Hit}
N.~Hitchin.
\newblock Stable bundles and integrable systems.
\newblock {\em Duke Math. J.}, 54(1):91--114, 1987.

\bibitem{Hit_mirror}
N.~Hitchin.
\newblock Higgs bundles and characteristic classes.
\newblock In {\em Arbeitstagung Bonn, In memory of Friedrich Hirzebruch},
  Progress in Mathematics, pages 247--264. Birkhauser, 2016.

\bibitem{ISS_I0*}
P.~Ivanics, A.~Andras~Stipsicz, and Sz. Szab\'o.
\newblock Hitchin fibrations on two-dimensional moduli spaces of irregular
  higgs bundles with one singular fiber.
\newblock arXiv:1808.10125.

\bibitem{ISS}
P.~Ivanics, A.~Stipsicz, and Sz. Szab\'o.
\newblock Two-dimensional moduli spaces of rank 2 {H}iggs bundles over
  {$\mathbb{C}P^1$} with one irregular singular point.
\newblock {\em J. Geom. Phys.}, 130:184--212, 2018.

\bibitem{Kodaira}
K.\;Kodaira.
\newblock On compact analytic surfaces: {II}.
\newblock {\em Ann. Math.}, 77:563--626, 1963.

\bibitem{Miranda}
R.~Miranda.
\newblock Persson's list of singular fibers for a rational elliptic surface.
\newblock {\em Math. Z.}, 205(2):191--211, 1990.

\bibitem{Mocsizuki}
T.~Mochizuki.
\newblock {\em Wild harmonic bundles and wild pure twistor {D}-modules}.
\newblock Number 340 in Ast\'erisque. {S}oci\'et\'e {M}ath\'ematique de
  {F}rance, 2011.

\bibitem{Oda-Seshadri}
T.~Oda and C.~Seshadri.
\newblock Compactifications of the generalized {J}acobian variety.
\newblock {\em Transactions of the {AMS}}, 253, 1979.

\bibitem{Persson}
U.~Persson.
\newblock Configurations of {K}odaira fibers on rational elliptic surfaces.
\newblock {\em Math. Z.}, 205(1):1--47, 1990.

\bibitem{Seshadri}
C.~Seshadri.
\newblock Moduli of vector bundles on curves with parabolic structures.
\newblock {\em Bull. AMS}, 83(1), 1976.

\bibitem{SSS}
A.~Stipsicz, Z.~Szab\'o, and \'A. Szil\'ard.
\newblock Singular fibers in elliptic fibrations on the rational elliptic
  surface.
\newblock {\em Periodica Mathematica Hungarica}, 54:137--162, 2007.

\bibitem{SYZ}
A.~Strominger, S.-T. Yau, and E.~Zaslow.
\newblock Mirror symmetry is {$T$}-duality.
\newblock {\em Nuclear Phys. B}, 479(1-2):243--259, 1996.

\bibitem{Sz_PW}
Sz. Szab\'o.
\newblock Perversity equals weight for painlevé systems.
\newblock arXiv:1802.03798.

\bibitem{Sz-spectral}
Sz. Szab\'o.
\newblock The birational geometry of unramified irregular {H}iggs bundles on
  curves.
\newblock {\em Intern. J. Math.}, 28(6), 2017.

\end{thebibliography}
\bibliographystyle{plain}

\end{document}